\documentclass[preprint,3p]{elsarticle}

\usepackage{lineno,hyperref}
\usepackage{amsfonts}
\usepackage{amsmath}
\usepackage{amssymb}
\usepackage{amsthm}
\usepackage{mathrsfs}
\usepackage{subcaption}
\usepackage{float}
\usepackage{xfrac}
\usepackage[ruled,vlined,linesnumbered]{algorithm2e}
\usepackage{multicol}

\modulolinenumbers[5]
\hypersetup{
  bookmarksnumbered = true,
  bookmarksopen=false,
  pdfborder=0 0 0,         
  pdffitwindow=true,      
  pdfnewwindow=true, 
  colorlinks=true,           
  linkcolor=blue,            
  citecolor=magenta,    
  filecolor=magenta,     
  urlcolor=cyan              
}

\usepackage[dvipsnames]{xcolor}
\usepackage{listings}
\definecolor{codegreen}{rgb}{0,0.6,0}
\definecolor{codegray}{rgb}{0.5,0.5,0.5}
\definecolor{codepurple}{rgb}{0.58,0,0.82}
\definecolor{backcolor}{rgb}{0.95,0.95,0.92}

\lstdefinestyle{mystyle}{
backgroundcolor=\color{backcolor},   
commentstyle=\color{codegreen},
keywordstyle=\color{magenta},
numberstyle=\tiny\color{codegray},
stringstyle=\color{codepurple},
basicstyle=\ttfamily\footnotesize,
breakatwhitespace=false,         
breaklines=true,                 
captionpos=b,                    
keepspaces=true,                 
numbers=left,                    
numbersep=5pt,                  
showspaces=false,                
showstringspaces=false,
showtabs=false,                  
tabsize=2,
otherkeywords={defined,elif,\# }
}
\newtheorem{define}{Definition}
\newtheorem{lemma}{Lemma}
\newtheorem{prop}{Proposition}
\newtheorem{rem}{Remark}
\newtheorem{theorem}{Theorem}

\newtheorem{corollary}{Corollary}
\newtheorem{assumption}{Assumption}

\newcommand{\ME}{\mbox{\tiny \sc ME}}
\newcommand{\iipt}{\hspace{2pt}}

\newcommand{\vipt}{\hspace{6pt}}
\newcommand{\lo}{\mbox{\tiny\sc L}}
\newcommand{\hi}{\mbox{\tiny\sc H}}
\newcommand{\LaxF}{\mbox{\tiny\sc LF}}

\usepackage{definitions}

\newcommand{\Mone}{\bcM^{(1)}}
\newcommand{\Mtwo}{\bcM^{(2)}}
\newcommand{\Done}{\cD^{(1)}}
\newcommand{\Dtwo}{\cD^{(2)}}
\newcommand{\Ione}{\cI^{(1)}}
\newcommand{\Itwo}{\cI^{(2)}}
\newcommand{\Kone}{\mathsf{k}^{(1)}}
\newcommand{\Ktwo}{\mathsf{k}^{(2)}}
\newcommand{\vy}{(y^{k} v_{k})}
\newcommand{\ny}{(y^{k} \hat{n}_{k})}

\newcommand{\old}{{(\ast)}}
\newcommand{\quand}{{\quad\text{and}\quad}}

\DeclareMathAlphabet\mathbc{OMS}{cmsy}{b}{n}
 
\DeclareMathOperator*{\argmin}{arg\,min}

\newcommand{\thornado}{{\scshape thornado}}
\newcommand{\flashx}{{\scshape Flash-X}}

\begin{document}

\begin{frontmatter}

\title{DG-IMEX Method for a Two-Moment Model for \\ Radiation Transport in the $\mathcal{O}(v/c)$ Limit \tnoteref{support}\tnoteref{copyright}}
\tnotetext[support]{
Research at Oak Ridge National Laboratory is supported under contract DE-AC05-00OR22725 from the U.S. Department of Energy to UT-Battelle, LLC.
This research was supported by the Exascale Computing Project (17-SC-20-SC), a collaborative effort of the U.S. Department of Energy Office of Science and the National Nuclear Security Administration.
This work was supported by the U.S. Department of Energy, Office of Science, Office of Advanced Scientific Computing Research via the Scientific Discovery through Advanced Computing (SciDAC) program.
This research used resources of the Oak Ridge Leadership Computing Facility at the Oak Ridge National Laboratory, which is supported by the Office of Science of the U.S. Department of Energy under Contract No. DE-AC05-00OR22725.
This research was supported, in part, by the National Science Foundation’s Gravitational Physics Program under grants NSF PHY 1806692 and 2110177.}
\tnotetext[copyright]{
This manuscript has been authored by UT-Battelle, LLC under Contract No. DE-AC05-00OR22725 with the U.S. Department of Energy. The United States Government retains and the publisher, by accepting the article for publication, acknowledges that the United States Government retains a non-exclusive, paid-up, irrevocable, world-wide license to publish or reproduce the published form of this manuscript, or allow others to do so, for United States Government purposes. The Department of Energy will provide public access to these results of federally sponsored research in accordance with the DOE Public Access Plan(http://energy.gov/downloads/doe-public-access-plan).}


\author[ornl]{M. Paul Laiu}
\ead{laiump@ornl.gov}

\author[ornl,utk-phys]{Eirik Endeve\corref{cor}}
\ead{endevee@ornl.gov}

\author[ornlnccs]{J. Austin Harris}
\ead{harrisja@ornl.gov}

\author[utk-phys]{Zachary Elledge}
\ead{zelledge@vols.utk.edu}

\author[utk-phys]{Anthony Mezzacappa}
\ead{mezz@utk.edu}

\cortext[cor]{Corresponding author. Tel.:+1 865 576 6349; fax:+1 865 241 0381}

\address[ornl]{Multiscale Methods and Dynamics Group, Oak Ridge National Laboratory, Oak Ridge, TN 37831 USA }

\address[utk-phys]{Department of Physics and Astronomy, University of Tennessee Knoxville, TN 37996-1200}

\address[ornlnccs]{Advanced Computing for Nuclear, Particles, and Astrophysics Group, Oak Ridge National Laboratory, Oak Ridge, TN 37831 USA }

\begin{abstract}
We consider particle systems described by moments of a phase-space density and propose a realizability-preserving numerical method to evolve a spectral two-moment model for particles interacting with a background fluid moving with nonrelativistic velocities. The system of nonlinear moment equations, with special relativistic corrections to $\mathcal{O}(v/c)$, expresses a balance between phase-space advection and collisions and includes velocity-dependent terms that account for spatial advection, Doppler shift, and angular aberration. This model is closely related to the one promoted by Lowrie et al. (2001; JQSRT, 69, 291-304) and similar to models currently used to study transport phenomena in large-scale simulations of astrophysical environments. The method is designed to preserve moment realizability, which guarantees that the moments correspond to a nonnegative phase-space density. The realizability-preserving scheme consists of the following key components: (i) a strong stability-preserving implicit-explicit (IMEX) time-integration method; (ii) a discontinuous Galerkin (DG) phase-space discretization with carefully constructed numerical fluxes; (iii) a realizability-preserving implicit collision update; and (iv) a realizability-enforcing limiter. In time integration, nonlinearity of the moment model necessitates solution of nonlinear equations, which we formulate as fixed-point problems and solve with tailored iterative solvers that preserve moment realizability with guaranteed convergence. We also analyze the simultaneous Eulerian-frame number and energy conservation properties of the semi-discrete DG scheme and propose an "energy limiter" that promotes Eulerian-frame energy conservation. Through numerical experiments, we demonstrate the accuracy and robustness of this DG-IMEX method and investigate its Eulerian-frame energy conservation properties.
\end{abstract}


\end{frontmatter}

\section{Introduction}
\label{sec:intro}

In this paper, we design and analyze a numerical method for solving a system of moment equations that model transport of neutral particles (e.g., photons, neutrons, or neutrinos) interacting with a background fluid moving with nonrelativistic velocities --- i.e., flows in which the ratio of the background flow velocity to the speed of light, $v/c$, is sufficiently small such that special relativistic corrections of order $(v/c)^{2}$ and higher can be neglected.  
Similar $\cO(v/c)$ models have been used to study transport phenomena in astrophysical environments \cite{mihalasMihalas_1999}, including neutrino transport in core-collapse supernovae (e.g., \cite{ramppJanka_2002,just_etal_2015,skinner_etal_2019,bruenn_etal_2020,mezzacappa_etal_2020}) and binary neutron star mergers (e.g., \cite{just_etal_2015b,foucart_2023}).  
The numerical method is based on the discontinuous Galerkin (DG) phase-space discretization and an implicit-explicit (IMEX) method for time integration, and we pay particular attention to the preservation of certain physical bounds by the fully discrete scheme.  
The bound-preserving property is achieved by carefully considering the phase-space and temporal discretizations, as well as the formulation of associated iterative nonlinear solvers.  

Neutral particle transport in physical systems where the particle mean-free path may be similar to, or exceed, other characteristic length scales demands a kinetic description based on the distribution function $f(\vect{p},\vect{x},t)$, which is a phase-space density providing, at time $t$, the number of particles in an infinitesimal phase-space volume $d\vect{x}d\vect{p}$ centered around phase-space coordinates $\{\vect{p},\vect{x}\}$.  
Here, $\vect{p}$ and $\vect{x}$ are momentum- and position-space coordinates, respectively.  
The evolution of $f$ is governed by a kinetic equation that expresses a balance between phase-space advection and collisions (e.g., interparticle collisions and/or collisions with a background); see, e.g., \cite{chapmanCowling_1970,mihalasMihalas_1999} for detailed expositions.  
In this paper, as a simplification, we consider the situation where particles described by a kinetic distribution function interact with an external background whose properties are prescribed and unaffected by $f$.  

The design of numerical methods to model transport of particles interacting with a moving fluid is complicated, in part, by the necessity to choose coordinates for discretization of momentum space.  
While relativistic kinetic theory provides the framework to freely specify momentum-space coordinates, the two most obvious reference frame choices, the Eulerian and comoving frames, come with distinct computational challenges (e.g., \cite{castor_1972,buchler_1979,buchler_1983,mihalasMihalas_1999}).  
On the one hand, choosing momentum-space coordinates associated with an Eulerian observer eases the discretization of the phase-space advection problem at the expense of complicating the particle-fluid interaction kinematics and, for moment models, the closure procedure.  
On the other hand, choosing momentum-space coordinates associated with the comoving frame (or comoving observer) --- defined as the sequence of inertial frames whose velocity instantaneously coincides with the fluid velocity \cite{buchler_1983,mihalasMihalas_1999} --- simplifies the description of particle-fluid interaction kinematics but at the expense of increased complexity in solving the phase-space advection problem numerically.  
Moreover, when particles equilibrate with the fluid, the distribution function becomes isotropic in the comoving frame, which simplifies the closure procedure for moment-based methods \cite{buchler_1983}.  
We also mention the mixed-frame approach (e.g., \cite{mihalasKlein_1982}), where the distribution function depends on Eulerian-frame momentum coordinates.  
Then, to evaluate comoving-frame emissivities and opacities at Eulerian-frame momentum coordinates, appropriate transformation laws and expansions to $\cO(v/c)$ are applied (see Section~7.2 in \cite{mihalasMihalas_1999}).  
The mixed-frame approach attempts to combine the best of both coordinate choices but has difficulties with certain collision operators and does not generalize to the relativistic case.  
Nagakura et al. \cite{nagakura_etal_2014} combine both coordinate choices in a relativistic framework, using a discrete ordinates method, which requires mapping of numerical data between momentum space coordinate systems.  
This approach has yet to be applied to moment models.  

Our primary goal is to model neutrino transport in large-scale core-collapse supernova simulations, which require the inclusion of a wide range of neutrino--matter interactions --- with various kinematic forms (e.g., \cite{burrows_etal_2006,janka_etal_2007,janka_2012,mezzacappa_etal_2020,fischer_etal_2023}) --- which tend to dominate the overall computational cost.  
Therefore, we opt for relative simplicity in the collision term, adopt momentum-space coordinates associated with the comoving frame, and focus our effort here on the discretization of the phase-space advection problem.  

Because of the high computational cost associated with solving kinetic equations numerically in full dimensionality with sufficient phase-space resolution, dimension-reduction techniques are frequently employed.  
One commonly used method is to define and solve for a sequence of moments, instead of $f$ directly.  
Specifically, we employ spherical-polar momentum-space coordinates $(\varepsilon,\vartheta,\varphi)$ and integrate the distribution function against angular basis functions (depending on momentum-space angles $\omega=(\vartheta,\varphi)$) to obtain spectral, angular moments (depending on particle energy $\varepsilon$, and $\vect{x}$ and $t$) representing number densities, number fluxes, etc.  
The hierarchy of moment equations is obtained by taking corresponding moments of the kinetic equation.  
In this study, we consider a so-called two-moment model, where we solve for the zeroth (scalar) and first (vector) moments.  
The resulting system of moment equations, accurate to $\cO(v/c)$, describes the evolution of the moments due to advection in phase-space (the left-hand side) and collisions with the background fluid (the right-hand side).  
Due to the choice of comoving-frame momentum coordinates, the left-hand side contains velocity-dependent terms that account for spatial advection, Doppler shift, and angular aberration.  
Moreover, the moment equations contain higher-order moments (rank-two and rank-three tensors) that must be expressed in terms of the lower-order moments to close the system of equations.  
Specifically, we consider an approximate, algebraic moment closure originating from the maximum-entropy closure proposed by Minerbo \cite{minerbo_1978} (see also \cite{cernohorskyBludman_1994,just_etal_2015}).  
Related two-moment models have recently been used to model neutrino transport in core-collapse supernova simulations (e.g., \cite{just_etal_2015,skinner_etal_2019}).  

In this paper, we consider a number-conservative two-moment model obtained by taking the flat spacetime, $\cO(v/c)$ limit of general-relativistic moment models, e.g., from \cite{shibata_etal_2011,cardall_etal_2013a,mezzacappa_etal_2020}.  
We refer to the model as number-conservative because, in the absence of collisions, the zeroth moment equation is conservative for the correct $\cO(v/c)$ Eulerian-frame number density.  
The model is closely related to the two-moment model promoted by Lowrie et al. \cite{lowrie_etal_2001}: With the assumption of one-dimensional, planar geometry, we obtain their equations by multiplying our equations with the particle energy $\varepsilon$.  
This two-moment model supports wave speeds that are bounded by the speed of light.  
It is also consistent, to $\cO(v/c)$, with conservation laws for Eulerian-frame energy and momentum.  
\emph{Key to this consistency is retention of certain $\cO(v/c)$ terms in the time derivative of the moment equations}, which are often omitted (e.g., \cite{vaytet_etal_2011,just_etal_2015,skinner_etal_2019}).  
However, retention of these terms increases the computational complexity of the algorithm because the evolved moments become nonlinear functions of the primitive (comoving-frame) moments needed to evaluate closure relations, which then introduces nonlinear, iterative solves that contribute to increased computational costs.  

We use the DG method \cite{cockburnShu_2001} to discretize the moment equations.  
The choice of comoving-frame momentum coordinates results in advection-type terms along the energy dimension and four-dimensional divergence operators in the left-hand side of the moment equations.  
We use the DG method to discretize all four phase-space dimensions.  
DG methods have advantages for modeling particle transport because of their ability to capture the asymptotic diffusion limit with coarse meshes \cite{larsenMorel_1989,adams_2001,guermondKanschat_2010} without modification of numerical fluxes (as in, e.g., \cite{audit_etal_2002}), and we leverage this property here.  
Moreover, their variational formulation and flexibility with respect to test functions make them suitable for designing methods that conserve particle number and total energy \emph{simultaneously} (e.g., \cite{ayuso_etal_2011,cheng_etal_2013b}), which can be more difficult to achieve with, e.g., finite-difference or finite-volume methods.  
We use IMEX time stepping \cite{ascher_etal_1997,pareschiRusso_2005} to integrate the ordinary differential equations resulting from the semi-discretization of the moment equations by the DG method.  
Following our prior works \cite{chu_etal_2019,laiu_etal_2021}, we integrate the phase-space advection problem explicitly and the collision term implicitly.  
However, different from our prior works, due to the additional $\cO(v/c)$ terms in the time derivatives of the moment equations, the implicit part is nonlinear, even for the simplified collision term we consider here, and requires an iterative solution procedure, which we formulate in this paper.  

Given appropriate initial and boundary conditions, the solution to moment models with maximum-entropy closure is known to be \emph{realizable}; i.e., the moment solution is consistent with a kinetic distribution $f$ that satisfies required physical bounds \cite{levermore_1996,alldredge2019regularized}.  
For particle systems obeying Bose--Einstein or Maxwell--Boltzmann statistics, $f$ is nonnegative, whereas for particle systems obeying Fermi--Dirac statistics, $f\in[0,1]$.
These bounds translate into constraints on the associated moments, and moments satisfying these constraints are referred to as ``realizable'' moments.  
Although moment realizability is preserved by continuous moment models, solving moment models numerically can result in unrealizable moments, which leads to ill-posedness of the closure procedure and can give unphysical results when coupling moment models to other physical models, such as fluid models.
Therefore, maintaining moment realizability has been a key challenge in the design of numerical schemes for solving moment equations and has been explored in existing work from different perspectives, including development of realizability-preserving spatio-temporal discretizations \cite{olbrant2012realizability,hauck2011high,chu_etal_2019}, design of realizability-enforcing limiters \cite{chu_etal_2019}, and relaxation of the realizability constraints via regularization \cite{alldredge2019regularized}.
While these existing approaches provide some essential components to construct a realizability-preserving scheme for the $\cO(v/c)$ two-moment model considered in this work, they focus on models without any relativistic corrections and do not fully address the challenges of preserving moment realizability when relativistic corrections are included.

The realizability-preserving numerical scheme proposed in this paper consists of the following key components. 
First, for time integration, we adopt a strong stability-preserving (SSP) IMEX method, which treats the advection terms explicitly and the collision term implicitly.  
This choice avoids excessive time-step restrictions in the highly collisional regime and gives explicit stage updates that can be expressed as a convex combination of multiple forward Euler steps, which is necessary for preserving realizability.  
Second, the DG method is equipped with tailored numerical fluxes, which, together with the SSP IMEX time integration method, maintains nonnegative cell-averaged number densities in the explicit update under a time-step restriction that takes the form of a hyperbolic-type Courant--Friedrichs--Lewy (CFL) condition.
Third, the realizability-enforcing limiter proposed in \cite{chu_etal_2019} is used to recover pointwise realizable moments after each stage of the IMEX method.
As discussed above, the moment closure procedure requires an iterative solver for nonlinear equations that convert evolved (conserved) moments to the primitive moments.  
To preserve realizability in this conversion process, we formulate the nonlinear equation as a fixed-point problem and apply an iterative solver analogous to the modified Richardson iteration (e.g., \cite{richardson1911ix,saad2003iterative}) to ensure realizability in each iteration. 
We prove the global convergence property of this iterative solver in the $\cO(v/c)$ regime. 
The convergence analysis is applicable to the maximum-entropy closure as well as its algebraic approximation. 
Finally, the nonlinear systems arising from the implicit step of the IMEX method can also be formulated as a fixed-point problem and solved in a similar fashion. 
The realizability-preserving and convergence analyses both carry through with minor modifications. 
With these components in hand, we prove that the proposed DG-IMEX scheme for solving the $\cO(v/c)$ two-moment model indeed preserves moment realizability.

The two-moment model we consider is number conservative and, in the continuum limit, consistent to $\cO(v/c)$ with phase-space conservation laws for Eulerian-frame energy and momentum.  
Because the Eulerian-frame energy is not a primary evolved quantity of the model, but is instead obtained from a nontrivial combination of the evolved quantities, similar consistency with this conservation law is not guaranteed at the discrete level.  
In the context of finite-difference methods, Liebend{\"o}rfer et al. \cite{liebendorfer_etal_2004} proposed a consistent discretization by carefully matching specific numerical flux terms in the finite-difference representation of the general-relativistic Boltzmann equation (see also \cite{muller_etal_2010} for an approach in the case of moment models).  
For the semi-discrete DG scheme proposed here, the numerical fluxes are tailored to maintain moment realizability, which limits the flexibility of following this procedure.  
However, the flexibility provided by the approximation spaces of the DG method can be helpful in this respect.  
For example, by testing with the particle energy $\varepsilon$, which is represented exactly by the DG approximation space with linear functions in the energy dimension, we obtain the two-moment model promoted in \cite{lowrie_etal_2001}.  
We further analyze the \emph{simultaneous} Eulerian-frame number and energy conservation properties of the semi-discrete DG scheme, and point out that our DG approximation of the background velocity, which is allowed to be discontinuous, can impact the ability to achieve consistency with Eulerian-frame energy conservation to $\cO(v/c)$.  
Moreover, we design an ``energy limiter'' that corrects for Eulerian-frame energy conservation violations introduced by the realizability-enforcing limiter mentioned above.  
Through numerical experiments, we observe that Eulerian-frame energy conservation violations grow as $(v/c)^{2}$, indicating the desired consistency for an $\cO(v/c)$ method.  

The paper is organized as follows.  
The mathematical formulation of the two-moment model is presented in Section~\ref{sec:model}, while the closure procedure and wave propagation speeds supported by the resulting moment model are presented and discussed in Section~\ref{sec:closure}.  
Section~\ref{sec:discretization} provides an overview of the numerical method, including the DG phase-space discretization, IMEX time discretization, and iterative solvers for the nonlinear systems arising from the conserved-to-primitive conversion problem and time-implicit evaluation of the collision term.  
Section~\ref{sec:realizability_preservation}, where the realizability-preserving property of the method is established, contains the main technical results of the paper.  
The simultaneous conservation of Eulerian-frame number and energy of the DG method is discussed in Section~\ref{sec:conservation}, where the energy limiter that corrects for Eulerian-frame energy conservation violations introduced by the realizability-enforcing limiter is also presented.  
The algorithms have been implemented in the toolkit for high-order neutrino radiation-hydrodynamics (\thornado\footnote{\href{https://github.com/endeve/thornado}{www.github.com/endeve/thornado}}) and have been ported to utilize graphics processing units (GPUs).  
Our GPU programming model and implementation strategy is briefly discussed in Section~\ref{sec:gpu}.  
Results from numerical experiments demonstrating the robustness and accuracy of our method are presented in Section~\ref{sec:numericalResults}, where we also present GPU and multi-core performance results and highlight the relative computational cost of algorithmic components.  
Some technical proofs are given in \ref{sec:appendix}.  

For the remainder of this paper we employ units in which the speed of light is unity ($c=1$).  
\section{Mathematical Model}
\label{sec:model}

We consider a kinetic model where we solve for angular moments of the distribution function $f\colon (\omega,\varepsilon,\vect{x},t)\in\bbS^{2}\times\bbR^{+}\times\bbR^{3}\times\bbR^{+}\to\bbR^{+}$, which gives the number of particles propagating in the direction $\omega\in\bbS^{2}:=\{\,\omega=(\vartheta,\varphi)~|~\vartheta\in[0,\pi], \varphi\in[0,2\pi)\,\}$, with energy $\varepsilon\in\bbR^{+}$, at position $\vect{x}\in\bbR^{3}$ and time $t\in\bbR^{+}$.  
We define angular moments of $f$ as
\begin{equation}
  \big\{\,\mathcal{D},\,\mathcal{I}^{i},\,\mathcal{K}^{ij},\,\mathcal{Q}^{ijk}\,\big\}(\varepsilon,\vect{x},t)
  =\f{1}{4\pi}\int_{\bbS^{2}}f(\omega,\varepsilon,\vect{x},t)\,\big\{\,1,\,\ell^{i},\,\ell^{i}\ell^{j},\,\ell^{i}\ell^{j}\ell^{k}\,\big\}\,d\omega,
  \label{eq:angularMoments}
\end{equation}
where $\ell^{i}(\omega)$ is the $i$th component of a unit vector parallel to the particle three-momentum $\vect{p}=\varepsilon\,\vect{\ell}$, and $d\omega=\sin\vartheta\,d\vartheta\,d\varphi$.  
We take $\vect{p}=\big(p^{1},p^{2},p^{3}\big)^{\intercal}$ to be the particle three-momentum, and $\varepsilon$ and $\omega$ the particle energy and direction in a spherical-polar momentum-space coordinate system associated with an observer instantaneously moving with the fluid three-velocity $\vect{v}$ (the comoving observer).  
This choice of momentum-space coordinates is commonly used to model particles interacting with a moving material, as it simplifies the particle--material interaction (collision) terms (see, e.g., \cite{buchler_1979,mihalasMihalas_1999}).  
For simplicity, we will assume that the components of the three-velocity $v^{i}$ are given functions of position $\vect{x}$, \emph{independent} of time $t$.  
In Eq.~\eqref{eq:angularMoments}, $\mathcal{D}$ and $\mathcal{I}^{i}$ are the comoving-frame, spectral particle density and flux density components, respectively.  

Moment models that incorporate moving fluid effects are derived in the framework of relativistic kinetic theory \cite{lindquist_1966}, and the moment model considered here is obtained from the general relativistic two-moment model from \cite{cardall_etal_2013a}.  
Specifically, we consider the number-conservative two-moment model presented in Section~4.7.3 in \cite{mezzacappa_etal_2020}, after taking the limit of flat spacetime, specializing to Cartesian spatial coordinates, and retaining velocity-dependent terms to $\mathcal{O}(v)$.  
In this limit, the zeroth-moment equation is given by
\begin{align}
  \pd{}{t}\big(\,\mathcal{D}+v^{i}\,\mathcal{I}_{i}\,\big)
  +\pd{}{i}\big(\,\mathcal{I}^{i}+v^{i}\,\mathcal{D}\,\big)
  -\f{1}{\varepsilon^{2}}\pd{}{\varepsilon}\big(\,\varepsilon^{3}\,\mathcal{K}^{i}_{\hspace{2pt}k}\,\pd{v^{k}}{i}\,\big)
  =\chi\,\big(\,\mathcal{D}_{0}-\mathcal{D}\,\big),
  \label{eq:spectralNumberEquation}
\end{align}
where $\pd{}{t}=\partial/\partial t$, $\pd{}{i}=\partial/\partial x^{i}$, and $\pd{}{\varepsilon}=\partial/\partial\varepsilon$.  
We use the Einstein summation convention, where repeated latin indices run from $1$ to $3$.  
In flat spacetime, assuming Cartesian spatial coordinates, we can raise and lower indices on vectors and tensors with the Kronecker tensor; e.g., $\mathcal{I}_{i}=\delta_{ij}\mathcal{I}^{j}$.  
On the right-hand side of Eq.~\eqref{eq:spectralNumberEquation}, $\chi\ge0$ is the absorption opacity, and $\mathcal{D}_{0}$ is the zeroth moment of an equilibrium distribution $f_{0}$.  
The corresponding first-moment equation is given by
\begin{align}
  &\pd{}{t}\big(\,\mathcal{I}_{j}+v^{i}\,\mathcal{K}_{ij}\,\big)
  +\pd{}{i}\big(\,\mathcal{K}^{i}_{\hspace{2pt}j}+v^{i}\,\mathcal{I}_{j}\,\big)
  -\f{1}{\varepsilon^{2}}\pd{}{\varepsilon}\big(\,\varepsilon^{3}\,\mathcal{Q}^{i}_{\hspace{2pt}kj}\,\pd{v^{k}}{i}\,\big) \nonumber \\
  &\hspace{12pt}
  +\mathcal{I}^{i}\,\pd{v_{j}}{i} - \mathcal{Q}^{i}_{\hspace{2pt}kj}\,\pd{v^{k}}{i}
  =-\kappa\,\mathcal{I}_{j},
  \label{eq:spectralNumberFluxEquation}
\end{align}
where $\kappa=\chi+\sigma$ is the sum of the absorption opacity and the opacity due to elastic and isotropic scattering ($\sigma\ge0$).  

The two-moment model given by Eqs.~\eqref{eq:spectralNumberEquation} and \eqref{eq:spectralNumberFluxEquation} correspond to the moment equations for number transport given by Just et al.~\cite{just_etal_2015}; their Equations (9a) and (9b).  
(See also Eq.~(125) in \cite{endeve_etal_2012} for the number-density equation.)  
The velocity-dependent terms in the spatial and energy derivatives in Eqs.~\eqref{eq:spectralNumberEquation} and \eqref{eq:spectralNumberFluxEquation} account for spatial advection and Doppler shift between adjacent comoving observers, respectively, while the fourth and fifth terms on the left-hand side of Eq.~\eqref{eq:spectralNumberFluxEquation} account for angular aberration between adjacent comoving observers (e.g., \cite{liebendorfer_etal_2004}).  
We point out that the velocity-dependent terms inside the time derivatives in Eqs.~\eqref{eq:spectralNumberEquation} and \eqref{eq:spectralNumberFluxEquation} were dropped in \cite{just_etal_2015}.  
By retaining these terms, Eq.~\eqref{eq:spectralNumberEquation} evolves the $\mathcal{O}(v)$ Eulerian-frame number density, and, as emphasized by Lowrie et al.~\cite{lowrie_etal_2001}, wave speeds remain bounded by the speed of light and the model is consistent with the correct $\mathcal{O}(v)$ Eulerian-frame energy and momentum equations.  
To elaborate on the latter, we define the ``conserved" moments that are evolved in Eqs.~\eqref{eq:spectralNumberEquation} and \eqref{eq:spectralNumberFluxEquation} as
\begin{equation}
	\mathcal{N} 
	:= \mathcal{D}+v^{i}\,\mathcal{I}_{i} \quand
	\mathcal{G}_{j} 
	:= \mathcal{I}_{j}+v^{i}\,\mathcal{K}_{ij},
	\label{eq:eachconservedMoments}
\end{equation}
respectively.  
Here, $\mathcal{N}$ is the correct $\mathcal{O}(v)$ Eulerian-frame number density, and, in the absence of sources on the right-hand side, Eq.~\eqref{eq:spectralNumberEquation} is a phase-space conservation law.  
The Eulerian-frame energy and momentum densities are related to $\mathcal{N}$ and $\mathcal{G}_{j}$ by 
\begin{equation}
	\mathcal{E} = \varepsilon\, ( \mathcal{N} +  v^{i} \,\mathcal{G}_{i} )
	= \varepsilon \,(\mathcal{D} + \, 2 v^{i} \,\mathcal{I}_{i}) +\mathcal{O}(v^2)
	\label{eq:conservedEnergy}
\end{equation}
and
\begin{equation}
	\mathcal{P}_{j} = \varepsilon \,( \mathcal{G}_{j} +  v_{j} \,\mathcal{N})
	=\varepsilon \,(\mathcal{I}_{j}+v^{i}\,\mathcal{K}_{ij} +  v_{j} \,\mathcal{D}) +\mathcal{O}(v^2),
	\label{eq:conservedMomentum}
\end{equation} 
respectively. 
The following proposition gives the energy and momentum conservation properties of the two-moment model in Eqs.~\eqref{eq:spectralNumberEquation}--\eqref{eq:spectralNumberFluxEquation}.  

\begin{prop}\label{prop:EnergyandMomentumConservation}
	The two-moment model given by Eqs.~\eqref{eq:spectralNumberEquation}--\eqref{eq:spectralNumberFluxEquation} is, up to $\mathcal{O}(v)$, consistent with phase-space conservation laws for the energy density $\mathcal{E}$ and momentum density $\mathcal{P}_{j}$.
\end{prop}
\begin{proof}
	By multiplying Eqs.~\eqref{eq:spectralNumberEquation} and \eqref{eq:spectralNumberFluxEquation} with appropriate factors and summing up the resulting equations, the evolution equations for the energy and momentum densities can be derived, respectively, as
	\begin{align}
		\pd{\mathcal{E}}{t} + \pd{}{i}\,{\mathcal{P}}^{i}
		-\f{1}{\varepsilon^{2}}\pd{}{\varepsilon}
		\big(\,\varepsilon^{4}\,{\mathcal{K}}_{\hspace{2pt}k}^{i}\,\pd{v^{k}}{i}\,\big)
		=\varepsilon\,\chi\,\big(\,\mathcal{D}_{0}-\mathcal{D}\,\big) - 
		\,\varepsilon\,\kappa\,v^{j}\,\mathcal{I}_{j}
		\label{eq:energyEquationEulerian}
	\end{align}
	and
	\begin{align}
		\pd{{\mathcal{P}}_{j}}{t}
		+\pd{}{i}\,\mathcal{S}_{\hspace{2pt}j}^{i}
		-\f{1}{\varepsilon^{2}}\pd{}{\varepsilon}
		\big(\,\varepsilon^{4}\,
		{\mathcal{Q}}_{\hspace{2pt}kj}^{i}\,\pd{v^{k}}{i}\,\big)
		=  - \varepsilon\,\kappa\,\mathcal{I}_{j} +  \varepsilon\,v_{j}\,\chi\,\big(\,\mathcal{D}_{0}-\mathcal{D}\,\big).
		\label{eq:momentumEquationEulerian}
	\end{align}
	Here, all $\mathcal{O}(v^2)$ terms are dropped, and the momentum flux density is denoted as $\mathcal{S}^{ij}:= \varepsilon\,(\mathcal{K}^{ij}+\mathcal{I}^{i}\,v^{j} + v^{i}\,\mathcal{I}^{j}$). 
	In the absence of sources on the right-hand side, Eqs.~\eqref{eq:energyEquationEulerian} and \eqref{eq:momentumEquationEulerian} become phase-space conservation laws for $\mathcal{E}$ and $\mathcal{P}_{j}$, respectively.  
\end{proof}

To close the two-moment model \eqref{eq:spectralNumberEquation}--\eqref{eq:spectralNumberFluxEquation}, the higher-order moments $\mathcal{K}^{ij}$ and $\mathcal{Q}^{ijk}$ must be specified.  
We will use an algebraic closure, which we discuss in more detail in Section~\ref{sec:closure}.  
To this end, we write the second-order moments as
\begin{equation}
  \mathcal{K}^{ij} = \mathsf{k}^{ij}\,\mathcal{D},
\end{equation}
where the symmetric variable Eddington tensor components are given by (e.g., \cite{levermore_1984})
\begin{equation}
  \mathsf{k}^{ij} = \f{1}{2}\,\Big[\,(1-\psi)\,\delta^{ij}+(3\psi-1)\,\hat{\mathsf{n}}^{i}\,\hat{\mathsf{n}}^{j}\,\Big],
  \label{eq:VariableEddingtonTensor}
\end{equation}
where $\hat{\mathsf{n}}^{i}=\mathcal{I}^{i}/\mathcal{I}$ and $\mathcal{I}=\sqrt{\mathcal{I}_{i}\mathcal{I}^{i}}$.  
The expression given by Eq.~\eqref{eq:VariableEddingtonTensor} satisfies the trace condition $\mathsf{k}^{i}_{\hspace{4pt}i}=\delta_{ij}\mathsf{k}^{ij}=1$ (cf. Eq.~\eqref{eq:angularMoments}), and the Eddington factor can be obtained from
\begin{equation}
   \psi = \hat{\mathsf{n}}_{i}\,\hat{\mathsf{n}}_{j}\,\mathsf{k}^{ij} = \f{\int_{\bbS^{2}}f\,(\hat{\mathsf{n}}_{i}\ell^{i})^{2}\,d\omega}{\int_{\bbS^{2}}f\,d\omega}.  
\end{equation}

Similarly, the third-order moments can be written as
\begin{equation}
  \mathcal{Q}^{ijk} = \mathsf{q}^{ijk}\,\mathcal{D},
\end{equation}
where we define the symmetric ``heat-flux'' tensor (e.g., \cite{just_etal_2015}),
\begin{equation}
  \mathsf{q}^{ijk} 
  = \f{1}{2}\,
  \Big[\,
    (h-\zeta)\,\Big(\,\hat{\mathsf{n}}^{i}\,\delta^{jk}+\hat{\mathsf{n}}^{j}\,\delta^{ik}
    +\hat{\mathsf{n}}^{k}\,\delta^{ij}\,\Big)+(5\zeta-3h)\,\hat{\mathsf{n}}^{i}\,\hat{\mathsf{n}}^{j}\,\hat{\mathsf{n}}^{k}
  \,\Big],
  \label{eq:heatfluxTensor}
\end{equation}
where $h=\mathcal{I}/\mathcal{D}$ is the flux factor.  
The expression in Eq.~\eqref{eq:heatfluxTensor} satisfies the trace condition $\delta_{jk}\,\mathsf{q}^{ijk}=\mathsf{q}^{ij}_{\vipt j}=\mathcal{I}^{i}/\mathcal{D}$, and the ``heat-flux'' factor can be obtained from 
\begin{equation}
  \zeta = \hat{\mathsf{n}}_{i}\,\hat{\mathsf{n}}_{j}\,\hat{\mathsf{n}}_{k}\,\mathsf{q}^{ijk} = \f{\int_{\bbS^{2}}f\,(\hat{\mathsf{n}}_{i}\ell^{i})^{3}\,d\omega}{\int_{\bbS^{2}}f\,d\omega}.  
\end{equation}
Eqs~\eqref{eq:spectralNumberEquation} and \eqref{eq:spectralNumberFluxEquation} are closed by specifying the Eddington and heat-flux factors in terms of the ``primitive'' moments $\vect{\mathcal{M}}=\big(\,\mathcal{D},\,\vect{\mathcal{I}}\,\big)^{\intercal}$; i.e., $\psi=\psi(\vect{\mathcal{M}})$ and $\zeta=\zeta(\vect{\mathcal{M}})$.  

Assuming a closure for the higher-order tensors, we define the vector of evolved moments,
\begin{equation}
  \vect{\mathcal{U}}(\vect{\mathcal{M}},\vect{v})
  =\left[\begin{array}{c}
  \mathcal{N} \\
  \mathcal{G}_{j}
  \end{array}\right]
  =\left[\begin{array}{c}
  \mathcal{D}+v^{i}\,\mathcal{I}_{i} \\
  \mathcal{I}_{j}+v^{i}\,\mathcal{K}_{ij}
  \end{array}\right],
  \label{eq:conservedMoments}
\end{equation}
the phase-space fluxes,
\begin{equation}
  \vect{\mathcal{F}}^{i}(\vect{\mathcal{U}},\vect{v})
  =\left[\begin{array}{c}
  \mathcal{I}^{i}+v^{i}\,\mathcal{D} \\
  \mathcal{K}^{i}_{\iipt j}+v^{i}\,\mathcal{I}_{j}
  \end{array}\right]
  \quad\text{and}\quad
  \vect{\mathcal{F}}^{\varepsilon}(\vect{\mathcal{U}},\vect{v})
  =-\left[\begin{array}{c}
  \mathcal{K}^{i}_{\iipt k} \\
  \mathcal{Q}^{i}_{\iipt kj}
  \end{array}\right]\,\pd{v^{k}}{i},
  \label{eq:phaseSpaceFluxes}
\end{equation}
and the sources,
\begin{equation}
  \vect{\mathcal{S}}(\vect{\mathcal{U}},\vect{v})
  =\left[\begin{array}{c}
  0 \\
  \mathcal{Q}^{i}_{\hspace{2pt}kj}\,\pd{v^{k}}{i} - \mathcal{I}^{i}\,\pd{v_{j}}{i}
  \end{array}\right]
  \quad\text{and}\quad
  \vect{\mathcal{C}}(\vect{\mathcal{U}})
  =\left[\begin{array}{c}
  \chi\,\big(\,\mathcal{D}_{0}-\mathcal{D}\,\big) \\
  -\kappa\,\mathcal{I}_{j}
  \end{array}\right],
  \label{eq:sources}
\end{equation}
so we can write the two-moment model in the compact form,
\begin{equation}
  \pd{\vect{\mathcal{U}}}{t}
  +\pderiv{}{x^{i}}\Big(\vect{\mathcal{F}}^{i}(\vect{\mathcal{U}},\vect{v})\Big)
  +\f{1}{\varepsilon^{2}}\pderiv{}{\varepsilon}\Big(\varepsilon^{3}\,\vect{\mathcal{F}}^{\varepsilon}(\vect{\mathcal{U}},\vect{v})\Big)
  =\vect{\mathcal{S}}(\vect{\mathcal{U}},\vect{v}) + \vect{\mathcal{C}}(\vect{\mathcal{U}}).
  \label{eq:twoMomentModelCompact}
\end{equation}
Note that the collision term $\vect{\mathcal{C}}$ does not depend explicitly on the three-velocity $\vect{v}$.  
This is a consequence of choosing comoving-frame, momentum-space coordinates.  

The moment closure is defined in terms of the primitive moments $\vect{\mathcal{M}}$, while we will evolve the ``conserved'' moments $\vect{\mathcal{U}}=\big(\,\mathcal{N},\mathcal{G}_{j}\,\big)^{\intercal}$.  
The relation between the conserved and primitive moments can be written as
\begin{equation}
  \vect{\mathcal{U}}
  =\vect{\mathcal{L}}(\vect{\mathcal{M}},\vect{v})\,\vect{\mathcal{M}},
  \label{eq:ConservedToPrimitive}
\end{equation}
where
\begin{equation}
  \vect{\mathcal{L}}(\vect{\mathcal{M}},\vect{v})
  =\left[\begin{array}{cccc}
  1 & v^{1} & v^{2} & v^{3} \\
  v^{k}\,\mathsf{k}_{k1}(\vect{\mathcal{M}}) & 1 & 0 & 0 \\
  v^{k}\,\mathsf{k}_{k2}(\vect{\mathcal{M}}) & 0 & 1 & 0 \\
  v^{k}\,\mathsf{k}_{k3}(\vect{\mathcal{M}}) & 0 & 0 & 1
  \end{array}\right].  
  \label{eq:MatrixFormL}
\end{equation}
When solving Eq.~\eqref{eq:twoMomentModelCompact} numerically, it is necessary to convert between primitive and conserved moments.  
Computing the conserved moments from the primitive moments is straightforward, but obtaining the primitive moments from the conserved moments is nontrivial because, for a given nontrivial velocity $\vect{v}$, there is no closed-form expression for $\vect{\mathcal{M}}$ in terms of $\vect{\mathcal{U}}$, due to the nonlinear dependence $\mathsf{k}_{ij}(\vect{\mathcal{M}})$.  
Thus, the primitive moments must be obtained through an iterative procedure, which we discuss in more detail later, where we will pay particular attention to maintaining physically-realizable moments throughout the iteration process.  
One is faced with a similar problem, e.g., when solving the relativistic Euler and magnetohydrodynamics equations (e.g., \cite{noble_etal_2006}).  
\section{Moment Closure}
\label{sec:closure}

We use the maximum-entropy closure of Minerbo \cite{minerbo_1978} to close the two-moment model.  
We let the admissible set of kinetic distribution functions be
\begin{equation}
  \mathfrak{R} := \left\{\, f ~|~ f\ge 0 \quad\text{and}\quad \f{1}{4\pi}\int_{\mathbb{S}^{2}}f\,d\omega > 0 \,\right\},
\end{equation}
which is then used to define moment realizability as below.%
\footnote{The admissible set $\mathfrak{R}$ and the realizable set $\mathcal{R}$ in this work are appropriate for particle systems obeying Bose--Einstein or Maxwell--Boltzmann statistics. The extension of this work to systems obeying Fermi--Dirac statistics, where $f$ is also bounded from above, is non-trivial and deferred to future work.}
\begin{define}
  The moments $\vect{\mathcal{M}}=(\mathcal{D},\vect{\mathcal{I}})^{\intercal}$ are realizable if they can be obtained from a distribution function $f(\omega)\in\mathfrak{R}$.  
  The set of all realizable moments $\mathcal{R}$ is
  \begin{equation}
    \mathcal{R} := \big\{\,\vect{\mathcal{M}}=(\mathcal{D},\vect{\mathcal{I}})^{\intercal} ~|~ \mathcal{D} > 0 ~\text{and}~ \gamma(\vect{\mathcal{M}})=\mathcal{D}-\mathcal{I}\ge0\,\big\},
    \label{eq:realizableSet}
  \end{equation}
  where the function $\gamma(\vect{\mathcal{M}})$ is concave.  
\end{define}
 
The Minerbo closure is based on the maximum-entropy principle, assuming an entropy functional of the form $s[f] = f\,\ln f - f$.  
The functional form of the distribution maximizing this entropy functional is, in this case, the Maxwell--Boltzmann distribution,
\begin{equation}
  f_{\ME}(\omega) = \exp\big(\alpha+\beta\,(\hat{\mathsf{n}}_{i}\ell^{i})\big),
  \label{eq:fME}
\end{equation}
where $\alpha$ and $\beta$ are determined from the constraints,
\begin{equation}
  \mathcal{D}=\f{1}{4\pi}\int_{\mathbb{S}^{2}}f_{\ME}(\omega)\,d\omega
  \quad\text{and}\quad
  \hat{\mathsf{n}}_{i}\,\mathcal{I}^{i}=\mathcal{I}=\f{1}{4\pi}\int_{\mathbb{S}^{2}}f_{\ME}(\omega)\,(\hat{\mathsf{n}}_{i}\ell^{i})\,d\omega.  
  \label{eq:closureConstraints}
\end{equation}
(Note that $f_{\ME}\in\mathfrak{R}$.)  
Letting $\hat{\mathsf{n}}_{i}\ell^{i}=\mu$, we can write $f_{\ME}$ as a function of $\mu$ and perform a change of variable to write the integrals in Eq.~\eqref{eq:closureConstraints} in terms of $\mu$, which allows us to evaluate the constraints in Eq.~\eqref{eq:closureConstraints} analytically (cf. \cite{minerbo_1978}) and leads to
\begin{equation}
  \mathcal{D}
  =e^{\alpha}\,\sinh(\beta)/\beta
  \quad\text{and}\quad
  \mathcal{I}
  =e^{\alpha}\,\big(\,\beta\,\cosh(\beta)-\sinh(\beta)\,\big)/\beta^{2}.
\end{equation}
The flux factor can then be written solely as a function of $\beta$; i.e., $h=\coth(\beta)-1/\beta=: L(\beta)$, where $L(\beta)$ is the Langevin function.  
Thus, for a given $h$, we can obtain $\beta(h)=L^{-1}(h)$.  
Note that $L(\beta)\in(-1,1)$, so that solutions for $\beta$ only exist for $h<1$ (i.e., for $\vect{\mathcal{M}}$ in the interior of $\mathcal{R}$).  
Using the maximum-entropy distribution in Eq.~\eqref{eq:fME}, direct calculations give, for $h \in[0,1)$,
\begin{equation}
  \psi(h)
  = 1 - \f{2\,h}{\beta(h)}
  \quad\text{and}\quad
  \zeta(h)
  = \coth(\beta(h))-3\psi(h)/\beta(h)  \label{eq:psiZetaMinerbo}.  
\end{equation}
When $h=1$ (i.e., when $\vect{\mathcal{M}}$ is on the boundary of $\mathcal{R}$), it is known \cite{fialkow1991recursiveness} that, for the two-moment case considered here, the underlying kinetic distribution is a weighted Dirac delta function.  
In this case, the moment closure is given by the associated Eddington and heat-flux factors $\psi(1) = \zeta(1) = 1$.
Instead of inverting the Langevin function for $\beta$, the Eddington and heat-flux factors, $\psi$ and $\zeta$, can be accurately approximated by polynomials in $h$.  
For $\psi$, the following polynomial approximation leads to a relative approximation error, $\delta\psi:=(\psi - \psi_{\mathsf{a}}) / \psi$, within $1\%$ \cite{cernohorskyBludman_1994}:
\begin{equation}
  \psi_{\mathsf{a}}(h) = \f{1}{3} + \f{2}{15}\,\big(\,3\,h^{2} - h^{3} + 3\,h^{4}\,\big).
  \label{eq:psiApproximate}
\end{equation}
For $\zeta$, the following approximation, given by \cite{just_etal_2015},
\begin{equation}
  \zeta_{\mathsf{a}}(h) = h\,\big(\,45 + 10\,h - 12\,h^{2} - 12\,h^{3} + 38\,h^{4} - 12\,h^{5} + 18\,h^{6}\,\big) / 75,
  \label{eq:zetaApproximate}
\end{equation}
has a relative approximation error, $\delta\zeta := (\zeta - \zeta_{\mathsf{a}}) / \zeta$, lower than $3\%$ .  
In Figure~\ref{fig:eddingtonFactors}, we plot the Eddington factor, $\psi$, the heat-flux factor, $\zeta$, and their polynomial approximations, $\psi_{\mathsf{a}}$ and $\zeta_{\mathsf{a}}$, and report the relative approximation error versus the flux factor, $h$.
\begin{figure}[h]
  \centering
  \begin{tabular}{cc}
    \includegraphics[width=0.5\textwidth]{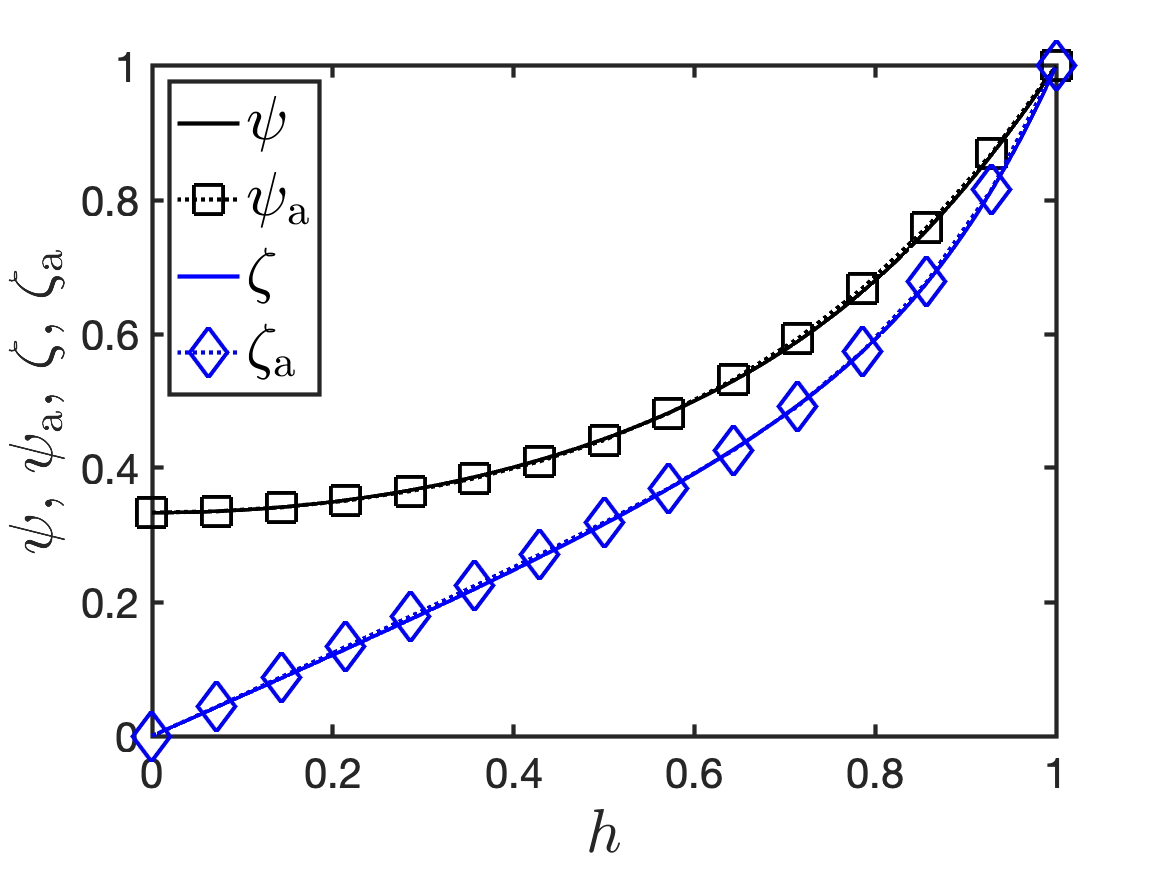} &
    \includegraphics[width=0.5\textwidth]{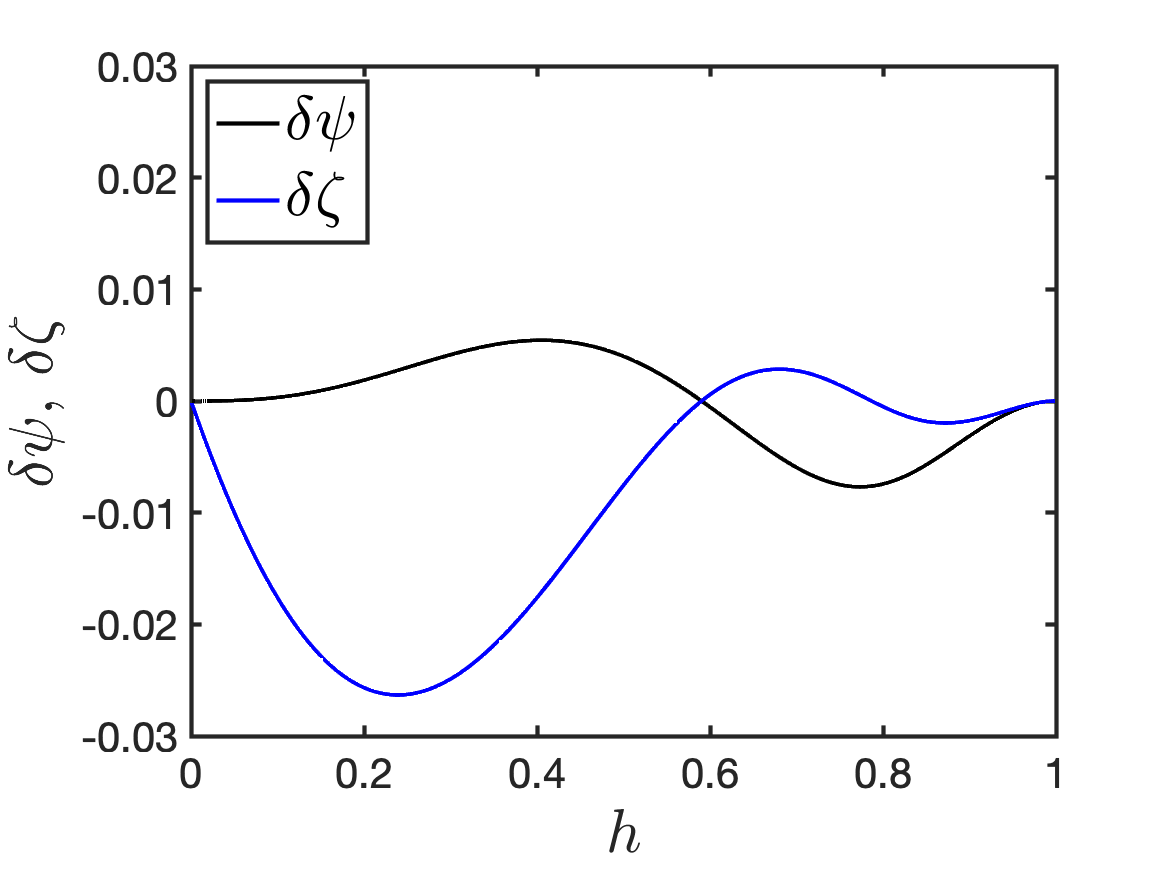}
  \end{tabular}
   \caption{The left plot shows the values of the Eddington factor, $\psi$, the heat-flux factor, $\zeta$, and their polynomial approximations, $\psi_{\mathsf{a}}$ and $\zeta_{\mathsf{a}}$, versus the flux factor, $h$. The right plot illustrates the relative errors, $\delta\psi=(\psi - \psi_{\mathsf{a}}) / \psi$ and $\delta\zeta = (\zeta - \zeta_{\mathsf{a}}) / \zeta$, versus $h$.}
  \label{fig:eddingtonFactors}
\end{figure}
It can be seen from Figure~\ref{fig:eddingtonFactors} that $\psi_{\mathsf{a}}$ and $\zeta_{\mathsf{a}}$ are quite accurate polynomial approximations to the Eddington and heat-flux factors. 
Thus, the approximate closure is used in the numerical tests for the two-moment model reported in Section~\ref{sec:numericalResults}, in which the two-moment model is closed by plugging the algebraic expressions given in Eqs.~\eqref{eq:psiApproximate} and \eqref{eq:zetaApproximate} into the Eddington and heat-flux tensors in Eqs.~\eqref{eq:VariableEddingtonTensor} and \eqref{eq:heatfluxTensor}, respectively.  

Next we explore the wave propagation speeds of the moment system in Eq.~\eqref{eq:twoMomentModelCompact} with the approximate Minerbo closure introduced above.  
To calculate the wave speed, we compute the maximum magnitude of the eigenvalues of the spatial-flux Jacobians with respect to the conserved moments, $(\partial_{\vect{\mathcal{U}}}\vect{\mathcal{F}}^{i})$, $i=1,2,3$.
Specifically, we compute the spatial-flux Jacobian by
\begin{equation}
	\Big(\pderiv{\vect{\mathcal{F}}^{i}}{\vect{\mathcal{U}}}\Big)
	=\Big(\pderiv{\vect{\mathcal{F}}^{i}}{\vect{\mathcal{M}}}\Big)\Big(\pderiv{\vect{\mathcal{U}}}{\vect{\mathcal{M}}}\Big)^{-1},
\end{equation}
where
\begin{equation}
	\Big(\pderiv{\vect{\mathcal{U}}}{\vect{\mathcal{M}}}\Big)_{ij}
	=\left[\begin{array}{cc}
		1 & v^{j} \\
		v^{k}\Big[\Big(\pderiv{\mathsf{k}_{ik}}{\mathcal{D}}\Big)\,\mathcal{D}+\mathsf{k}_{ik}\Big] &
		\delta_{ij} + v^{k}\Big(\pderiv{\mathsf{k}_{ik}}{\mathcal{I}^{j}}\Big)\,\mathcal{D}
	\end{array}\right]
\end{equation}
and 
\begin{equation}
	\Big(\pderiv{\vect{\mathcal{F}}^{i}}{\vect{\mathcal{M}}}\Big)
	=\left[\begin{array}{cccc}
		v^{i} & \delta^{i1} & \delta^{i2} & \delta^{i3} \\
		\Big(\pderiv{\mathsf{k}^{i}_{\iipt 1}}{\mathcal{D}}\Big)\,\mathcal{D}+\mathsf{k}^{i}_{\iipt 1} & 
		\Big(\pderiv{\mathsf{k}^{i}_{\iipt 1}}{\mathcal{I}^{1}}\Big)\,\mathcal{D}+v^{i} & 
		\Big(\pderiv{\mathsf{k}^{i}_{\iipt 1}}{\mathcal{I}^{2}}\Big)\,\mathcal{D} & 
		\Big(\pderiv{\mathsf{k}^{i}_{\iipt 1}}{\mathcal{I}^{3}}\Big)\,\mathcal{D} \\
		\Big(\pderiv{\mathsf{k}^{i}_{\iipt 2}}{\mathcal{D}}\Big)\,\mathcal{D}+\mathsf{k}^{i}_{\iipt 2} &
		\Big(\pderiv{\mathsf{k}^{i}_{\iipt 2}}{\mathcal{I}^{1}}\Big)\,\mathcal{D} &
		\Big(\pderiv{\mathsf{k}^{i}_{\iipt 2}}{\mathcal{I}^{2}}\Big)\,\mathcal{D}+v^{i} &
		\Big(\pderiv{\mathsf{k}^{i}_{\iipt 2}}{\mathcal{I}^{3}}\Big)\,\mathcal{D} \\
		\Big(\pderiv{\mathsf{k}^{i}_{\iipt 3}}{\mathcal{D}}\Big)\,\mathcal{D}+\mathsf{k}^{i}_{\iipt 3} &
		\Big(\pderiv{\mathsf{k}^{i}_{\iipt 3}}{\mathcal{I}^{1}}\Big)\,\mathcal{D} &
		\Big(\pderiv{\mathsf{k}^{i}_{\iipt 3}}{\mathcal{I}^{2}}\Big)\,\mathcal{D} &
		\Big(\pderiv{\mathsf{k}^{i}_{\iipt 3}}{\mathcal{I}^{3}}\Big)\,\mathcal{D}+v^{i}
	\end{array}\right]
\end{equation}
follow from the definitions given in Eqs.~\eqref{eq:conservedMoments} and \eqref{eq:phaseSpaceFluxes}, respectively.
With this expression, we are able to demonstrate the following proposition, which states that the maximum wave speed is bounded above by the speed of light in a one-dimensional setting.
\begin{prop}\label{prop:waveSpeed}
	Suppose $\vect{v}=(v,0,0)$ and $\bcI=(\cI,0,0)$, with $|v|\leq 1$, $|\cI|\leq \cD$, and $\cD>0$. Let $\lambda_{\max}:=\max\big(|\lambda(\partial_{\vect{\mathcal{U}}}\vect{\mathcal{F}}^{1})|\big)$ denote the maximum magnitude of the spatial-flux Jacobian eigenvalues. Then $\lambda_{\max}\leq1$.
\end{prop}

\begin{proof}
	In this setting, the spatial-flux Jacobian reduces to a 2-by-2 matrix, because the entries associated with the $x^2$ and $x^3$ axes are all zeros. 
	In addition, the only nonzero component of the Eddington tensor is $\mathsf{k}_{11}$, which takes the values of the (approximate) Eddington factor $\psi_{\mathsf{a}}$. Thus, the partial derivatives $\pderiv{\mathsf{k}_{11}}{\mathcal{D}}$ and $\pderiv{\mathsf{k}^{1}_{\iipt 1}}{\mathcal{I}^{1}}$ become $\pderiv{\psi_{\mathsf{a}}}{\mathcal{D}}$ and $\pderiv{\psi_{\mathsf{a}}}{\mathcal{I}^{1}}$, respectively.
	Evaluating these partial derivatives using the chain rule then leads to
	\begin{equation}\label{eq:fluxJacobian}
	\Big(\pderiv{\vect{\mathcal{F}}^{1}}{\vect{\mathcal{U}}}\Big) = 
	\frac{1}{	1  - v^2 \psi_{\mathsf{a}} + v(1 + v h)\psi_{\mathsf{a}}^\prime}
	\left[\begin{array}{cc}
	v  - v \psi_{\mathsf{a}} + v(v +  h)\psi_{\mathsf{a}}^\prime	 &  1 - v^2\\
	(1 - v^2) (\psi_{\mathsf{a}} - h \psi_{\mathsf{a}}^{\prime})	 & v  - v \psi_{\mathsf{a}} + (1 + v h)\psi_{\mathsf{a}}^\prime
	\end{array}\right],
	\end{equation}
where $\psi_{\mathsf{a}}^\prime$ denotes the derivative of the approximate Eddington factor, $\psi_{\mathsf{a}}$, in Eq.~\eqref{eq:psiApproximate} with respect to the flux factor, $h$.
To prove the claim, we need to show that the eigenvalues of $(\partial_{\vect{\mathcal{U}}}\vect{\mathcal{F}}^{1})$ are in $[-1,1]$.
Since $\psi_{\mathsf{a}}$ and $\psi_{\mathsf{a}}^\prime$ are both one-dimensional polynomials in $h$, the proof of the claim is straightforward but tedious. Here we omit the detailed analysis and show in Figure~\ref{fig:WaveSpeed} the computed values of $\lambda_{\max}$ for $v\in[0,1]$ and $h\in[0,1]$, which illustrates that $\lambda_{\max}$ is bounded from above by one.
\end{proof}

\begin{figure}[h]
	\centering
	\subfloat[One-dimensional case]{\label{fig:WaveSpeed} \includegraphics[width=0.45\textwidth]{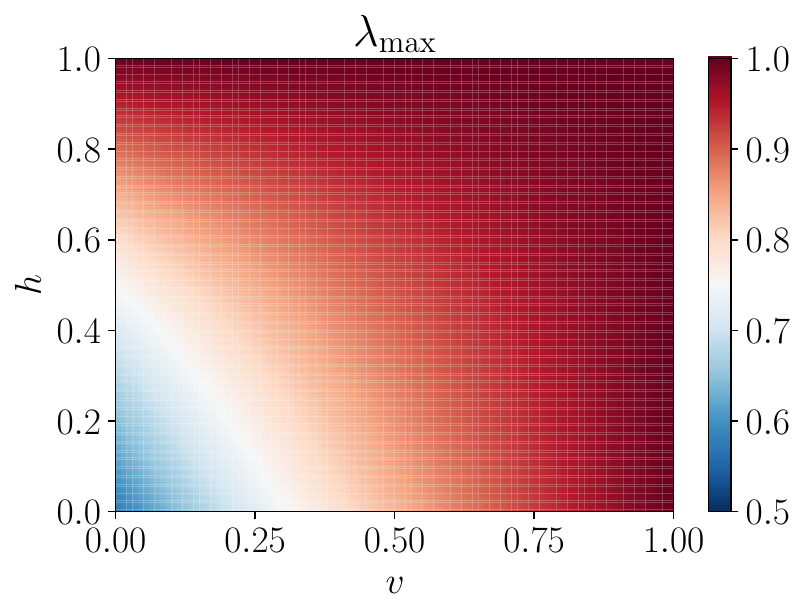}}\,\,
	\subfloat[Three-dimensional case]{\label{fig:WaveSpeedViolation}\includegraphics[width=0.475\textwidth]{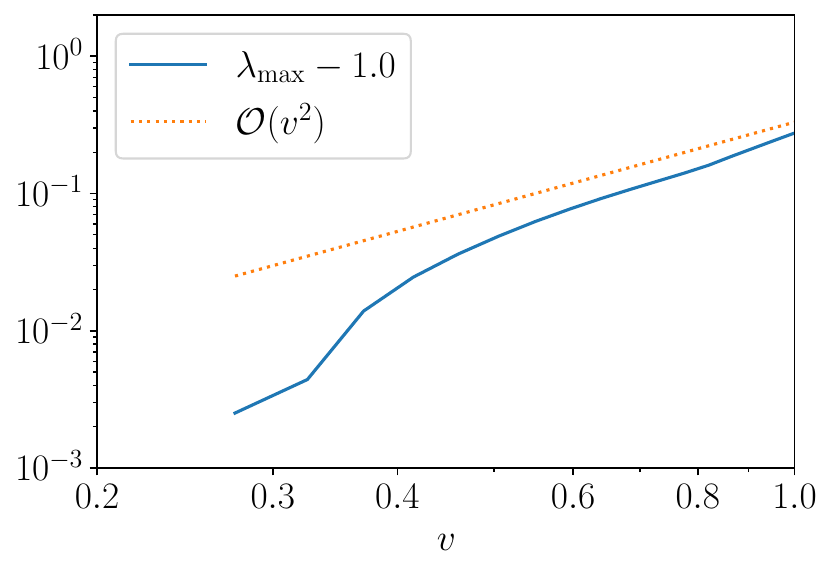}}
	\caption{Figures on both panels show the value of $\lambda_{\max}$, the maximum magnitude of the spatial-flux Jacobian eigenvalues in various configurations. Figure~\ref{fig:WaveSpeed} plots the computed values of $\lambda_{\max} = \max(\lambda(\partial_{\vect{\mathcal{U}}}\vect{\mathcal{F}}^{1}))$ at $v\in[0,1]$ and $h\in[0,1]$ in a simplified one-dimensional case considered in Proposition~\ref{prop:waveSpeed}. The result verifies the claim $\lambda_{\max}\leq1$ in Proposition~\ref{prop:waveSpeed}. Figure~\ref{fig:WaveSpeedViolation} shows that, in the three-dimensional case, the maximum wave speed of the two-moment model Eq.~\eqref{eq:twoMomentModelCompact} scales as $1+\mathcal{O}(v^2)$.
	Here, the maximum wave speed is given by $\displaystyle\lambda_{\max} := \max_{i=1,2,3}\max_{\bcU\in\cR}(\lambda_{\max}^{i})$, where $\lambda_{\max}^{i} = \max(\lambda(\partial_{\vect{\mathcal{U}}}\vect{\mathcal{F}}^{i}))$.}
\end{figure}

This result is an extension of the wave speed analysis in \cite[Section~6.2]{lowrie_etal_2001}, in which it is assumed that the Eddington factor is independent of the flux factor; i.e., that $\psi_{\mathsf{a}}^{\prime}=0$.

\begin{rem}
	In the three-dimensional case, the magnitude of the eigenvalues of the spatial-flux Jacobian are bounded above by $1+\mathcal{O}(v^2)$, which we provide verification of in Figure~\ref{fig:WaveSpeedViolation}. 
	Although the upper bound can exceed unity, which implies that the wave speed of the two-moment model in Eq.~\eqref{eq:twoMomentModelCompact} can become unphysical, it shows that including the velocity-dependent term in the time derivatives in Eqs.~\eqref{eq:spectralNumberEquation} and \eqref{eq:spectralNumberFluxEquation} improves the maximum wave speed estimation from $1+\mathcal{O}(v)$ (see, e.g., discussions in \cite{lowrie_etal_2001}) to $1+\mathcal{O}(v^2)$.  
	Note that in the design of the numerical flux discussed in Section~\ref{sec:dgMethod}, we use unity as the estimate for the maximum wave speed, which appears to be valid in the regimes for which the $\cO(v)$ model is applicable.  
	In particular, unphysical wave speeds are not observed for $v\leq0.25$, as shown in Figure~\ref{fig:WaveSpeedViolation}, for which we do not yet have a theoretical explanation.
\end{rem}

\section{Numerical Scheme}
\label{sec:discretization}
\subsection{Discontinuous Galerkin Phase-Space Discretization}
\label{sec:dgMethod}

We use the DG method to discretize Eq.~\eqref{eq:twoMomentModelCompact} in phase-space.  
To this end we divide the phase-space domain $D=D_{\varepsilon}\times D_{\vect{x}}$ into a disjoint union $\mathcal{T}$ of open elements $\vect{K}=K_{\varepsilon}\times\vect{K}_{\vect{x}}$, so that $D=\cup_{\vect{K}\in\mathcal{T}}\vect{K}$.  
Here, $D_{\varepsilon}$ is the energy domain and $D_{\vect{x}}$ is the $d_{\vect{x}}$-dimensional spatial domain, and
\begin{equation}
  \vect{K}_{\vect{x}} = \big\{\,\vect{x} \colon x^{i} \in K_{\vect{x}}^{i} := (x_{\lo}^{i},x_{\hi}^{i}) ~|~ i=1,\ldots,d_{\vect{x}}\,\big\}
  \quad\text{and}\quad
  K_{\varepsilon} := (\varepsilon_{\lo},\varepsilon_{\hi}),
\end{equation}
where $x_{\lo}^{i}$ ($x_{\hi}^{i}$) is the low (high) boundary of the spatial element in the $i$th spatial dimension, and $\varepsilon_{\lo}$ ($\varepsilon_{\hi}$) is the low (high) boundary of the energy element.  
We also define $\tau(\varepsilon)=\varepsilon^{2}$ and denote the volume of a phase-space element by
\begin{equation}
  |\vect{K}| = \int_{\vect{K}}\tau\,d\varepsilon\,d\vect{x}, \quad\text{where}\quad
  d\vect{x} = \prod_{i=1}^{d_{\vect{x}}}dx^{i}.  
\end{equation}
The length of an element in the $i$th dimension is $|K_{\vect{x}}^{i}|=x_{\hi}^{i}-x_{\lo}^{i}$, and $|K_{\varepsilon}|=\varepsilon_{\hi}-\varepsilon_{\lo}$.  
We also define the phase-space surface element $\tilde{\vect{K}}^{i}=(\times_{j\ne i}K_{\vect{x}}^{j})\times K_{\varepsilon}$ and the spatial coordinates orthogonal to the $i$th spatial dimension $\tilde{\vect{x}}^{i}$, so that as a set $\vect{x}=\{\,x^{i},\tilde{\vect{x}}^{i}\,\}$.  
Finally, we let $\vect{z}=(\varepsilon,\vect{x})$ denote the phase-space coordinate, and define $d\vect{z}=d\varepsilon d\vect{x}$, $d\tilde{\vect{z}}^{i}=d\varepsilon d\tilde{\vect{x}}^{i}$, and let, again as a set, $\tilde{\vect{z}}^{i}=\{\varepsilon,\tilde{\vect{x}}^{i}\}$.  

On each element $\vect{K}$, we let the approximation space for the DG method be
\begin{equation}
  \mathbb{V}_{h}^{k}(\vect{K})
  =\big\{\,\varphi_{h} \colon \varphi_{h}|_{\vect{K}}\in\mathbb{Q}^{k}(\vect{K}), \forall \vect{K}\in\mathcal{T}\,\big\},
  \label{eq:approximationSpace}
\end{equation}
where $\mathbb{Q}^{k}(\vect{K})$ is the phase-space tensor product of one-dimensional polynomials of maximal degree $k$.  
We will denote the approximation space on spatial elements as $\mathbb{V}_{h}^{k}(\vect{K}_{\vect{x}})$, which is defined as in Eq.~\eqref{eq:approximationSpace}, where $\mathbb{Q}^{k}(\vect{K}_{\vect{x}})$ is the spatial tensor product of one-dimensional polynomials of maximal degree $k$.  
We will use $\mathbb{V}_{h}^{k}(\vect{K}_{\vect{x}})$ to approximate the fluid three-velocity $\vect{v}=(v^{1},v^{2},v^{3})$, which will be assumed to be a given function of $\vect{x}$.  

The semi-discrete DG problem is then to find $\vect{\mathcal{U}}_{h}\in\mathbb{V}_{h}^{k}(\vect{K})$, which approximates $\vect{\mathcal{U}}$ in Eq.~\eqref{eq:twoMomentModelCompact}, such that
\begin{equation}
  \big(\,\pd{\vect{\mathcal{U}}_{h}}{t},\varphi_{h}\,\big)_{\vect{K}}
  =\vect{\mathcal{B}}_{h}\big(\,\vect{\mathcal{U}}_{h},\vect{v}_{h},\varphi_{h}\,\big)_{\vect{K}}
  + \big(\,\vect{\mathcal{C}}(\vect{\mathcal{U}}_{h}),\varphi_{h}\,\big)_{\vect{K}},
  \label{eq:dgSemiDiscrete}
\end{equation}
for all test functions $\varphi_{h}\in\mathbb{V}_{h}^{k}(\vect{K})$, $\vect{v}_{h}\in\mathbb{V}_{h}^{k}(\vect{K}_{\vect{x}})$, and all $\vect{K}\in\mathcal{T}$.  
In Eq.~\eqref{eq:dgSemiDiscrete}, we have defined the inner product
\begin{equation}
  \big(\,a_{h},b_{h}\,\big)_{\vect{K}}
  =\int_{\vect{K}}a_{h}\,b_{h}\,\tau\,d\vect{z},
  \quad a_{h},b_{h}\in\mathbb{V}_{h}^{k}(\vect{K})
\end{equation}
and the phase-space advection operator
\begin{equation}
  \vect{\mathcal{B}}_{h}\big(\,\vect{\mathcal{U}}_{h},\vect{v}_{h},\varphi_{h}\,\big)_{\vect{K}}
  =\vect{\mathcal{B}}_{h}^{\vect{x}}\big(\,\vect{\mathcal{U}}_{h},\vect{v}_{h},\varphi_{h}\,\big)_{\vect{K}}
  +\vect{\mathcal{B}}_{h}^{\varepsilon}\big(\,\vect{\mathcal{U}}_{h},\vect{v}_{h},\varphi_{h}\,\big)_{\vect{K}}
  +\big(\,\vect{\mathcal{S}}(\vect{\mathcal{U}}_{h},\vect{v}_{h}),\varphi_{h}\,\big)_{\vect{K}},
  \label{eq:bilinearFormAdvection}
\end{equation}
where the contribution from position space fluxes is
\begin{align}
  \vect{\mathcal{B}}_{h}^{\vect{x}}\big(\,\vect{\mathcal{U}}_{h},\vect{v}_{h},\varphi_{h}\,\big)_{\vect{K}}
  &=-\sum_{i=1}^{d_{\vect{x}}}\int_{\tilde{\vect{K}}^{i}}
  \Big[\,
    \widehat{\vect{\mathcal{F}}^{i}}\big(\vect{\mathcal{U}}_{h},\vect{v}_{h}\big)\,\varphi_{h}|_{x_{\hi}^{i}}
    -\widehat{\vect{\mathcal{F}}^{i}}\big(\vect{\mathcal{U}}_{h},\vect{v}_{h}\big)\,\varphi_{h}|_{x_{\lo}^{i}}
  \,\Big]\,\tau\,d\tilde{\vect{z}}^{i} \nonumber \\
  &\hspace{12pt}
  +\sum_{i=1}^{d_{\vect{x}}}\big(\,\vect{\mathcal{F}}^{i}(\vect{\mathcal{U}}_{h},\vect{v}_{h}),\pd{\varphi_{h}}{i}\,\big)_{\vect{K}}
  \label{eq:bilinearFormAdvectionPosition}
\end{align}
and the contribution from energy space fluxes is
\begin{align}
  \vect{\mathcal{B}}_{h}^{\varepsilon}\big(\,\vect{\mathcal{U}}_{h},\vect{v}_{h},\varphi_{h}\,\big)_{\vect{K}}
  &=-\int_{\vect{K}_{\vect{x}}}
  \Big[\,
    \varepsilon^{3}\,\widehat{\vect{\mathcal{F}}^{\varepsilon}}\big(\vect{\mathcal{U}}_{h},\vect{v}_{h}\big)\,\varphi_{h}|_{\varepsilon_{\hi}}
    -\varepsilon^{3}\,\widehat{\vect{\mathcal{F}}^{\varepsilon}}\big(\vect{\mathcal{U}}_{h},\vect{v}_{h}\big)\,\varphi_{h}|_{\varepsilon_{\lo}}
  \,\Big]\,d\vect{x} \nonumber \\
  &\hspace{12pt}
  +\big(\,\varepsilon\,\vect{\mathcal{F}}^{\varepsilon}(\vect{\mathcal{U}}_{h},\vect{v}_{h}),\pd{\varphi_{h}}{\varepsilon}\,\big)_{\vect{K}}.  
  \label{eq:bilinearFormAdvectionEnergy}
\end{align}

In Eq.~\eqref{eq:bilinearFormAdvectionPosition}, $\widehat{\vect{\mathcal{F}}^{i}}\big(\vect{\mathcal{U}}_{h},\vect{v}_{h}\big)$ is a numerical flux approximating the flux on the surface $\tilde{\vect{K}}^{i}$, which is evaluated using the global Lax--Friedrichs (LF) flux
\begin{equation}
  \widehat{\vect{\mathcal{F}}^{i}}\big(\vect{\mathcal{U}}_{h},\vect{v}_{h}\big)|_{x^{i}}
  =\mathscr{F}_{\LaxF}^{i}\big(\vect{\mathcal{U}}_{h}(x^{i,-},\tilde{\vect{z}}^{i}),\vect{\mathcal{U}}_{h}(x^{i,+},\tilde{\vect{z}}^{i}),\hat{\vect{v}}(x^{i},\tilde{\vect{x}}^{i})\big),
  \label{eq:numericalFluxPosition}
\end{equation}
where $x^{i,\mp}=\lim_{\delta\to0^{+}}x^{i}\mp\delta$ and where we write the global LF flux function as
\begin{equation}
  \mathscr{F}_{\LaxF}^{i}\big(\vect{\mathcal{U}}_{a},\vect{\mathcal{U}}_{b},\hat{\vect{v}}\big)
  =\f{1}{2}\,
  \big(\,
    \vect{\mathcal{F}}^{i}(\vect{\mathcal{U}}_{a},\hat{\vect{v}})+\vect{\mathcal{F}}^{i}(\vect{\mathcal{U}}_{b},\hat{\vect{v}})
    -\alpha^{i}\,(\,\vect{\mathcal{U}}_{b}[\hat{\vect{v}}^{i}]-\vect{\mathcal{U}}_{a}[\hat{\vect{v}}^{i}]\,)
  \,\big),
  \label{eq:globalLF}
\end{equation}
where $\alpha^{i}$ is the largest (absolute) eigenvalue of the flux Jacobian $\partial\vect{\mathcal{F}}^{i}/\partial\vect{\mathcal{U}}$ over the entire domain, for which we simply set $\alpha^{i}=1$.%
\footnote{With this choice, at the expense of potentially increased numerical dissipation when the flux factor is small (see Figure~\ref{fig:WaveSpeed}), computation of flux Jacobian eigenvalues are avoided, and the realizability analysis is simplified.}  
The components of the fluid three-velocity at the element interface is computed as the average
\begin{equation}
  \hat{\vect{v}}(x^{i},\tilde{\vect{x}}^{i})
  =\f{1}{2}\big(\,\vect{v}_{h}(x^{i,-},\tilde{\vect{x}}^{i})+\vect{v}_{h}(x^{i,+},\tilde{\vect{x}}^{i})\,\big).  
  \label{eq:faceVelocity}
\end{equation}
Note that the three-velocity components can be discontinuous across element interfaces.  
\begin{rem}\label{rem:LF_flux}
  In the flux function in Eq.~\eqref{eq:globalLF}, we have defined the dissipative term to be proportional to $(\,\vect{\mathcal{U}}_{b}[\hat{\vect{v}}^{i}]-\vect{\mathcal{U}}_{a}[\hat{\vect{v}}^{i}]\,)$, where $\hat{\vect{v}}^{i}=\big(\,\delta^{i1}\,\hat{v}^{1},\,\delta^{i2}\,\hat{v}^{2},\,\delta^{i3}\,\hat{v}^{3}\,\big)^{\intercal}$, as opposed to the standard LF flux where the dissipative term is proportional to $(\,\vect{\mathcal{U}}_{b}[\hat{\vect{v}}]-\vect{\mathcal{U}}_{a}[\hat{\vect{v}}]\,)$.  
  We have found this to be necessary in order to improve the realizability-preserving property of the scheme in the multi-dimensional setting (see Section~\ref{sec:realizability_preservation}).  
\end{rem}

In order to compute the energy space fluxes $\vect{\mathcal{F}}^{\varepsilon}(\vect{\mathcal{U}}_{h},\vect{v}_{h})$ and the sources $\vect{\mathcal{S}}(\vect{\mathcal{U}}_{h},\vect{v}_{h})$, we need to approximate spatial derivatives of the three-velocity components within elements.  
We denote the derivative of the $i$th velocity component with respect to $x^{j}$ by $(\pd{v^{i}}{j})_{h}\in\mathbb{V}_{h}^{k}(\vect{K}_{\vect{x}})$, and compute this by demanding that
\begin{equation}
  \int_{\vect{K}_{\vect{x}}}(\pd{v^{i}}{j})_{h}\,\varphi_{h}\,d\vect{x}
  =\int_{\tilde{\vect{K}}_{\vect{x}}^{j}}\Big[\,\hat{v}^{i}\,\varphi_{h}|_{x_{\hi}^{j}}-\hat{v}^{i}\,\varphi_{h}|_{x_{\lo}^{j}}\,\Big]\,d\tilde{\vect{x}}^{j}
  -\int_{\vect{K}_{\vect{x}}}v_{h}^{i}\,\pd{\varphi_{h}}{j}\,d\vect{x}
  \label{eq:velocityDerivatives}
\end{equation}
holds for all $\varphi_{h}\in\mathbb{V}_{h}^{k}(\vect{K}_{\vect{x}})$ and all $\vect{K}_{\vect{x}}$, and where $\hat{v}^{i}(x^{j},\tilde{\vect{x}}^{j})$ is computed as in Eq.~\eqref{eq:faceVelocity}.  

The energy space flux $\widehat{\vect{\mathcal{F}}^{\varepsilon}}\big(\vect{\mathcal{U}}_{h},\vect{v}_{h}\big)$ in Eq.~\eqref{eq:bilinearFormAdvectionEnergy} is also computed using an LF-type flux
\begin{equation}
  \widehat{\vect{\mathcal{F}}^{\varepsilon}}\big(\vect{\mathcal{U}}_{h},\vect{v}_{h}\big)|_{\varepsilon}
  =\mathscr{F}_{\LaxF}^{\varepsilon}\big(\vect{\mathcal{U}}_{h}(\varepsilon^{-},\vect{x}),\vect{\mathcal{U}}_{h}(\varepsilon^{+},\vect{x}),\vect{v}_{h}(\vect{x})\big),
  \label{eq:numericalFluxEnergy}
\end{equation}
where $\varepsilon^{\mp}=\lim_{\delta\to0^{+}}\varepsilon\mp\delta$, and we take the LF flux function to be given by
\begin{equation}
  \mathscr{F}_{\LaxF}^{\varepsilon}\big(\vect{\mathcal{U}}_{a},\vect{\mathcal{U}}_{b},\vect{v}_h\big)
  =\f{1}{2}\,
  \big(\,
    \vect{\mathcal{F}}^{\varepsilon}(\vect{\mathcal{U}}_{a},\vect{v}_h)+\vect{\mathcal{F}}^{\varepsilon}(\vect{\mathcal{U}}_{b},\vect{v}_h)
    -\alpha^{\varepsilon}\,(\,\vect{\mathcal{M}}_{b}-\vect{\mathcal{M}}_{a}\,)
  \,\big),
  \label{eq:globalLFenergy}
\end{equation}
where $\alpha^{\varepsilon}$ is an estimate of the largest absolute eigenvalue of the flux Jacobian $\partial\vect{\mathcal{F}}^{\varepsilon}/\partial\vect{\mathcal{U}}$.  
To estimate $\alpha^{\varepsilon}$ we consider the quadratic form
\begin{equation}
  Q(\vect{v}_{h}) =(-\pd{v^{i}}{j})_{h}\,\ell_{i}\,\ell^{j}= \vect{\ell}^{\intercal}A(\vect{v}_{h})\,\vect{\ell},
  \quad\text{where}\quad A_{ij}(\vect{v}_{h})=-\f{1}{2}\big(\,(\pd{v^{j}}{i})_{h}+(\pd{v^{i}}{j})_{h}\big).
  \label{eq:quadraticForm}
\end{equation}
It can be shown that $|Q(\vect{v}_{h})|\le\lambda_{A}$, where $\lambda_{A}$ is the largest absolute eigenvalue of the matrix $A$.  
(Since $A$ is symmetric, the eigenvalues are real.)
Hence, we set $\alpha^{\varepsilon}=\lambda_{A}$.  

\begin{rem}
  In the energy space flux function in Eq.~\eqref{eq:globalLFenergy}, the numerical dissipation term is given in terms of the primitive moments $\vect{\mathcal{M}}$ rather than the conserved moments $\vect{\mathcal{U}}$. 
  This choice is motivated by the realizability analysis in Section~\ref{sec:realizabilityEnergy}.
\end{rem}

\begin{rem}
  For simplicity we assume that the absorption and scattering opacity ($\chi$ and $\sigma$, respectively), appearing in the second term on the right-hand side of Eq.~\eqref{eq:dgSemiDiscrete}, are constant within each phase-space element $\vect{K}$.  
\end{rem}

In this work, we consider the nodal DG scheme (see, e.g., \cite{hesthavenWarburton_2008} for an overview), which writes $\vect{\mathcal{U}}_{h}\in\mathbb{V}_{h}^{k}(\vect{K})$ as an expansion of tensor products of one-dimensional Lagrange polynomials of degrees up to $k$ in each element. 
As in \cite{laiu_etal_2021}, we use the $(k+1)$-point Legendre--Gauss (LG) quadrature points (see, e.g., \cite{abramowitzStegun_1988}) as the interpolation points for the Lagrange polynomials. 
Following the standard practice (i.e., for Ritz--Galerkin), we choose the test functions $\varphi_h$ to be identical to the trial functions, which are the tensor products of Lagrange polynomials used in the expansion of $\vect{\mathcal{U}}_{h}$, and evaluate the inner products $(\cdot,\cdot)_{\vect{K}}$ using the $(k+1)$-point LG quadrature rule.
In the remainder of this paper, we denote the sets of the $(k+1)$-point LG quadrature points in an element $\vect{K}$ on $K_{\varepsilon}$ and $K_{\vect{x}}^{i}$ by $S_\varepsilon^{\vect{K}}:=\{\varepsilon_1,\dots,\varepsilon_{k+1}\}$ and $S_i^{\vect{K}}:=\{x^{i}_1,\dots,x^{i}_{k+1}\}$, respectively. Then the set of local DG nodes in element $\vect{K}$ is denoted as
\begin{equation}\label{eq:localNodes}
	\textstyle
	S^{\vect{K}}_{\otimes} := S_{\varepsilon}^{\vect{K}} \otimes \big( \bigotimes_{i=1}^{d_{\vect{x}}} S_i^{\vect{K}} \big) \:.
\end{equation}
With this notation, the semidiscretized Eq.~\eqref{eq:dgSemiDiscrete} can then be written as 
\begin{equation}\label{eq:nodalSemiDiscrete}
	\pd{\vect{\mathcal{U}}_{\vect{k}}}{t}
	=\vect{\mathsf{B}}\big(\,\vect{\mathcal{U}}_{h},\vect{v}_{h}\,\big)_{\vect{k}}
	+ \vect{\mathsf{C}}(\vect{\mathcal{U}}_{\vect{k}})\:,\quad \forall \vect{K}\in\mathcal{T}\:,
\end{equation}
where $\vect{\mathsf{B}}$ and $\vect{\mathsf{C}}$ denote the advection and collision operators acting on the collection of nodal values $\vect{\mathcal{U}}_{\vect{k}}(t) := \{ \vect{\mathcal{U}}_{h}(\varepsilon, \vect{x}, t)\colon (\varepsilon, \vect{x})\in S^{\vect{K}}_{\otimes} \}$.
Here the subscript ${\vect{k}}$ implies evaluations at points in $S^{\vect{K}}_{\otimes}$.
This nodal representation will become useful in the following sections. 
To simplify the notations therein, we will introduce a few auxiliary point sets in phase-space, which become useful in the realizability analysis in Sections~\ref{sec:realizabilitySpatial} and \ref{sec:realizabilityEnergy}.
In element $\vect{K}$, let $\widehat{S}_{\varepsilon}^{\vect{K}}:=\{\hat{\varepsilon}_1,\dots,\hat{\varepsilon}_{\hat{k}}\}$ and $\widehat{S}_{i}^{\vect{K}}:=\{\hat{x}^{i}_1,\dots,\hat{x}^{i}_{\hat{k}}\}$ denote the sets of quadrature points given by the $\hat{k}$-point Legendre--Gauss--Lobatto (LGL) quadrature rule (see, e.g., \cite{abramowitzStegun_1988}) on $K_{\varepsilon}$ and $K_{\vect{x}}^{i}$, respectively. Here $\hat{k}\geq\frac{k+5}{2}$ is chosen so that the quadrature integrates polynomials up to degree $k+2$ exactly, which is required in the analysis.
In element $\vect{K}$, we define the auxiliary sets $\widehat{S}^{\vect{K}}_{\varepsilon,\otimes}$ and $\widehat{S}^{\vect{K}}_{i,\otimes}$, $i=1,\dots,d_{\vect{x}}$, as
\begin{equation}\label{eq:AuxSets}
	\textstyle
	\widehat{S}^{\vect{K}}_{\varepsilon,\otimes} := \widehat{S}_{\varepsilon}^{\vect{K}} \otimes \big(\bigotimes_{i=1}^{d_{\vect{x}}} S^{\vect{K}}_{i} \big) \quand
	\widehat{S}^{\vect{K}}_{i,\otimes} :=  S_{\varepsilon}^{\vect{K}} \otimes \big(\bigotimes_{j=1,j\neq i}^{d_{\vect{x}}} S^{\vect{K}}_{j} \big) \otimes \widehat{S}^{\vect{K}}_{i},
\end{equation}
respectively. We denote the union of these auxiliary sets in element $\vect{K}$ as
\begin{equation}\label{eq:AuxSetUnion}
	\textstyle
	\widehat{S}^{\vect{K}}_{\otimes} :=
	\widehat{S}^{\vect{K}}_{\varepsilon,\otimes} \cup \big( \bigcup_{i=1}^{d_{\vect{x}}} \widehat{S}^{\vect{K}}_{i,\otimes} \big)
\end{equation}
and further denote the union of the auxiliary sets and the local DG nodes as
\begin{equation}\label{eq:AllSetUnion}
	\textstyle
	\widetilde{S}^{\vect{K}}_{\otimes} :=
	S^{\vect{K}}_{\otimes} \cup \widehat{S}^{\vect{K}}_{\otimes}\:.
\end{equation}
An illustration of the local point sets $S^{\vect{K}}$, $\widehat{S}^{\vect{K}}_{1,\otimes}$, and $\widehat{S}^{\vect{K}}_{\varepsilon,\otimes}$ is given in Figure~\ref{fig:localNodes}, in which the case $d_{\vect{x}}=1$ and $\vect{x}=x^1$ is considered. 
Therefore $\widehat{S}^{\vect{K}}_{i,\otimes}$ is simply $\widehat{S}^{\vect{K}}_{1,\otimes}:={S}^{\vect{K}}_{\varepsilon} \otimes \widehat{S}^{\vect{K}}_{1}$, as defined in Eq.~\eqref{eq:AuxSets}.
\begin{figure}[h]
	\includegraphics[width=\textwidth]{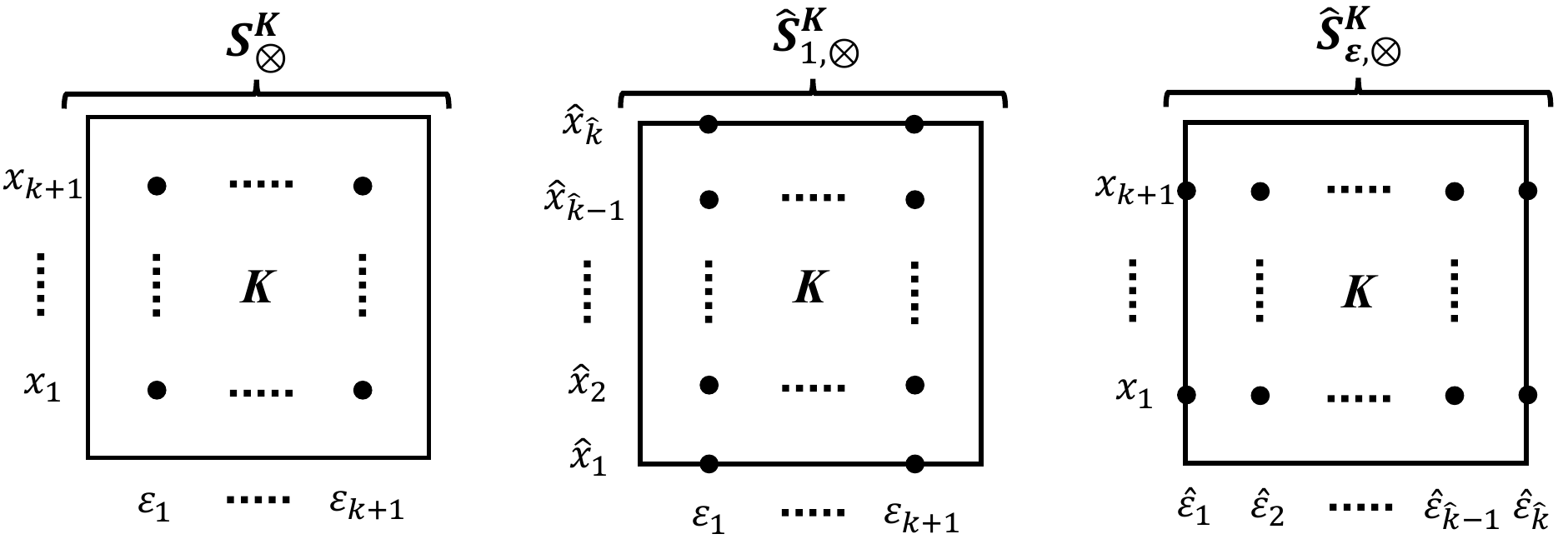}
	\caption{Illustration of the collection of DG nodes $S^{\vect{K}}$ and the auxiliary phase-space point sets $\widehat{S}^{\vect{K}}_{1,\otimes}$ and $\widehat{S}^{\vect{K}}_{\varepsilon,,\otimes}$ in an element $\vect{K}$ in a computational domain $\mathbb{R} \times \mathbb{R}^{+}$. These sets are defined in Eqs.~\eqref{eq:localNodes} and \eqref{eq:AuxSets}, respectively. In this figure, $\widehat{S}^{\vect{K}}_{i,\otimes}$ reduces to $\widehat{S}^{\vect{K}}_{1,\otimes}$ since here $\vect{x}=x^{1}$ is considered.}
	\label{fig:localNodes}
\end{figure}

\subsection{Time Integration}
\label{sec:timeIntegration}

We use IMEX methods to evolve the semi-discrete two-moment model in Eq.~\eqref{eq:dgSemiDiscrete} forward in time, where the phase-space advection term is treated explicitly and the collision term is treated implicitly.  
The general $s$-stage IMEX scheme can then be written as \cite{ascher_etal_1997,pareschiRusso_2005}
\begin{align}
  \big(\,\vect{\mathcal{U}}_{h}^{(i)},\varphi_{h}\,\big)_{\vect{K}}
  &=\big(\,\vect{\mathcal{U}}_{h}^{n},\varphi_{h}\,\big)_{\vect{K}} \nonumber \\
  &\hspace{12pt}\label{eq:s_IMEX_1}
  +\dt\sum_{j=1}^{i-1}\tilde{\alpha}_{ij}\,\vect{\mathcal{B}}_{h}\big(\,\vect{\mathcal{U}}_{h}^{(j)},\vect{v}_{h},\varphi_{h}\,\big)_{\vect{K}}
  +\dt\sum_{j=1}^{i}\alpha_{ij}\,\big(\,\vect{\mathcal{C}}(\vect{\mathcal{U}}_{h}^{(j)}),\varphi_{h}\,\big)_{\vect{K}}, \\
  \big(\,\vect{\mathcal{U}}_{h}^{n+1},\varphi_{h}\,\big)_{\vect{K}}
  &=\big(\,\vect{\mathcal{U}}_{h}^{n},\varphi_{h}\,\big)_{\vect{K}} \nonumber \\
  &\hspace{12pt}\label{eq:s_IMEX_2}
  +\dt\sum_{i=1}^{s}\tilde{w}_{i}\,\vect{\mathcal{B}}_{h}\big(\,\vect{\mathcal{U}}_{h}^{(i)},\vect{v}_{h},\varphi_{h}\,\big)_{\vect{K}}
  +\dt\sum_{i=1}^{s}w_{i}\,\big(\,\vect{\mathcal{C}}(\vect{\mathcal{U}}_{h}^{(i)}),\varphi_{h}\,\big)_{\vect{K}},
\end{align}
for $i=1,\ldots,s$, all $\vect{K}\in\mathcal{T}$, and all $\varphi_{h}\in\mathbb{V}_{h}^{k}(\vect{K})$.  
Here the coefficients $\tilde{\alpha}_{ij}$, $\alpha_{ij}$, $\tilde{w}_{i}$, and $w_{i}$ are required to satisfy certain order conditions for achieving the desired accuracy of the IMEX scheme.  
In addition, to preserve realizability of the evolved moments, each stage in the IMEX update needs to be formulated as convex combinations of realizable terms, which results in additional restrictions on the choices of coefficients. 
We refer the readers to \cite[Section~6]{chu_etal_2019} for details on the order and convex-invariant conditions on the coefficients in the IMEX scheme.

\subsection{Iterative Solvers for Nonlinear Systems}
\label{sec:iterative_solvers}

In this section, we introduce the iterative solvers for the nonlinear systems that occur in the evolution of the IMEX scheme in Eqs.~\eqref{eq:s_IMEX_1}--\eqref{eq:s_IMEX_2}.
In Section~\ref{sec:moment_conversion}, we present the iterative solver for the conversion of conserved moments $\bcU$ to primitive moments $\bcM$. This moment conversion is required to evaluate the closures for the higher-order moments $\mathcal{K}^{ij}$ and  $\mathcal{Q}^{ijk}$ at each stage of the IMEX scheme, since these closures are defined in terms of the primitive moments as discussed in Section~\ref{sec:closure}.
In Section~\ref{sec:collision_solver}, we discuss the solver for the nonlinear equations arising from the implicit update in the IMEX scheme. Even though the simplified collision term $\bcC(\bcU)$ in Eq.~\eqref{eq:sources} appears to be linear in terms of the primitive moments, the implicit system is still nonlinear because the IMEX scheme evolves the conserved moments. This nonlinear system formulation is also extendable to handle systems with collision terms that include more comprehensive physics; e.g., neutrino–electron scattering and thermal pair processes, as considered in \cite{laiu_etal_2021}. 

Under the nodal DG framework (see Eq.~\eqref{eq:nodalSemiDiscrete}), each of these nonlinear systems can be formulated locally at each node in the phase-space element because there is no coupling between nodes in either the moment conversion or the collision solve. Therefore, the nonlinear systems considered in this section are written in terms of the nodal moments at a given phase-space node $\vect{z}\in S^{\vect{K}}_{\otimes}$, $\forall \vect{K}\in\cT$, where $S^{\vect{K}}_{\otimes}$ is the set of DG nodes in element $\vect{K}$, as defined in Eq.~\eqref{eq:localNodes}.
For convenience, we drop the subscript from the nodal representation in this section, and note that, such nonlinear systems must be solved at each $\vect{z}\in S^{\vect{K}}_{\otimes}$ and in each $\vect{K}\in\cT$ to perform the moment conversion from $\vect{\mathcal{U}}$ to $\vect{\mathcal{M}}$ or the implicit steps in the IMEX scheme.

\subsubsection{Moment Conversion Solver}
\label{sec:moment_conversion}

For a given conserved moment $\bcU\in\cR$, finding a corresponding primitive moment $\bcM\in\cR$ that satisfies Eq.~\eqref{eq:ConservedToPrimitive} requires solving a nonlinear system. 
A naive approach is to formulate Eq.~\eqref{eq:ConservedToPrimitive} as a fixed-point problem
\begin{equation}\label{eq:fixed_pt}
	\bcM = \left(\begin{array}{c} \cD  \\ \cI_{j} \end{array} \right)	
	= \left(\begin{array}{c}
		- v^{i}\cI_{i} +  \cN\\
		- v^{i}\mathsf{k}_{ij} \cD +  \cG_{j} 
	\end{array}\right):=\tilde{\bcH}_{\bcU}(\bcM) .
\end{equation}
However, when standard fixed-point iteration, i.e., Picard iteration (see, e.g., \cite[Section~I.8]{hairer1993solving}), is applied to solve Eq.~\eqref{eq:fixed_pt}, this formulation does not guarantee that the resulting moments are realizable at each iteration, which, in turn, may result in failures to convergence on problems in this form. 
To address these issues, we adopt the idea from Richardson iteration, see, e.g., \cite{richardson1911ix} and \cite[Section~13.2.1]{saad2003iterative}, for solving linear systems and reformulate the fixed-point problem in Eq.~\eqref{eq:fixed_pt} as
\begin{equation}\label{eq:richardson_fixed_pt}
	\bcM = \left(\begin{array}{c} \cD  \\ \cI_{j} \end{array} \right)	
	= \left(\begin{array}{c} \cD \\ \cI_j  \end{array} \right)	 
	- \lambda \left(\begin{array}{c}
		\cD + v^{i}\cI_{i} -  \cN\\
		\cI_{j} + v^{i}\mathsf{k}_{ij} \cD -  \cG_{j} 
	\end{array}\right):=\bcH_{\bcU}(\bcM) ,
\end{equation}
where $\lambda\in(0,1]$ is a constant. Here we choose $\lambda = (1+v)^{-1}$, where $v := |\vect{v}| = \sqrt{v_{i}v^{i}}\,$, to guarantee the realizability-preserving and convergence properties of the Picard iteration method; i.e., 
\begin{equation}\label{eq:Picard}
	\bcM^{[k+1]} = \bcH_{\bcU}(\bcM^{[k]}) .
\end{equation}
The realizability-preserving and convergence properties of Eq.~\eqref{eq:Picard} are stated and proved in Section~\ref{sec:momentConversionRealizability}.

\subsubsection{Collision Solver}
\label{sec:collision_solver}

The implicit steps in Eq.~\eqref{eq:s_IMEX_1} require solving nonlinear systems to find the updated conserved moments. Similar to the implicit systems considered in \cite{laiu_etal_2021}, these systems take the form 
\begin{equation}\label{eq:ImplicitSystem}
	\bcU = \bcU^{\old} + \dt_{\bcC} \, \bcC(\bcU)\:,
\end{equation}
where $\bcU^{\old}$ denotes the known conserved moments from the explicit steps, $\bcU$ denotes the unknown updated conserved moments driven by the implicit collision term $\bcC(\bcU)$ defined in Eq.~\eqref{eq:sources}, and $\dt_{\bcC}$ denotes the effective time step size for the implicit system.
Since the sources are expressed in terms of primitive moments, we solve Eq.~\eqref{eq:ImplicitSystem} as a nonlinear fixed-point problem on the unknown primitive moments and use the primitive moment solution to compute the collision term $\bcC(\bcU)$, which is then used to update the conserved moments $\bcU$ in Eq.~\eqref{eq:ImplicitSystem}.
As in the moment conversion case discussed in Section~\ref{sec:moment_conversion}, we apply the idea from Richardson iteration to Eq.~\eqref{eq:ImplicitSystem} and formulate it as a fixed-point problem in terms of the primitive moments; i.e., 
\begin{equation}\label{eq:collision_fixed_point0}
	\bcM = \left(\begin{array}{c} \cD \\ \cI_j \end{array}\right) 
	= \left(\begin{array}{c} \cD \\ \cI_j \end{array}\right)  - \lambda \,\left(\begin{array}{c} \cD + v^{i}\cI_{i} - \cN^{\old} - \dt_{\bcC} \chi\,\big(\,\cD_{0}-\cD\,\big) \\ 
		\cI_j  + v^{i}\mathsf{k}_{ij} \cD - \cG_j^{\old} + \dt_{\bcC} \kappa\,\cI_{j}   \end{array}\right) 
	=: \tilde{\bcQ}(\bcM),
\end{equation}
where $\cN^{\old}$ and $\bcG^{\old}$ denote the number density and number flux components of the given conserved moment $\bcU^{\old}$, respectively, and the constant $\lambda\in(0,1]$.
Although, this formulation is consistent with the one considered in Section~\ref{sec:moment_conversion} when there are no collisions ($\chi = \kappa = 0$), it cannot guarantee that the realizability of moments is preserved when collisions are taken into account.
To address this issue, we follow the approach taken in \cite{laiu_etal_2021} and reformulate the fixed-point problem as
\begin{equation}\label{eq:collision_fixed_point1}
	\bcM = \left(\begin{array}{c} \cD \\ \cI_j \end{array}\right) 
	=  \Lambda\left(\begin{array}{c} (1-\lambda) \cD  +\lambda (- v^{i}\cI_{i} + \cN^{\old} + \dt_{\bcC} \chi\,\cD_{0})\\ 
		(1-\lambda) \cI_j   + \lambda (- v^{i}\mathsf{k}_{ij} \cD+\cG_j^{\old})  \end{array}\right) 
	=: \bcQ(\bcM)\:,
\end{equation}
where $\Lambda:=\textup{diag}(\mu_{\chi},\mu_{\kappa})$ with $\mu_{\chi} = (1+\lambda \, \dt_{\bcC} \, \chi)^{-1}$ and $\mu_{\kappa} = (1+ \lambda \, \dt_{\bcC} \, \kappa)^{-1}$.
Applying Picard iteration to this fixed-point problem then leads to the iterative scheme
\begin{equation}\label{eq:collision_Picard}
	\bcM^{[k+1]} = \bcQ(\bcM^{[k]})\:.
\end{equation}
In Section~\ref{sec:collision}, we prove the realizability-preserving and convergence properties of this iterative solver with $\lambda = (1+v)^{-1}$.

\section{Realizability-Preserving Property}
\label{sec:realizability_preservation}

In this section, we show that, by imposing a proper time-step restriction and a realizability-enforcing limiter, the DG scheme with IMEX time integration given in Section~\ref{sec:discretization} preserves the realizability of both conserved and primitive moments. 
To this end, we focus on the analysis of a forward-backward Euler method for its simplicity. The theoretical results can be extended to more general IMEX methods that are strong stability-preserving (SSP) with the size of time steps dependent only on the explicit part, such as the IMEX scheme implemented in the numerical tests reported in Section~\ref{sec:numericalResults}.
Specifically, we analyze the realizability-preserving property of the following numerical scheme.
\begin{subequations}
\begin{align}
	\big(\,\widehat{\vect{\mathcal{U}}}_{h}^{n+\sfrac{1}{2}},\varphi_{h}\,\big)_{\vect{K}}
	&=\big(\,{\vect{\mathcal{U}}}_{h}^{n},\varphi_{h}\,\big)_{\vect{K}} + \dt\,\vect{\mathcal{B}}_{h}\big(\,{\vect{\mathcal{U}}}_{h}^{n},\vect{v}_{h},\varphi_{h}\,\big)_{\vect{K}},
	\label{eq:imexForwardEuler}\\
	{\vect{\mathcal{U}}}_{h}^{n+\sfrac{1}{2}} 
	&=\, \texttt{RealizabilityLimiter}\,(\,\widehat{\vect{\mathcal{U}}}_{h}^{n+\sfrac{1}{2}}\,),
  \label{eq:realizabilityLimiter}\\
	\big(\,\widehat{\vect{\mathcal{U}}}_{h}^{n+1},\varphi_{h}\,\big)_{\vect{K}}
	&=\big(\,{\vect{\mathcal{U}}}_{h}^{n+\sfrac{1}{2}},\varphi_{h}\,\big)_{\vect{K}} + \dt\,\big(\,\vect{\mathcal{C}}(\widehat{\vect{\mathcal{U}}}_{h}^{n+1}),\varphi_{h}\,\big)_{\vect{K}},
	\label{eq:imexBackwardEuler}\\
	{\vect{\mathcal{U}}}_{h}^{n+1} 
	&=\, \texttt{RealizabilityLimiter}\,(\,\widehat{\vect{\mathcal{U}}}_{h}^{n+1}\,).
	\label{eq:realizabilityLimiter2}
\end{align}
\end{subequations}
Here, $\texttt{RealizabilityLimiter}()$ denotes the realizability-enforcing limiter proposed in \cite{chu_etal_2019}, the details of which is given in Section~\ref{sec:realizabilityLimiter} for completeness.

Loosely speaking, the realizability-preserving property of the scheme \eqref{eq:imexForwardEuler}--\eqref{eq:realizabilityLimiter2} requires that, if the current moments $\vect{\mathcal{U}}_{h}^{n}$ are realizable, then the updated moments $\vect{\mathcal{U}}_{h}^{n+1}$ remain realizable. 
In the following paragraphs, we summarize the realizability-preserving properties proved in this section, where more detailed realizability results and conditions are described using the sets of phase-space points defined in Section~\ref{sec:dgMethod}.

A key assumption in the realizability analysis is the exact closure assumption.
\begin{assumption}[Exact closures]\label{assum:exact_closure}
	The moment closures for closing the higher order moments $\mathcal{K}^{ij}$ and $\mathcal{Q}^{ijk}$ in Eq.~\eqref{eq:imexForwardEuler} are exact, i.e., given lower order primitive moments $(\mathcal{D},\,\mathcal{I}^{i})$, the moments $(\mathcal{K}^{ij},\,\mathcal{Q}^{ijk})$ are computed such that $(\mathcal{D},\,\mathcal{I}^{i},\,\mathcal{K}^{ij},\,\mathcal{Q}^{ijk})$ satisfy Eq.~\eqref{eq:angularMoments} for some nonnegative distribution $f$.
\end{assumption}
We note that Assumption~\ref{assum:exact_closure} holds when the exact Minerbo closure is used, i.e., when the Eddington and heat-flux factors are given in Eq.~\eqref{eq:psiZetaMinerbo} (as opposed to the approximation given in Eqs.~\eqref{eq:psiApproximate}--\eqref{eq:zetaApproximate}). 
Evaluating (either the exact or approximate) Eddington factor and heat-flux factor uses the flux factor $h\,(=\mathcal{I}/\mathcal{D})$ of the primitive moments $\vect{\mathcal{M}}=\big(\,\mathcal{D},\,\vect{\mathcal{I}}\,\big)^{\intercal}$.
Since the numerical scheme Eqs.~\eqref{eq:imexForwardEuler}--\eqref{eq:realizabilityLimiter2} evolves the conserved moments $\vect{\mathcal{U}}$, evaluating moment closures requires the conversion between conserved and primitive moments.
In other words, given $\vect{\mathcal{U}}$ (or $\vect{\mathcal{M}}$), the solver needs to compute the associated $\vect{\mathcal{M}}$ (or $\vect{\mathcal{U}}$) that satisfies Eq.~\eqref{eq:ConservedToPrimitive}.

Under Assumption~\ref{assum:exact_closure}, we state the main theoretical result of the realizability-preserving analysis for the scheme Eqs.~\eqref{eq:imexForwardEuler}--\eqref{eq:realizabilityLimiter2} in Theorem~\ref{thm:realizability}, where this scheme is shown to preserve realizability of moments $\vect{\mathcal{U}}_{h}$ on the point set $\widetilde{S}^{\vect{K}}_{\otimes}$ defined in Eq.~\eqref{eq:AllSetUnion}, for all elements $\vect{K}\in\cT$. 

\begin{theorem}[Realizability preservation]\label{thm:realizability}
	Suppose (i) Assumption~\ref{assum:exact_closure} holds, (ii) $v_h := |\vect{v}_h|<1$ for all $\vect{K}\in\cT$,  and (iii) the time step $\dt$ in Eq.~\eqref{eq:imexForwardEuler} satisfies the hyperbolic-type time-step restriction
	\begin{equation}\label{eq:timestepRestriction}
		\begin{alignedat}{2}
		\dt \leq \min\big\{&\dt_{\vect{x}}^{\min}\, ,  \,\dt_{\varepsilon}^{\min}\big\}\:,\text{\,with\,}\\
		\dt_{\vect{x}}^{\min}:= \min_{\vect{K}\in\cT}\, \min_i (1-v_h)\, {C^i} |K_{\vect{x}}^i|&\quand
		\dt_{\varepsilon}^{\min}:=\min_{\vect{K}\in\cT}\,(1-v_h)\,{C^\varepsilon} |K_{\varepsilon}|/{\varepsilon_{\hi}}
		\end{alignedat}
	\end{equation}
	where $C^i$ and $C^\varepsilon$, which are independent of the size of elements in the discretization, are given in Eqs.~\eqref{eq:CFLSpatial} and \eqref{eq:CFLEnergy}, respectively.
	Then the scheme \eqref{eq:imexForwardEuler}--\eqref{eq:realizabilityLimiter2} is realizability-preserving, i.e., $\vect{\mathcal{U}}_{h}^{n+1}\in\cR$ on $\widetilde{S}^{\vect{K}}_{\otimes}$, $\forall \vect{K}\in\cT$, provided that $\vect{\mathcal{U}}_{h}^{n}\in\cR$ on $\widetilde{S}^{\vect{K}}_{\otimes}$, $\forall \vect{K}\in\cT$. 
\end{theorem}
Theorem~\ref{thm:realizability} is a direct consequence of the following Propositions~\ref{prop:realizabilityExplicit},  \ref{prop:realizabilityLimiter}, \ref{prop:realizabilityMomentConversion}, and \ref{prop:realizabilityImplicit}, which provide the realizability-preserving properties of the explicit update (Eq.~\eqref{eq:imexForwardEuler}), the realizability-enforcing limiter (Eq.~\eqref{eq:realizabilityLimiter}), the moment conversion (Eq.~\eqref{eq:ConservedToPrimitive}), and the implicit update (Eq.~\eqref{eq:imexBackwardEuler}), respectively.
In these propositions, the notion of cell-averaged moments will come in handy.
Given $\vect{\mathcal{U}}_{h}^{n}$, the cell-averaged moments $\vect{\mathcal{U}}_{\vect{K}}:=(\cN_{\vect{K}},\vect{\mathcal{G}}_{\vect{K}})$ are defined as
\begin{equation}
	\vect{\mathcal{U}}_{\vect{K}} = \big(\,\vect{\mathcal{U}}_{h}\,\big)_{\vect{K}}/|\vect{K}|.  
	\label{eq:cellAverage}
\end{equation}

\begin{prop}[Explicit advection update]\label{prop:realizabilityExplicit}
	Suppose (i) Assumption~\ref{assum:exact_closure} holds, (ii) $v_h <1$ for all $\vect{K}\in\cT$, and (iii) $\dt$ in Eq.~\eqref{eq:imexForwardEuler} satisfies the restriction \eqref{eq:timestepRestriction}.
	Let $\widehat{\vect{\mathcal{U}}}_{\vect{K}}^{n+\sfrac{1}{2}}:=(\widehat{\cN}_{\vect{K}}^{n+\sfrac{1}{2}},\widehat{\vect{\mathcal{G}}}_{\vect{K}}^{n+\sfrac{1}{2}})$ denote the element average of the moment $\widehat{\vect{\mathcal{U}}}_{h}^{n+\sfrac{1}{2}}$ (as defined in Eq.~\eqref{eq:cellAverage}) updated by Eq.~\eqref{eq:imexForwardEuler} from $\vect{\mathcal{U}}_{h}^{n}$.
	Then, it is guaranteed that, $\forall \vect{K}\in\cT$, $\widehat{\cN}_{\vect{K}}^{n+\sfrac{1}{2}}>0$, provided
	$\vect{\mathcal{U}}_{h}^{n}\in\cR$ on $S^{\vect{K}}_{\otimes}$, $\forall \vect{K}\in\cT$.
	Further, when a reduced one-dimensional planar geometry%
	\footnote{An example of this one-dimensional geometry is the reduced case of the full three-dimensional geometry when the fluxes in two of the three spatial dimensions are assumed to be zero. See Section~\ref{sec:realizabilitySource} for further discussions.} 
	is considered, it is guaranteed that, $\forall \vect{K}\in\cT$, $\widehat{\vect{\mathcal{U}}}_{\vect{K}}^{n+\sfrac{1}{2}}\in\cR$ when an additional time-step restriction \eqref{eq:sourceTimeStepRestriction} is satisfied.
\end{prop}

\begin{prop}[Realizability-enforcing limiter]\label{prop:realizabilityLimiter}
	Suppose $\widehat{\cN}_{\vect{K}}>0$ on element $\vect{K}$, applying the realizability-enforcing limiter given in Algorithm~\ref{algo:realizabilityLimiter} (see Section~\ref{sec:realizabilityLimiter}) to the moments $\widehat{\vect{\mathcal{U}}}_{h}$ on $\vect{K}$ leads to realizable moments $\vect{\mathcal{U}}_{h}\in\cR$ on $S^{\vect{K}}_{\otimes} \cup \widehat{S}^{\vect{K}}_{\otimes}$ in $\vect{K}$.
\end{prop}

\begin{prop}[Moment conversion]\label{prop:realizabilityMomentConversion}
	Suppose that Assumption~\ref{assum:exact_closure} holds and that $v_h<1$ on all $\vect{K}\in\cT$, then
	the conversion between conserved and primitive moments following the relation in Eq.~\eqref{eq:ConservedToPrimitive} preserves realizability, i.e., for a pair of conserved and primitive moments $(\vect{\mathcal{U}},\vect{\mathcal{M}})$ satisfying Eq.~\eqref{eq:ConservedToPrimitive}, $\vect{\mathcal{U}}\in\cR$ if and only if $\vect{\mathcal{M}}\in\cR$.
	Further, given $\vect{\mathcal{U}}\in\cR$, the iterative solver \eqref{eq:Picard} in Section~\ref{sec:moment_conversion} converges to the unique $\vect{\mathcal{M}}\in\cR$ that satisfies Eq.~\eqref{eq:ConservedToPrimitive}.
\end{prop}

\begin{prop}[Implicit collision solve]\label{prop:realizabilityImplicit}
	Suppose Assumption~\ref{assum:exact_closure} holds and $v_h<1$ on all $\vect{K}\in\cT$. Let $\vect{\mathcal{U}}_{h}^{n+\sfrac{1}{2}}\in\cR$ on $S^{\vect{K}}_{\otimes}$ for all $\vect{K} \in \cT$, then solving  Eq.~\eqref{eq:imexBackwardEuler} with the iterative solvers considered in  Section~\ref{sec:iterative_solvers} gives $\widehat{\vect{\mathcal{U}}}_{h}^{n+1}\in\cR$ on $S^{\vect{K}}_{\otimes}$ for all $\vect{K} \in \cT$.
\end{prop}

These propositions form a basis for the proof of Theorem~\ref{thm:realizability}.
Specifically, Proposition~\ref{prop:realizabilityExplicit} guarantees that the updated moments $\widehat{\vect{\mathcal{U}}}_{h}^{n+\sfrac{1}{2}}$ from Eq.~\eqref{eq:imexForwardEuler} have a nonnegative cell-averaged density $\widehat{\mathcal{N}}_{\vect{K}}^{n+\sfrac{1}{2}}$ for each $\vect{K}\in\cT$. It follows from Proposition~\ref{prop:realizabilityLimiter} that the limited moments $\vect{\mathcal{U}}_{h}^{n+\sfrac{1}{2}}$ are realizable on ${S}^{\vect{K}}_{\otimes}$ for all $\vect{K}\in\cT$.
Solving Eq.~\eqref{eq:imexBackwardEuler} on each nodal point in ${S}^{\vect{K}}_{\otimes}$ for all $\vect{K}\in\cT$ gives the updated moment $\widehat{\vect{\mathcal{U}}}_{h}^{n+1}$, which is guaranteed to be realizable on  ${S}^{\vect{K}}_{\otimes}$, $\forall \vect{K}\in\cT$, by Proposition~\ref{prop:realizabilityImplicit}.
Applying the realizability-enforcing limiter again to $\widehat{\vect{\mathcal{U}}}_{h}^{n+1}$ on every $\vect{K}\in\cT$ leads to $\vect{\mathcal{U}}_{h}^{n+1}$, which is realizable on $\widehat{S}^{\vect{K}}_{\otimes}$, $\forall \vect{K}\in\cT$, again from Proposition~\ref{prop:realizabilityLimiter}.

In Sections~\ref{sec:streaming}, \ref{sec:realizabilityLimiter}, \ref{sec:momentConversionRealizability}, and \ref{sec:collision}, we prove Propositions~\ref{prop:realizabilityExplicit},  \ref{prop:realizabilityLimiter}, \ref{prop:realizabilityMomentConversion}, and \ref{prop:realizabilityImplicit}, respectively. 
These results together lead to the main realizability-preserving property of the numerical scheme Eqs.~\eqref{eq:imexForwardEuler}--\eqref{eq:imexBackwardEuler} given in Theorem~\ref{thm:realizability}, under the exact closure assumption, Assumption~\ref{assum:exact_closure}.
In Section~\ref{sec:approxClosure}, we extend the realizability-preserving and convergence results in Propositions~\ref{prop:realizabilityMomentConversion} and \ref{prop:realizabilityImplicit} to the case of evaluating the closure with the approximate Eddington factor $\psi_{\mathsf{a}}$ in Eq.~\eqref{eq:psiApproximate}, which is often used in practice to reduce computational cost.

\subsection{Explicit Advection Update}
\label{sec:streaming}

In this section, we prove Proposition~\ref{prop:realizabilityExplicit} by deriving the time-step restriction \eqref{eq:timestepRestriction} under which the updated cell-averaged number density  $\widehat{\mathcal{N}}_{\vect{K}}^{n+\sfrac{1}{2}}>0$.  In a one-dimensional planar geometry, we show that $\widehat{\vect{\mathcal{U}}}_{\vect{K}}^{n+\sfrac{1}{2}}\in\cR$ under an additional time-step restriction given in Eq.~\eqref{eq:sourceTimeStepRestriction}.

Since constant functions are in the approximation space $\mathbb{V}_{h}^{k}(\vect{K})$, we start with deriving the update formula for cell-averaged moments by setting $\varphi_{h}=1$ in Eq.~\eqref{eq:imexForwardEuler}, which leads to 
\begin{align}
  \widehat{\vect{\mathcal{U}}}_{\vect{K}}^{n+\sfrac{1}{2}}
  &=\vect{\mathcal{U}}_{\vect{K}}^{n}+\f{\dt}{|\vect{K}|}\,\vect{\mathcal{B}}_{h}\big(\,\vect{\mathcal{U}}_{h}^{n},\vect{v}_{h}\,\big)_{\vect{K}} \nonumber \\
  &=\gamma^{\vect{x}}\,\Big\{\,\vect{\mathcal{U}}_{\vect{K}}^{n}+\f{\dt}{\gamma^{\vect{x}}\,|\vect{K}|}\,\vect{\mathcal{B}}_{h}^{\vect{x}}\big(\,\vect{\mathcal{U}}_{h}^{n},\vect{v}_{h}\,\big)_{\vect{K}}\,\Big\}
  +\gamma^{\varepsilon}\,\Big\{\,\vect{\mathcal{U}}_{\vect{K}}^{n}+\f{\dt}{\gamma^{\varepsilon}\,|\vect{K}|}\,\vect{\mathcal{B}}_{h}^{\varepsilon}\big(\,\vect{\mathcal{U}}_{h}^{n},\vect{v}_{h}\,\big)_{\vect{K}}\,\Big\} \nonumber \\
  &\hspace{12pt}
  +\gamma^{\mathcal{S}}\,\Big\{\,\vect{\mathcal{U}}_{\vect{K}}^{n}+\f{\dt}{\gamma^{\mathcal{S}}\,|\vect{K}|}\,\big(\,\vect{\mathcal{S}}(\vect{\mathcal{U}}_{h}^{n},\vect{v}_{h})\,\big)_{\vect{K}}\,\Big\},   \label{eq:fullCellAverageUpdateSplit}  \\
  &=:\gamma^{\vect{x}}\,\widehat{\vect{\mathcal{U}}}_{\vect{K}}^{n+\sfrac{1}{2},\,\vect{x}} +   \gamma^{\varepsilon}\,\widehat{\vect{\mathcal{U}}}_{\vect{K}}^{n+\sfrac{1}{2},\,\varepsilon} +   \gamma^{\mathcal{S}}\,\widehat{\vect{\mathcal{U}}}_{\vect{K}}^{n+\sfrac{1}{2},\,\mathcal{S}}\:,\nonumber
\end{align}
where we have defined $\gamma^{\vect{x}},\gamma^{\varepsilon},\gamma^{\mathcal{S}}>0$, satisfying $\gamma^{\vect{x}}+\gamma^{\varepsilon}+\gamma^{\mathcal{S}}=1$.  
In the following subsections, we show that, when $\vect{\mathcal{U}}_h^n\in\cR$ on $\widehat{S}^{\vect{K}}_{\otimes}$ for all $\vect{K}\in\cT$, $\widehat{\vect{\mathcal{U}}}_{\vect{K}}^{n+\sfrac{1}{2},\vect{x}}$ and $\widehat{\vect{\mathcal{U}}}_{\vect{K}}^{n+\sfrac{1}{2},\varepsilon}$ are realizable under time-step restrictions given in Eq.~\eqref{eq:timestepRestriction} (Sections~\ref{sec:realizabilitySpatial} and \ref{sec:realizabilityEnergy}) and that $\widehat{\mathcal{N}}_{\vect{K}}^{n+\sfrac{1}{2},\mathcal{S}}>0$ (Section~\ref{sec:realizabilitySource})  for all $\vect{K}\in\cT$. 
Further, we show in Section~\ref{sec:realizabilitySource} that $\widehat{\vect{\mathcal{U}}}_{\vect{K}}^{n+\sfrac{1}{2},\mathcal{S}}$ is realizable in one-dimensional, planar geometry under an additional time-step restriction given in Eq.~\eqref{eq:sourceTimeStepRestriction}. 
Since the realizable set $\mathcal{R}$ is convex and $\widehat{\vect{\mathcal{U}}}_{\vect{K}}^{n+\sfrac{1}{2}}$ is written as a convex combination of $\widehat{\vect{\mathcal{U}}}_{\vect{K}}^{n+\sfrac{1}{2},\vect{x}}$, $\widehat{\vect{\mathcal{U}}}_{\vect{K}}^{n+\sfrac{1}{2},\varepsilon}$, and $\widehat{\vect{\mathcal{U}}}_{\vect{K}}^{n+\sfrac{1}{2},\mathcal{S}}$ in Eq.~\eqref{eq:fullCellAverageUpdateSplit}, we thus conclude that, under the time-step restrictions in Eqs.~\eqref{eq:timestepRestriction} and \eqref{eq:sourceTimeStepRestriction},  (i) $\widehat{\mathcal{N}}_{\vect{K}}^{n+\sfrac{1}{2}}>0$ and (ii)  $\widehat{\vect{\mathcal{U}}}_{\vect{K}}^{n+\sfrac{1}{2}}\in\cR$ in a planar geometry.

\subsubsection{Position Space Fluxes}
\label{sec:realizabilitySpatial}

For transport in position space we follow the approach in \cite{chu_etal_2019} and write
\begin{equation}
  \widehat{\vect{\mathcal{U}}}_{\vect{K}}^{n+\sfrac{1}{2},\vect{x}}
  =\vect{\mathcal{U}}_{\vect{K}}^{n}+\f{\dt_{\vect{x}}}{|\vect{K}|}\,\vect{\mathcal{B}}_{h}^{\vect{x}}\big(\,\vect{\mathcal{U}}_{h}^{n},\vect{v}_{h}\,\big)_{\vect{K}}
  \quad(\dt_{\vect{x}}=\dt/\gamma^{\vect{x}}).
\end{equation}
To find sufficient conditions such that $\widehat{\vect{\mathcal{U}}}_{\vect{K}}^{n+\sfrac{1}{2},\vect{x}}\in\mathcal{R}$, we define (cf.~\cite{chu_etal_2019})
\begin{equation}
  \Gamma^{i}\big[\vect{\mathcal{U}}_{h}^{n}\big](\tilde{\vect{z}}^{i})
  =\f{1}{|K_{\vect{x}}^{i}|}
  \Big[\,
    \int_{K_{\vect{x}}^{i}}\vect{\mathcal{U}}_{h}^{n}\,dx^{i} 
    - \f{\dt_{\vect{x}}}{\beta^{i}}
    \Big(\,
      \widehat{\vect{\mathcal{F}}^{i}}\big(\vect{\mathcal{U}}_{h}^{n},\vect{v}_{h}\big)|_{x_{\hi}^{i}}
      -\widehat{\vect{\mathcal{F}}^{i}}\big(\vect{\mathcal{U}}_{h}^{n},\vect{v}_{h}\big)|_{x_{\lo}^{i}}
    \,\Big)
  \,\Big],
  \label{eq:GammaSpatial}
\end{equation}
so that
\begin{equation}
  \widehat{\vect{\mathcal{U}}}_{\vect{K}}^{n+\sfrac{1}{2},\vect{x}}
  =\sum_{i=1}^{d_{\vect{x}}}\f{\beta^{i}}{|\tilde{\vect{K}}^{i}|}
  \int_{\tilde{\vect{K}}^{i}}\Gamma^{i}\big[\vect{\mathcal{U}}_{h}^{n}\big](\tilde{\vect{z}}^{i})\,\tau\,d\tilde{\vect{z}}^{i},
  \label{eq:cellAverageInTermsOfGammaSpatial}
\end{equation}
where we have defined the set of positive constants $\{\beta^{i}\}_{i=1}^{d_{\vect{x}}}$ satisfying $\sum_{i=1}^{d_{\vect{x}}}\beta^{i}=1$.  

If a quadrature rule $\tilde{\vect{\mathcal{Q}}}^{i}:C^{0}(\tilde{\vect{K}}^{i})\to\mathbb{R}$ with positive weights, e.g., the tensor product of one-dimensional LG quadrature, is used to approximate the integral in Eq.~\eqref{eq:cellAverageInTermsOfGammaSpatial}, it is sufficient to show that, under the assumptions in Proposition~\ref{prop:realizabilityExplicit}, $\Gamma^{i}\big[\vect{\mathcal{U}}_{h}^{n}\big](\tilde{\vect{z}}^{i})\in\mathcal{R}$ holds for $\tilde{\vect{z}}^{i}\in\tilde{\vect{\mathcal{S}}}^{i}\subset\tilde{\vect{K}}^{i}$, where $\tilde{\vect{\mathcal{S}}}^{i}$ denotes the set of quadrature points given by $\tilde{\vect{\mathcal{Q}}}^{i}$. 
We prove this sufficient condition in the remainder of this subsection.

Let $\hat{Q}^{i}:C^{0}(K_{\vect{x}}^{i})\to\mathbb{R}$ denote the $\hat{k}$-point LGL quadrature rule on $K_{\vect{x}}^{i}$ with points $\widehat{S}^{\vect{K}}_{i}=\{x_{\lo}^{i}=\hat{x}_{1}^{i},\ldots,\hat{x}_{\hat{k}}^{i}=x_{\hi}^{i}\}$ as defined in Section~\ref{sec:dgMethod} and strictly positive weights $\{\hat{w}_{q}\}_{q=1}^{\hat{k}}$, normalized such that $\sum_{q=1}^{\hat{k}}\hat{w}_{q}=1$.  
Since $\hat{k}\geq \frac{k+5}{2}$, this quadrature integrates $\vect{\mathcal{U}}_{h}^{n}$ exactly, and thus we have
\begin{equation}
  \int_{\vect{K}_{\vect{x}}^{i}}\vect{\mathcal{U}}_{h}^{n}(x^{i})\,dx^{i} 
  = \hat{Q}^{i}\big[\vect{\mathcal{U}}_{h}^{n}\big]
  = |K_{\vect{x}}^{i}|\,\sum_{q=1}^{\hat{k}}\hat{w}_{q}\,\vect{\mathcal{U}}_{h}^{n}(\hat{x}_{q}^{i}),
  \label{eq:glQuadratureSpace}
\end{equation}
where, for notational convenience, we have suppressed explicit dependence on $\tilde{\vect{z}}^{i}$ in writing $\vect{\mathcal{U}}_{h}^{n}(\hat{x}_{q}^{i},\tilde{\vect{z}}^{i})=\vect{\mathcal{U}}_{h}^{n}(\hat{x}_{q}^{i})$.  
Similarly, $\vect{\mathcal{U}}_{h}^{n}(x_{\lo}^{i,\pm},\tilde{\vect{z}}^{i})=\vect{\mathcal{U}}_{h}^{n}(x_{\lo}^{i,\pm})$ and $\vect{\mathcal{U}}_{h}^{n}(x_{\hi}^{i,\pm},\tilde{\vect{z}}^{i})=\vect{\mathcal{U}}_{h}^{n}(x_{\hi}^{i,\pm})$.  
Then, using the quadrature rule in Eq.~\eqref{eq:glQuadratureSpace} and the LF flux in Eq.~\eqref{eq:globalLF}, we can write Eq.~\eqref{eq:GammaSpatial} as a convex combination
\begin{align}
  &\Gamma^{i}\big[\vect{\mathcal{U}}_{h}^{n}\big](\tilde{\vect{z}}^{i}) \nonumber \\
  &=\sum_{q=2}^{\hat{k}-1}\hat{w}_{q}\,\vect{\mathcal{U}}_{h}^{n}(\hat{x}_{q}^{i})
  + \hat{w}_{1}\,\Phi_{1}^{i}\big[\,\vect{\mathcal{U}}_{h}^{n}(x_{\lo}^{i,-}),\,\vect{\mathcal{U}}_{h}^{n}(x_{\lo}^{i,+}),\,\hat{\vect{v}}(x_{\lo}^{i})\,\big]
  + \hat{w}_{\hat{k}}\,\Phi_{\hat{k}}^{i}\big[\,\vect{\mathcal{U}}_{h}^{n}(x_{\hi}^{i,-}),\,\vect{\mathcal{U}}_{h}^{n}(x_{\hi}^{i,+}),\,\hat{\vect{v}}(x_{\hi}^{i})\,\big],
  \label{eq:GammaSpatialConvex}
\end{align}
where
\begin{align}
  \Phi_{1}^{i}\big[\,\vect{\mathcal{U}}_{a},\vect{\mathcal{U}}_{b},\hat{\vect{v}}\,\big]
  &=\vect{\mathcal{U}}_{b}+\lambda_{\vect{x}}^{i}\,\mathscr{F}_{\LaxF}^{i}\big(\vect{\mathcal{U}}_{a},\vect{\mathcal{U}}_{b},\hat{\vect{v}}\big), \label{eq:phi1} \\
  \Phi_{\hat{k}}^{i}\big[\,\vect{\mathcal{U}}_{a},\vect{\mathcal{U}}_{b},\hat{\vect{v}}\,\big]
  &=\vect{\mathcal{U}}_{a}-\lambda_{\vect{x}}^{i}\,\mathscr{F}_{\LaxF}^{i}\big(\vect{\mathcal{U}}_{a},\vect{\mathcal{U}}_{b},\hat{\vect{v}}\big), \label{eq:phiN}
\end{align}
and $\lambda_{\vect{x}}^{i}=\dt_{\vect{x}}/(\beta^{i}\,\hat{w}_{\hat{k}}\,|K_{\vect{x}}^{i}|)$.  
Since Eq.~\eqref{eq:GammaSpatialConvex} is a convex combination, it is sufficient to show the realizability of each term independently to obtain  $\Gamma^{i}\big[\vect{\mathcal{U}}_{h}^{n}\big](\tilde{\vect{z}}^{i})\in\mathcal{R}$.  
For the first term on the right-hand side of Eq.~\eqref{eq:GammaSpatialConvex}, it is sufficient that $\vect{\mathcal{U}}_{h}^{n}(\hat{x}_{q}^{i})\in\mathcal{R}$, which holds under the assumption that $\vect{\mathcal{U}}_{h}^{n}\in\mathcal{R}$ on $S^{\vect{K}}_{\otimes}$ for all $\vect{K}\in\cT$.
It remains to find conditions for which $\Phi_{1}^{i},\Phi_{\hat{k}}^{i}\in\mathcal{R}$, which we summarize in the following lemmas.  
\begin{lemma}
  Define
  \begin{equation}
    \Theta_{\pm}^{i}(\,\vect{\mathcal{U}},\hat{\vect{v}}\,) = \vect{\mathcal{U}}[\hat{\vect{v}}^{i}] \pm \vect{\mathcal{F}}^{i}(\,\vect{\mathcal{U}},\hat{\vect{v}}\,),
  \end{equation}
  where $\vect{\mathcal{U}}[\hat{\vect{v}}^{i}]$ and $\vect{\mathcal{F}}^{i}(\,\vect{\mathcal{U}},\hat{\vect{v}}\,)$ are defined as in Eqs.~\eqref{eq:conservedMoments} and \eqref{eq:phaseSpaceFluxes}, respectively, and $\hat{\vect{v}}^{i}=\big(\,\delta^{i1}\,\hat{v}^{1},\,\delta^{i2}\,\hat{v}^{2},\,\delta^{i3}\,\hat{v}^{3}\,\big)^{\intercal}$ as defined in Remark~\ref{rem:LF_flux}.
  Suppose that $\vect{\mathcal{U}}\in\mathcal{R}$ and $\hat{v}=|\hat{\vect{v}}|<1$.  
  Then $\Theta_{\pm}^{i}(\,\vect{\mathcal{U}},\hat{\vect{v}}\,)\in\mathcal{R}$.  
  \begin{proof}
    The first component of $\Theta_{\pm}^{i}(\,\vect{\mathcal{U}},\hat{\vect{v}}\,)$ can be written as
    \begin{equation}
      \f{1}{4\pi}\int_{\mathbb{S}^{2}}(\,1 \pm \hat{v}^{i}\,)\,(\,1 \pm \ell^{i}\,)\,f\,d\omega,
    \end{equation}
    while the remaining components can be written as
    \begin{equation}
      \f{1}{4\pi}\int_{\mathbb{S}^{2}}(\,1 \pm \hat{v}^{i}\,)\,(\,1 \pm \ell^{i}\,)\,f\,\ell_{j}\,d\omega, \quad (j=1,2,3).  
    \end{equation}
    Since $(\,1 \pm \hat{v}^{i}\,)\,(\,1 \pm \ell^{i}\,)\,f\in\mathfrak{R}$, the result follows.  
  \end{proof}
  \label{lem:realizableTheta}
\end{lemma}
\begin{lemma}
  Let $\Phi_{1}^{i}$ and $\Phi_{\hat{k}}^{i}$ be defined as in Eqs.~\eqref{eq:phi1} and \eqref{eq:phiN}, respectively.  
  Assume that the following holds
  \begin{enumerate}
    \item[\textup{(a)}] $\vect{\mathcal{U}}_{a},\vect{\mathcal{U}}_{b}\in\mathcal{R}$, defined as in Eq.~\eqref{eq:conservedMoments} as the moments of distributions $f_{a},f_{b}\in\mathfrak{R}$.  
    \item[\textup{(b)}] The three-velocity in Eq.~\eqref{eq:faceVelocity} satisfies $\hat{v}=|\hat{\vect{v}}|<1$.  
    \item[\textup{(c)}] The time step $\dt_{\vect{x}}$ is chosen such that $\lambda_{\vect{x}}^{i}\le(1-\hat{v})$.  
  \end{enumerate}
  Then $\Phi_{1}^{i}\big[\,\vect{\mathcal{U}}_{a},\vect{\mathcal{U}}_{b},\hat{\vect{v}}\,\big],\Phi_{\hat{k}}^{i}\big[\,\vect{\mathcal{U}}_{a},\vect{\mathcal{U}}_{b},\hat{\vect{v}}\,\big]\in\mathcal{R}$.
  \begin{proof}
    Define
    \begin{equation}
      \Theta_{0}^{i}(\,\vect{\mathcal{U}},\hat{\vect{v}}\,)
      =\f{\vect{\mathcal{U}}[\hat{\vect{v}}]-\lambda_{\vect{x}}^{i}\,\vect{\mathcal{U}}[\hat{\vect{v}}^{i}]}{1-\lambda_{\vect{x}}^{i}}.  
    \end{equation}
    Then, using the LF flux in Eq.~\eqref{eq:globalLF}, we can write
    \begin{equation}
      \Phi_{1}^{i}\big[\,\vect{\mathcal{U}}_{a},\vect{\mathcal{U}}_{b},\hat{\vect{v}}\,\big]
      = (1-\lambda_{\vect{x}}^{i})\,\Theta_{0}^{i}(\,\vect{\mathcal{U}}_{b},\hat{\vect{v}}\,)
      +\f{1}{2}\lambda_{\vect{x}}^{i}\,\Theta_{+}^{i}(\,\vect{\mathcal{U}}_{a},\hat{\vect{v}}\,)
      +\f{1}{2}\lambda_{\vect{x}}^{i}\,\Theta_{+}^{i}(\,\vect{\mathcal{U}}_{b},\hat{\vect{v}}\,),
    \end{equation}
    which is a convex combination for $\lambda_{\vect{x}}^{i}<1$.  
    From assumptions (a) and (b) above, it follows from Lemma~\ref{lem:realizableTheta} that $\Theta_{+}^{i}(\,\vect{\mathcal{U}}_{a},\hat{\vect{v}}\,),\Theta_{+}^{i}(\,\vect{\mathcal{U}}_{b},\hat{\vect{v}}\,)\in\mathcal{R}$.  
    It remains to show that $\Theta_{0}^{i}(\,\vect{\mathcal{U}}_{b},\hat{\vect{v}}\,)\in\mathcal{R}$.  
    The first component of $\Theta_{0}^{i}(\,\vect{\mathcal{U}}_{b},\hat{\vect{v}}\,)$ can be written as
    \begin{equation}
      \f{1}{4\pi}\int_{\mathbb{S}^{2}}\mathsf{f}(\omega)\,d\omega,
      \quad\text{where}\quad
      \mathsf{f}(\omega) = \f{[(1-\hat{\vect{v}}\cdot\vect{\ell})-\lambda_{\vect{x}}^{i}\,(1+\hat{v}^{i}\,\ell^{i})]}{(1-\lambda_{\vect{x}}^{i})}\,f,
    \end{equation}
    while the remaining components can be written as
    \begin{equation}
      \f{1}{4\pi}\int_{\mathbb{S}^{2}}\mathsf{f}(\omega)\,\ell_{j}(\omega)\,d\omega, \quad (j=1,2,3).  
    \end{equation}
    From assumptions (b) and (c), it follows that $\mathsf{f}\in\mathfrak{R}$, which implies $\Theta_{0}^{i}(\,\vect{\mathcal{U}}_{b},\hat{\vect{v}}\,)\in\mathcal{R}$.  
    The proof for $\Phi_{\hat{k}}^{i}\big[\,\vect{\mathcal{U}}_{a},\vect{\mathcal{U}}_{b},\hat{\vect{v}}\,\big]$ is analogous and is omitted.  
  \end{proof}
  \label{lem:realizablePhi}
\end{lemma}

To this end, the results of Lemma~\ref{lem:realizablePhi} lead to $\widehat{\vect{\mathcal{U}}}_{\vect{K}}^{n+\sfrac{1}{2},\vect{x}}\in\cR$ under the assumptions therein. It is straightforward to verify that these assumptions are fulfilled for each $\vect{K}\in\cT$ when the assumptions in Proposition~\ref{prop:realizabilityExplicit} hold. In particular, from Eq.~\eqref{eq:faceVelocity}, it is clear that $\hat{v}<1$ is implied by $v_h<1$.
Also, by defining 
\begin{equation}\label{eq:CFLSpatial}
	C^{i}:=\gamma^{\vect{x}} \beta^{i} \hat{w}_{\hat{k}}\:, 
\end{equation}
the time-step restriction in Eq.~\eqref{eq:timestepRestriction} guarantees $\lambda_{\vect{x}}^{i}\le(1-\hat{v})$ for all $\vect{K}\in\cT$.
Therefore, we have shown that, under the assumptions of Proposition~\ref{prop:realizabilityExplicit}, $\widehat{\vect{\mathcal{U}}}_{\vect{K}}^{n+\sfrac{1}{2},\vect{x}}\in\cR$ for all $\vect{K}\in\cT$.

\subsubsection{Energy Space Fluxes}
\label{sec:realizabilityEnergy}

For energy space advection, we define
\begin{equation}
  \widehat{\vect{\mathcal{U}}}_{\vect{K}}^{n+\sfrac{1}{2},\varepsilon}
  =\vect{\mathcal{U}}_{\vect{K}}^{n}+\f{\dt_{\varepsilon}}{|\vect{K}|}\,\vect{\mathcal{B}}_{h}^{\varepsilon}\big(\,\vect{\mathcal{U}}_{h}^{n},\vect{v}_{h}\,\big)_{\vect{K}}
  \quad (\dt_{\varepsilon}=\dt/\gamma^{\varepsilon}),
\end{equation}
and seek to find sufficient conditions such that $\widehat{\vect{\mathcal{U}}}_{\vect{K}}^{n+\sfrac{1}{2},\varepsilon}\in\mathcal{R}$.  
We proceed in a fashion similar to that in Section~\ref{sec:realizabilitySpatial}, and define
\begin{equation}
  \Gamma^{\varepsilon}[\vect{\mathcal{U}}_{h}](\vect{x})
  =\f{1}{|K_{\varepsilon}|}
  \Big[\,
    \int_{K_{\varepsilon}}\vect{\mathcal{U}}_{h}\,\varepsilon^{2}d\varepsilon
    -\dt_{\varepsilon}\,
    \Big(\,
      \varepsilon^{3}\,\widehat{\vect{\mathcal{F}}^{\varepsilon}}\big(\vect{\mathcal{U}}_{h},\vect{v}_{h}\big)|_{\varepsilon_{\hi}}
      -\varepsilon^{3}\,\widehat{\vect{\mathcal{F}}^{\varepsilon}}\big(\vect{\mathcal{U}}_{h},\vect{v}_{h}\big)|_{\varepsilon_{\lo}}
    \,\Big)
  \,\Big]
\end{equation}
so that
\begin{equation}
  \widehat{\vect{\mathcal{U}}}_{\vect{K}}^{n+\sfrac{1}{2},\varepsilon}
  =\f{1}{|\vect{K}_{\vect{x}}|}\int_{\vect{K}_{\vect{x}}}\Gamma^{\varepsilon}[\vect{\mathcal{U}}_{h}](\vect{x})\,d\vect{x}.
\end{equation}
Evaluating the integrals in the energy dimension using the same $\hat{k}$-point LGL quadrature rule leads to 
\begin{equation}  \label{eq:GammaEnergyConvex}
  \Gamma^{\varepsilon}[\vect{\mathcal{U}}_{h}](\vect{x})
  =\sum_{q=2}^{\hat{k}-1}\hat{w}_{q}\,\hat{\varepsilon}_{q}^{2}\,\vect{\mathcal{U}}_{h}(\hat{\varepsilon}_{q})
  +\hat{w}_{1}\,\varepsilon_{\lo}^{2}\,\Phi_{1}^{\varepsilon}\big[\,\vect{\mathcal{U}}_{h}(\varepsilon_{\lo}^{-}),\vect{\mathcal{U}}_{h}(\varepsilon_{\lo}^{+}),\vect{v}_{h}\,\big]
  +\hat{w}_{\hat{k}}\,\varepsilon_{\hi}^{2}\,\Phi_{\hat{k}}^{\varepsilon}\big[\,\vect{\mathcal{U}}_{h}(\varepsilon_{\hi}^{-}),\vect{\mathcal{U}}_{h}(\varepsilon_{\hi}^{+}),\vect{v}_{h}\,\big]\:,
\end{equation}
where the integral of the moments is exact when $\hat{k}\geq\frac{k+5}{2}$, i.e., 
\begin{equation}
	\int_{K_{\varepsilon}}\vect{\mathcal{U}}_{h}(\varepsilon)\,\varepsilon^{2}d\varepsilon
	=\hat{Q}^{\varepsilon}\big[\vect{\mathcal{U}}_{h}\big]
	=|K_{\varepsilon}|\sum_{q=1}^{\hat{k}}\hat{w}_{q}\,\vect{\mathcal{U}}_{h}(\hat{\varepsilon}_{q})\,\hat{\varepsilon}_{q}^{2}\:.
\end{equation}
Since $\Gamma^{\varepsilon}[\vect{\mathcal{U}}_{h}](\vect{x})$ is written as a convex combination in Eq.~\eqref{eq:GammaEnergyConvex}, the realizability of each term on the right-hand side gives the realizability of  $\Gamma^{\varepsilon}[\vect{\mathcal{U}}_{h}](\vect{x})$. Since $\vect{\mathcal{U}}_{h}(\hat{\varepsilon}_{q})\in\cR$ for each $\hat{\varepsilon}_{q}$ under the assumption that $\vect{\mathcal{U}}_{h}\in\cR$ on $\widehat{S}^{\vect{K}}_{\otimes}$ for all $\vect{K}\in\cT$, we focus on proving realizability of $\Phi_{1}^{\varepsilon}$ and $\Phi_{\hat{k}}^{\varepsilon}$, which are defined as
\begin{align}
  \Phi_{1}^{\varepsilon}\big[\,\vect{\mathcal{U}}_{a},\vect{\mathcal{U}}_{b},\vect{v}_h\,\big]
  &=\vect{\mathcal{U}}_{b}
    +\lambda_{\lo}^{\varepsilon}\,\mathscr{F}_{\LaxF}^{\varepsilon}\big(\vect{\mathcal{U}}_{a},\vect{\mathcal{U}}_{b},\vect{v}_h\big) \label{eq:phi1Energy} \\
  &=(1-\alpha^{\varepsilon}\lambda_{\lo}^{\varepsilon})\,\Theta_{0,\lo}^{\varepsilon}(\vect{\mathcal{U}}_{b},\vect{v}_h)
  +\f{1}{2}\,\alpha^{\varepsilon}\lambda_{\lo}^{\varepsilon}\,\Theta_{+}^{\varepsilon}(\vect{\mathcal{U}}_{a},\vect{v}_h)
  +\f{1}{2}\,\alpha^{\varepsilon}\lambda_{\lo}^{\varepsilon}\,\Theta_{+}^{\varepsilon}(\vect{\mathcal{U}}_{b},\vect{v}_h), \nonumber \\
  \Phi_{\hat{k}}^{\varepsilon}\big[\,\vect{\mathcal{U}}_{a},\vect{\mathcal{U}}_{b},\vect{v}_h\,\big]
  &=\vect{\mathcal{U}}_{a}
    -\lambda_{\hi}^{\varepsilon}\,\mathscr{F}_{\LaxF}^{\varepsilon}\big(\vect{\mathcal{U}}_{a},\vect{\mathcal{U}}_{b},\vect{v}_h\big) \label{eq:phiNEnergy} \\
  &=(1-\alpha^{\varepsilon}\lambda_{\hi}^{\varepsilon})\,\Theta_{0,\hi}^{\varepsilon}(\vect{\mathcal{U}}_{a},\vect{v}_h)
  +\f{1}{2}\,\alpha^{\varepsilon}\lambda_{\hi}^{\varepsilon}\,\Theta_{-}^{\varepsilon}(\vect{\mathcal{U}}_{a},\vect{v}_h)
  +\f{1}{2}\,\alpha^{\varepsilon}\lambda_{\hi}^{\varepsilon}\,\Theta_{-}^{\varepsilon}(\vect{\mathcal{U}}_{b},\vect{v}_h), \nonumber
\end{align}
where we used the definition of $\mathscr{F}_{\LaxF}^{\varepsilon}$ given in Eq.~\eqref{eq:globalLFenergy} and defined $\lambda_{\lo/\hi}^{\varepsilon}=\varepsilon_{\lo/\hi}\,\dt_{\varepsilon}/(\hat{w}_{\hat{k}}\,|K_{\varepsilon}|)$, 
\begin{equation}\label{eq:thetaEnergy}
  \Theta_{0,\lo/\hi}^{\varepsilon}(\vect{\mathcal{U}},\vect{v}_h)
  =\f{\vect{\mathcal{U}}[\vect{v}_h]-\alpha^{\varepsilon}\lambda_{\lo/\hi}^{\varepsilon}\,\vect{\mathcal{M}}}{1-\alpha^{\varepsilon}\lambda_{\lo/\hi}^{\varepsilon}},  \quand
  \Theta_{\pm}^{\varepsilon}(\vect{\mathcal{U}},\vect{v}_h)
  =\vect{\mathcal{M}} \pm \f{1}{\alpha^{\varepsilon}}\,\vect{\mathcal{F}}^{\varepsilon}(\vect{\mathcal{U}},\vect{v}_h) \:.
\end{equation}

Similar to the approach in Section~\ref{sec:realizabilitySpatial}, the following two lemmas show realizability of $\Theta_{\pm}^{\varepsilon}$ and $\Theta_{0,\lo/\hi}^{\varepsilon}$.

\begin{lemma}
	Let $\Theta_{\pm}^{\varepsilon}(\vect{\mathcal{U}},\vect{v}_h)$ be given as in Eq.~\eqref{eq:thetaEnergy}.  
	Assume that $\vect{\mathcal{U}}\in\mathcal{R}$.  
	Then $\Theta_{\pm}^{\varepsilon}\in\mathcal{R}$.  
	\begin{proof}
		The first component of $\Theta_{\pm}^{\varepsilon}$ can be written as
		\begin{equation}
			\f{1}{4\pi}\int_{\mathbb{S}^{2}}\mathsf{f}_{\pm}[\vect{v}_h,\alpha^{\varepsilon}](\omega)\,d\omega,
			\quad\text{where}\quad
			\mathsf{f}_{\pm}[\vect{v}_h,\alpha^{\varepsilon}](\omega)
			=\big(\,1\pm Q(\vect{v}_h)/\alpha^{\varepsilon}\,\big)\,f(\omega),
		\end{equation}
		and where $Q(\vect{v}_h)$ is the quadratic form in Eq.~\eqref{eq:quadraticForm}.  
		Similarly, the remaining components of $\Theta_{\pm}^{\varepsilon}$ can be written as
		\begin{equation}
			\f{1}{4\pi}\int_{\mathbb{S}^{2}}\mathsf{f}_{\pm}[\vect{v}_h,\alpha^{\varepsilon}](\omega)\,\vect{\ell}(\omega)\,d\omega.  
		\end{equation}
		Since $|Q(\vect{v}_h)|/\alpha^{\varepsilon}\le1$, it follows that $\mathsf{f}_{\pm}[\vect{v}_h,\alpha^{\varepsilon}](\omega)\in\mathfrak{R}$ and $\Theta_{\pm}^{\varepsilon}\in\mathcal{R}$.  
	\end{proof}
\end{lemma}

\begin{lemma}\label{lem:realizableTheta0Energy}
  Consider $\Theta_{0,\lo/\hi}^{\varepsilon}(\vect{\mathcal{U}},\vect{v}_h)$ as defined in Eq.~\eqref{eq:thetaEnergy}.  
  Assume that $\vect{\mathcal{U}},\vect{\mathcal{M}}\in\mathcal{R}$, $v_h=|\vect{v}_h|<1$, and $\eta_{\lo/\hi}^{\varepsilon}:=\alpha^{\varepsilon}\lambda_{\lo/\hi}^{\varepsilon}<(1-v_h)$.  
  Then, $\Theta_{0,\lo/\hi}^{\varepsilon}(\vect{\mathcal{U}},\vect{v}_h)\in\mathcal{R}$.  
  \begin{proof}
    The first component of $\Theta_{0,\lo/\hi}^{\varepsilon}(\vect{\mathcal{U}},\vect{v}_h)$ can be written as
    \begin{equation}
      \f{1}{4\pi}\int_{\mathbb{S}^{2}}\mathsf{f}[\vect{v}_h,\eta_{\lo/\hi}^{\varepsilon}](\omega)\,d\omega,
      \quad\text{where}\quad
      \mathsf{f}[\vect{v}_h,\eta_{\lo/\hi}^{\varepsilon}](\omega) = \f{(1-\vect{v}_h\cdot\vect{\ell}-\eta_{\lo/\hi}^{\varepsilon})}{(1-\eta_{\lo/\hi}^{\varepsilon})}\,f(\omega).
    \end{equation}
    The remaining components of $\Theta_{0,\lo/\hi}^{\varepsilon}(\vect{\mathcal{U}},\vect{v}_h)$ can be written as
    \begin{equation}
      \f{1}{4\pi}\int_{\mathbb{S}^{2}}\mathsf{f}[\vect{v}_h,\eta_{\lo/\hi}^{\varepsilon}](\omega)\,\vect{\ell}(\omega)\,d\omega.  
    \end{equation}
    Since $v_h<1$ and $\eta_{\lo/\hi}^{\varepsilon}<1-v$, we have $(1-\vect{v}_h\cdot\vect{\ell}-\eta_{\lo/\hi}^{\varepsilon})\ge(1-v_h)-\eta_{\lo/\hi}^{\varepsilon}>0$.  
    This, together with $f\in\mathfrak{R}$, implies that $\mathsf{f}[\vect{v}_h,\eta_{\lo/\hi}^{\varepsilon}](\omega)\in\mathfrak{R}$ and $\Theta_{0,\lo/\hi}^{\varepsilon}(\vect{\mathcal{U}},\vect{v}_h)\in\mathcal{R}$.
  \end{proof}
\end{lemma}

Analogous to the spatial advection case, the assumptions in Lemma~\ref{lem:realizableTheta0Energy} are fulfilled for all $\vect{K}\in\cT$ under assumptions of Proposition~\ref{prop:realizabilityExplicit}, when 
\begin{equation}\label{eq:CFLEnergy}
	C^{\varepsilon}:=\gamma^{\varepsilon} \alpha^{\varepsilon} \hat{w}_{\hat{k}}
\end{equation}
is used in the time-step restriction \eqref{eq:timestepRestriction}. 
Under these assumptions, $\eta_{\lo/\hi}^{\varepsilon}:=\alpha^{\varepsilon}\lambda_{\lo/\hi}^{\varepsilon}<(1-v_h)\le 1$.  
Therefore $\Phi_{1}^{\varepsilon}$ and $\Phi_{\hat{k}}^{\varepsilon}$ are convex combinations of realizable terms, and are thus realizable.
We have shown that, under the assumptions of Proposition~\ref{prop:realizabilityExplicit}, $\widehat{\vect{\mathcal{U}}}_{\vect{K}}^{n+\sfrac{1}{2},\varepsilon}\in\cR$ for all $\vect{K}\in\cT$.

\subsubsection{Sources}
\label{sec:realizabilitySource}

The last part of the explicit update involves the source term in the number flux equation.  
We define
\begin{align}
  \widehat{\vect{\mathcal{U}}}_{\vect{K}}^{n+\sfrac{1}{2},\vect{\mathcal{S}}}
  &=\vect{\mathcal{U}}_{\vect{K}}^{n}+\f{\dt_{\vect{\mathcal{S}}}}{|\vect{K}|}\big(\,\vect{\mathcal{S}}(\vect{\mathcal{U}}_{h}^{n},\vect{v}_{h})\,\big)_{\vect{K}}
  \quad(\dt_{\vect{\mathcal{S}}}=\dt/\gamma^{\vect{\mathcal{S}}}) \nonumber \\
  &=\f{1}{|\vect{K}|}\int_{\vect{K}}\Big[\,\vect{\mathcal{U}}_{h}^{n}+\dt_{\vect{\mathcal{S}}}\,\vect{\mathcal{S}}(\vect{\mathcal{U}}_{h}^{n},\vect{v}_{h})\,\Big]\,\tau\,d\vect{z}.
  \label{eq:source_update}
\end{align}
From the definition of the source term $\vect{\mathcal{S}}$ in Eq.~\eqref{eq:sources}, the number density is not affected in the source update. Thus we have $\widehat{\mathcal{N}}_{\vect{K}}^{n+\sfrac{1}{2},\vect{\mathcal{S}}} = {\mathcal{N}}_{\vect{K}}^{n} > 0$, which, together with the results obtained in Sections~\ref{sec:realizabilitySpatial} and \ref{sec:realizabilityEnergy}, concludes the proof of the first claim in Proposition~\ref{prop:realizabilityExplicit}.

Ideally, one would expect to show that $\widehat{\vect{\mathcal{U}}}_{\vect{K}}^{n+\sfrac{1}{2},\vect{\mathcal{S}}}\in\mathcal{R}$ under time-step restrictions similar to the ones in Sections~\ref{sec:realizabilitySpatial} and \ref{sec:realizabilityEnergy}. Unfortunately, this is not true in the three-dimensional case considered in this paper. 
In the rest of this section, we will show that (i) realizability of $\widehat{\vect{\mathcal{U}}}_{\vect{K}}^{n+\sfrac{1}{2},\vect{\mathcal{S}}}$ is preserved by the semi-discrete equation, i.e., without time discretization, and (ii) with the forward Euler discretization in Eq.~\eqref{eq:imexForwardEuler}, $\widehat{\vect{\mathcal{U}}}_{\vect{K}}^{n+\sfrac{1}{2},\vect{\mathcal{S}}}\in\cR$ in a reduced, one-dimensional planar geometry.

\begin{prop}[Semi-discrete source update]
Given a quadrature rule $\vect{Q}:C^{0}(\vect{K})\to\mathbb{R}$ with positive weights and points given by the set $\vect{S}^{\vect{K}}_{\otimes}$, we show that, for all $\vect{z}\in\vect{S}^{\vect{K}}_{\otimes}\subset\vect{K}$, the solution $\vect{\mathcal{U}}_{h}(\vect{z},t)$ to the semi-discrete equation
\begin{equation}\label{eq:continuum_source_update}
	\p_t \vect{\mathcal{U}}_{h}(\vect{z},t) = \vect{\mathcal{S}}(\vect{\mathcal{U}}_{h}(\vect{z},t),\vect{v}_{h}(\vect{z}))
\end{equation}
remains in the realizable set $\cR$ for all $t\geq t_0$, provided that $\vect{\mathcal{U}}_{h}(\vect{z},t_0)$ is realizable. 
\end{prop}
This semi-discrete equation is consistent with the source update portion in Eq.~\eqref{eq:twoMomentModelCompact} and results in Eq.~\eqref{eq:source_update} after applying forward Euler discretization and cell-averaging.

\begin{proof}
To show that $\vect{\mathcal{U}}_{h}(\vect{z},t)\in\cR$ for $t\geq t_0$, we first observe that since the first component of $\vect{\mathcal{S}}(\vect{\mathcal{U}},\vect{v})$ is zero (see Eq.~\eqref{eq:sources}), the source update does not affect $\mathcal{N}_h$. Thus, showing $\vect{\mathcal{U}}_{h}(\vect{z},t)\in \cR$ is equivalent to proving that $\mathcal{G}_h(\vect{z},t) \leq \mathcal{N}_h(\vect{z})$, where $\mathcal{G}_h(\vect{z},t) = |\vect{\mathcal{G}}_h(\vect{z},t)|$ with $\vect{\mathcal{G}}_h$ the number flux governed by Eq.~\eqref{eq:continuum_source_update}.
Due to the continuity of $\mathcal{G}_{h}(\vect{z},t)$ in time, it suffices to show that if $\mathcal{G}_h(\vect{z},\hat{t}) = \mathcal{N}_h(\vect{z})$ for some $\hat{t}\geq t_0$, then $\mathcal{G}_h(\vect{z},t) = \mathcal{N}_h(\vect{z})$ for all $t\geq\hat{t}$, i.e., the number flux magnitude does not exceed the number density. Indeed, the number flux portion of Eq.~\eqref{eq:continuum_source_update} is given by
\begin{equation}
	\p_t \mathcal{G}_{h,j}
	= \mathcal{Q}^{i}_{\iipt kj}(\pd{v^{k}}{i})_{h}-\mathcal{I}^{i}(\pd{v_{j}}{i})_{h} 
	=\f{1}{4\pi}\int_{\mathbb{S}^{2}}
	\big(\,\ell^{i}(\omega)\ell_{k}(\omega)\ell_{j}(\omega)(\pd{v^{k}}{i})_{h}-\ell^{i}(\omega)(\pd{v_{j}}{i})_{h}\,\big)
	\,f(\omega,t)\,d\omega. 
\end{equation}
Suppose $\mathcal{G}_h(\vect{z},\hat{t}) = \mathcal{N}_h(\vect{z})$ for some $\hat{t}\geq t_0$, it is known \cite{fialkow1991recursiveness} that the distribution function $f(\omega)$ takes the form of a Dirac delta function, i.e., $f(\omega) = c\, \delta(\hat{\omega})$ for some $c>0$, $\hat{\omega}\in\bbS^2$. Therefore, at $t=\hat{t}$, we have $\mathcal{G}_h^j = c \, \big(\,1+v^{k}\ell_{k}(\hat{\omega})\,\big)\,\ell^{j}(\hat{\omega})$ and
\begin{equation}
	\p_t \mathcal{G}_{h,j}
	=\f{c}{4\pi}
	\big(\,\ell^{i}(\hat{\omega})\ell_{k}(\hat{\omega})\ell_{j}(\hat{\omega})(\pd{v^{k}}{i})_{h}-\ell^{i}(\hat{\omega})(\pd{v_{j}}{i})_{h}\,\big)\:.
\end{equation}
Thus,
\begin{equation}\label{eq:continuumRealizability}
	\begin{alignedat}{2}
	\f{1}{2} \p_t (\mathcal{G}_{h})^2 &=\, \mathcal{G}_{h}^j \p_t \mathcal{G}_{h,j}\\
	&=	\f{c^2}{4\pi} \big(\,1+v^{k}\ell_{k}(\hat{\omega})\,\big)\,\ell^{j}(\hat{\omega}) 
	\big(\,\ell^{i}(\hat{\omega})\ell_{k}(\hat{\omega})\ell_{j}(\hat{\omega})(\pd{v^{k}}{i})_{h}-\ell^{i}(\hat{\omega})(\pd{v_{j}}{i})_{h}\,\big) = 0\:,
	\end{alignedat}
\end{equation}
where the fact $\ell^{i}\ell_{i}=1$ is used in the last equality.
Eq.~\eqref{eq:continuumRealizability} indicates that the number flux magnitude does not change once $\mathcal{G}_h(\vect{z},\hat{t}) = \mathcal{N}_h(\vect{z})$ for some $\hat{t}\geq t_0$ and implies that $\vect{\mathcal{U}}_{h}(\vect{z},t)\in\cR$ for $t\geq t_0$.
\end{proof}

\begin{rem}
	The result in Eq.~\eqref{eq:continuumRealizability} also explains why the discretized source update \eqref{eq:source_update} cannot guarantee realizability of the updated moments. 
	Specifically, Eq.~\eqref{eq:continuumRealizability} suggests that, for moments on the realizable boundary ($\mathcal{G}_h = \mathcal{N}_h$), the continuous source update \eqref{eq:continuum_source_update} moves the moments tangentially with the boundary of the realizable set. 
	Once explicit discretization is applied, e.g., Eq.~\eqref{eq:source_update}, the update may result in unrealizable moments, regardless of the time-step size.
\end{rem}

Next, we show that, in a one-dimensional planar geometry \cite[Section~6.5]{mihalasMihalas_1999}, the discretized source update Eq.~\eqref{eq:source_update} preserves realizability of the moments when a time-step restriction is satisfied.
In the planar geometry, the spatial fluxes are zero in two of the three spatial dimensions (e.g., $\p_{x^2} f = \p_{x^3}f = 0$) with the angular direction reduced from $\omega=(\vartheta, \varphi)\in\bbS^2$ to $\mu=\cos\,\vartheta\in[-1,1]$. In the remainder of this subsection, we use $x(=x^1)$ to denote the only spatial dimension that has nonzero fluxes, and use a scalar function $v$ to denote the velocity which varies only in the $x$ direction. Moreover, the primitive moments in the planar geometry are given by
\begin{equation}
	  \big\{\,\mathcal{D},\,\mathcal{I},\,\mathcal{K},\,\mathcal{Q}\,\big\}(\varepsilon,{x},t)
	=\f{1}{2}\int_{-1}^{1}f(\omega,\varepsilon,{x},t)\,\big\{\,1,\,\mu,\,\mu^{2},\,\mu^{3}\,\big\}\,d\mu,
	\label{eq:reducedAngularMoments}
\end{equation}
and the conserved moments are 
$\mathcal{N} =\mathcal{D}+v\,\mathcal{I}$ and $\mathcal{G}=\mathcal{I}+v\,\mathcal{K}$.
In this case, the semi-discrete source update Eq.~\eqref{eq:continuum_source_update} reduces to
\begin{equation}
	\p_t \cN = 0\:,\quand \p_t \cG = \frac{1}{2} \int_{-1}^{1} \mu(\mu^2-1)(\p_x v)_{h}\,f(\mu) d\mu\:.
\end{equation}
The following proposition shows that the discretized version of this source update preserves moment realizability under a time step restriction.

\begin{prop}\label{prop:realizableSource1D}
	In the planar geometry,
	suppose Assumption~\ref{assum:exact_closure} holds, $v_{h}<1$, and the time step satisfies 
	\begin{equation}\label{eq:sourceTimeStepRestriction}
		\dt \leq \f{1}{2}\,\gamma^{\vect{\mathcal{S}}}\, \f{1-v_{h}}{|(\pd{v}{x})_{h}|}, \quad\bigg(\,\text{ i.e., }\,	\dt_{\vect{\mathcal{S}}} \leq \f{1}{2}\, \f{1-v_{h}}{|(\pd{v}{x})_{h}|}\bigg).
	\end{equation} 
	Then the discretized source update Eq.~\eqref{eq:source_update} gives a realizable cell-averaged moment $\widehat{\vect{\mathcal{U}}}_{\vect{K}}^{n+\sfrac{1}{2},\vect{\mathcal{S}}}$ for all $\vect{K}\in\cT$, provided
	$\vect{\mathcal{U}}_{h}^{n} \in \cR$ on all $\vect{K}\in\cT$.
\end{prop}
\begin{proof}
	In this proof, we show the realizability of $\widehat{\vect{\mathcal{U}}}_{h}^{n+\sfrac{1}{2},\vect{\mathcal{S}}}:= \vect{\mathcal{U}}_{h}^{n}+\dt_{\vect{\mathcal{S}}}\,\vect{\mathcal{S}}(\vect{\mathcal{U}}_{h}^{n},\vect{v}_{h})$, which leads to the realizability of $\widehat{\vect{\mathcal{U}}}_{\vect{K}}^{n+\sfrac{1}{2},\vect{\mathcal{S}}}$ when the element integral in Eq.~\eqref{eq:source_update} is evaluated using quadrature rules with positive weights in both the spatial and energy dimensions. 
	
	We start with denoting $\widehat{\vect{\mathcal{U}}}_{h}^{n+\sfrac{1}{2},\vect{\mathcal{S}}}=:(\widehat{\mathcal{N}}_{h}^{n+\sfrac{1}{2},\vect{\mathcal{S}}}, \widehat{\mathcal{G}}_{h}^{n+\sfrac{1}{2},\vect{\mathcal{S}}})$. 
	In the planar geometry, the number density $\widehat{\mathcal{N}}_{h}^{n+\sfrac{1}{2},\vect{\mathcal{S}}}$ and number flux $\widehat{\mathcal{G}}_{h}^{n+\sfrac{1}{2},\vect{\mathcal{S}}}$ are both scalar-valued.
	From Assumption~\ref{assum:exact_closure} and the definition of the source terms $\vect{\mathcal{S}}$ in Eq.~\eqref{eq:sources}, we can write 
	\begin{align}
		\widehat{\mathcal{N}}_{h}^{n+\sfrac{1}{2},\vect{\mathcal{S}}}
	&=\f{1}{2}\int_{-1}^{1}\big(\,1+v_{h}\mu\,\big)\,f(\mu)\,d\mu,\\
	\widehat{\mathcal{G}}_{h}^{n+\sfrac{1}{2},\vect{\mathcal{S}}}
	&=\f{1}{2}\int_{-1}^{1}
	\big[\,
	\big(\,1+v_{h}\mu\,\big)\,\mu
	+\dt_{\vect{\mathcal{S}}}\,\big(\,\mu^3(\pd{v}{x})_{h}-\mu(\pd{v}{x})_{h}\,\big)
	\,\big]\,f(\mu)\,d\mu \\
	&=\f{1}{2}\int_{-1}^{1}
	\big(\,1+v_{h}\mu\,\big)\,f(\mu)\,\big[\,
	\mu
	-\dt_{\vect{\mathcal{S}}}\,(\pd{v}{x})_{h}\mu\,\f{1-\mu^2}{1+v_{h}\mu}\,
	\,\big]\,d\mu,\nonumber
\end{align}
where $f\in\mathfrak{R}$. Since $v_{h}<1$ and $\mu\in[-1,1]$, it is clear that $\widehat{\mathcal{N}}_{h}^{n+\sfrac{1}{2},\vect{\mathcal{S}}}>0$.
We next prove $\widehat{\mathcal{N}}_{h}^{n+\sfrac{1}{2},\vect{\mathcal{S}}} - |\widehat{\mathcal{G}}_{h}^{n+\sfrac{1}{2},\vect{\mathcal{S}}} |\geq0$ when $\dt_{\vect{\mathcal{S}}}$ satisfies Eq.~\eqref{eq:sourceTimeStepRestriction}.
By Cauchy-Schwartz inequality,
\begin{align}
	|\widehat{\mathcal{G}}_{h}^{n+\sfrac{1}{2},\vect{\mathcal{S}}}|^2
	&\leq
	\f{1}{4}\int_{-1}^{1}
	(1+v_{h}\mu)\,f(\mu)\,d\mu
	\int_{-1}^{1}
	(1+v_{h}\mu)\,f(\mu)\,\big[\,
	\mu
	-\dt_{\vect{\mathcal{S}}}\,(\pd{v}{x})_{h}\mu\,\f{1-\mu^2}{1+v_{h}\mu}\,
	\,\big]^2\,d\mu.
\end{align}
We then show that, under Eq.~\eqref{eq:sourceTimeStepRestriction},
$
\big[\,\mu-\dt_{\vect{\mathcal{S}}}\,(\pd{v}{x})_{h}\mu\,\f{1-\mu^2}{1+v_{h}\mu}\,\big]^2
\leq 1$ for $\mu\in[-1,1]$.
This inequality clearly holds when $\mu=\pm 1$ and $\mu=0$. 
We thus focus on the case when $\mu\in(-1,1)$ and $\mu\neq0$.
Since $\dt_{\vect{\mathcal{S}}}>0$ and $\,\f{1-\mu^2}{1+v_{h}\mu}\geq0$, the inequality holds when
\begin{equation}
	\dt_{\vect{\mathcal{S}}} \leq \f{1+v_{h}\mu}{(\pd{v}{x})_{h}\mu\,(1-\mu)}\,\text{ if }\, (\pd{v}{x})_{h}\,\mu>0,\,\text{ and }\,
	\dt_{\vect{\mathcal{S}}} \leq \f{1+v_{h}\mu}{(-(\pd{v}{x})_{h}\mu) (1+\mu)}\,\text{ if }\, (\pd{v}{x})_{h}\,\mu<0.
\end{equation}
It is straightforward to verify that Eq.~\eqref{eq:sourceTimeStepRestriction} gives a sufficient condition to the two time-step restrictions above.
\end{proof}

\subsection{Realizability-enforcing Limiter}
\label{sec:realizabilityLimiter}

It has been shown in Proposition~\ref{prop:realizabilityExplicit} that, when starting from realizable moments $\vect{\mathcal{U}}_{h}^{n}$, the explicit update in Eq.~\eqref{eq:imexForwardEuler} is guaranteed to provide updated cell-averaged moments $\widehat{\vect{\mathcal{U}}}_{\vect{K}}^{n+\sfrac{1}{2}}$ with number density $\widehat{\mathcal{N}}_{\vect{K}}^{n+\sfrac{1}{2}}>0$ for every $\vect{K}$ under a reasonable time-step restriction.
In this section, we discuss how the realizability-enforcing limiter proposed in \cite{chu_etal_2019} is used here in Eq.~\eqref{eq:realizabilityLimiter} to enforce realizability of moments $\vect{\mathcal{U}}_{h}$ at a point set $\widetilde{S}^{\vect{K}}_{\otimes}$ defined in Eq.~\eqref{eq:AllSetUnion}, which covers all DG nodal points as well as the auxiliary points in element $\vect{K}$.

In \cite{chu_etal_2019}, the realizability-enforcing limiter was formulated following the approach considered in \cite{zhangShu_2010a,zhangShu_2010b} for constructing bound-preserving limiters for high-order DG schemes. 
The limiter enforces moment realizability at each quadrature point in a DG element by relaxing unrealizable moments towards the realizable cell-averaged moments. 
Specifically, this limiter replaces unrealizable moments with their convex combinations with the cell-averaged moment, which preserves the Eulerian-frame particle number in each element (but not the energy; see Section~\ref{sec:EnergyLimiter} for further discussions) when the same convex combination factor ($\theta^{\cN}_{\vect{K}}$ and $\theta^{\vect{\mathcal{U}}}_{\vect{K}}$) is applied to all moments within the element.
For completeness, the steps taken in this realizability-enforcing limiter are summarized in Algorithm~\ref{algo:realizabilityLimiter}.  
We refer to \cite{chu_etal_2019} and references therein for detailed discussions.

\begin{algorithm}[h]
	\normalsize
	\medskip
	\caption{${\vect{\mathcal{U}}}_{h}=\texttt{RealizabilityLimiter}(\widehat{\vect{\mathcal{U}}}_{h})$}
	\label{algo:realizabilityLimiter}
	
	{\bf Inputs:} Discretized moments $\widehat{\vect{\mathcal{U}}}_h$ with $\widehat{\mathcal{N}}_{\vect{K}}>0$ for all $\vect{K}\in\cT$.\\
	{\bf Parameter:} $0<\delta\ll 1$.\\
	\For{each element $\vect{K}$}{
		\If{$\widehat{\vect{\mathcal{U}}}_{\vect{K}}\in\cR$}{
			\tcc{limit number density}
			$\tilde{\mathcal{N}}_{h} \leftarrow \theta^{\cN}_{\vect{K}} \widehat{\mathcal{N}}_{h}+ (1-\theta^{\cN}_{\vect{K}}) \widehat{\mathcal{N}}_{\vect{K}}$ with
			$\theta^{\cN}_{\vect{K}}\leftarrow \min\{\f{\widehat{\mathcal{N}}_{\vect{K}}}{\widehat{\mathcal{N}}_{\vect{K}} - \min_{\vect{z}\in \widetilde{S}^{\vect{K}}_{\otimes}} \widehat{\mathcal{N}}_{h}(\vect{z})},1\}$\;
			
			\tcc{build intermediate moments}
			$\tilde{\vect{\mathcal{U}}}_{h} \leftarrow (\tilde{\mathcal{N}}_{h}, \widehat{\vect{\mathcal{G}}}_{h})$\;
			
			\tcc{limit full moments}
			${\vect{\mathcal{U}}}_{h}\leftarrow\theta^{\vect{\mathcal{U}}}_{\vect{K}} \tilde{\vect{\mathcal{U}}}_{h}+ (1-\theta^{\vect{\mathcal{U}}}_{\vect{K}})\, \widehat{\vect{\mathcal{U}}}_{\vect{K}}$ where\\
			$
				\theta^{\vect{\mathcal{U}}}_{\vect{K}}\leftarrow \argmin_{\theta}\{\theta\in[0,1]\colon \gamma(\theta\, \tilde{\vect{\mathcal{U}}}_{h}(\vect{z})+ (1-\theta)\, \widehat{\vect{\mathcal{U}}}_{\vect{K}}) \geq0,\,\forall\vect{z}\in \widetilde{S}^{\vect{K}}_{\otimes}\} 
			$
			with $\gamma$ defined in Eq.~\eqref{eq:realizableSet}.
		}
		\Else{
			\tcc{replace number densities with the cell average and shrink number fluxes accordingly}
			${\vect{\mathcal{U}}}_{h} \leftarrow ({\mathcal{N}}_{h},{ \vect{\mathcal{G}}}_{h})$
			with ${\mathcal{N}}_{h} = \widehat{\mathcal{N}}_{\vect{K}}$ and ${ \vect{\mathcal{G}}}_{h} =  (1-\delta)\,\widehat{\mathcal{N}}_{\vect{K}}\, \frac{\widehat{\vect{\mathcal{G}}}_{h}}{|\widehat{\vect{\mathcal{G}}}_{h}|}$\;
		}
	}

\end{algorithm}	

As seen in Algorithm~\ref{algo:realizabilityLimiter}, starting from discretized moment $\widehat{\vect{\mathcal{U}}}_{h}$ with positive cell-averaged number density $\widehat{\mathcal{N}}_{\vect{K}}$, the limiter enforces realizability of the resulting moments ${\vect{\mathcal{U}}}_{h}$ in the point set $\widetilde{S}^{\vect{K}}_{\otimes}$ by limiting toward the cell-averaged moments.
The limiter is guaranteed to provide realizable outputs at the point set when the starting moment has a positive cell-averaged number density, thus Proposition~\ref{prop:realizabilityLimiter} holds.

We note that, when approximate closures are considered, the explicit update may not result in moments with positive cell-averaged number density (since Assumption~\ref{assum:exact_closure} does not hold). 
If a negative cell-averaged number density is observed in element $\vect{K}$, we set the moments in $\vect{K}$ to be an isotropic moment with close to zero but positive number density and zero number flux. 
This safeguard affects the conservation property of the scheme, however, we do not observe a negative cell-averaged number density in any of the numerical experiments presented in Section~\ref{sec:numericalResults}.

\subsection{Conversion between Conserved and Primitive Moments}
\label{sec:momentConversionRealizability}

In this section, we prove Proposition~\ref{prop:realizabilityMomentConversion} by showing that, under Assumption~\ref{assum:exact_closure} and assuming $v_h<1$, (i) the conversion between conserved and primitive moments preserves realizability and (ii) the iterative solver in Eq.~\eqref{eq:Picard} is guaranteed to converge to a unique $\vect{\mathcal{M}}\in\cR$ that satisfies Eq.~\eqref{eq:ConservedToPrimitive} given $\vect{\mathcal{U}}\in\cR$.

In the following two lemmas, we show that the realizability is preserved in the conversion between conserved and primitive moments.
\begin{lemma}\label{lem:ForwardRealizability}
	Suppose Assumption~\ref{assum:exact_closure} holds and $v<1$. Let $\vect{\mathcal{U}}$ be given as in Eq.~\eqref{eq:ConservedToPrimitive} with $\vect{\mathcal{M}}\in\mathcal{R}$,
	then $\vect{\mathcal{U}}\in\mathcal{R}$.  
\end{lemma}
\begin{proof}
	Let $f\in\mathfrak{R}$ be the underlying distribution for $\vect{\mathcal{M}}\in\cR$. Then, from Eq.~\eqref{eq:ConservedToPrimitive}, the components of $\vect{\mathcal{U}}$ can be written as
	\begin{equation}
		\big(\mathcal{N}, \mathcal{G}_{j}\big)^{\intercal} = \f{1}{4\pi}\int_{\mathbb{S}^{2}}\big(\,1+v^{i}\,\ell_{i}(\omega)\,\big)\,f(\omega)\,\big(1, \ell_{j}(\omega)\big)^{\intercal}\,d\omega:=\f{1}{4\pi}\int_{\mathbb{S}^{2}}\mathsf{f}(\omega)\,\big(1, \ell_{j}(\omega)\big)^{\intercal}\,d\omega.
	\end{equation}
	Since $f\in\mathfrak{R}$ and $v^{i}\,\ell_{i}\in(-1,1)$, it follows that $\mathsf{f}(\omega):=\big(\,1+v^{i}\,\ell_{i}(\omega)\,\big)\,f(\omega)\in\mathfrak{R}$ and thus $\vect{\mathcal{U}}\in\cR$.
\end{proof}

\begin{lemma}\label{lem:BackwardRealizability}
	Suppose Assumption~\ref{assum:exact_closure} holds, $v<1$, and $\vect{\mathcal{U}}\in\cR$.
	Then there exists some $\vect{\mathcal{M}}\in\mathcal{R}$ that satisfies Eq.~\eqref{eq:ConservedToPrimitive}.  
\end{lemma}
\begin{proof}
	Let $\mathsf{f}\in\mathfrak{R}$ denote the underlying distribution for $\vect{\mathcal{U}}\in\cR$. Then the components of $\vect{\mathcal{U}}$ can be written as
	\begin{equation}
	\big(\mathcal{N}, \mathcal{G}_{j}\big)^{\intercal} = \f{1}{4\pi}\int_{\mathbb{S}^{2}}\mathsf{f}(\omega)\,\big(1, \ell_{j}(\omega)\big)^{\intercal}\,d\omega.
\end{equation}
	Since $\mathsf{f}\in\mathfrak{R}$ and $v^{i}\,\ell_{i}\in(-1,1)$, it follows that $f(\omega):=\big(\,1+v^{i}\,\ell_{i}(\omega)\,\big)^{-1}\,\mathsf{f}(\omega)\in\mathfrak{R}$.
	Taking the moments of $f$ leads to $\vect{\mathcal{M}}\in\cR$. Using the relation between $f$ and $\mathsf{f}$ it is then straightforward to verify that $\vect{\mathcal{M}}$ satisfies Eq.~\eqref{eq:ConservedToPrimitive}.
\end{proof}

Lemma~\ref{lem:BackwardRealizability} shows the existence of realizable primitive moments corresponding to given conserved moments.  
However, it does not provide guarantees on the convergence of the iterative solver we use to find the primitive moments. 
In the remainder of this subsection, we prove that the iterative solver in Eq.~\eqref{eq:Picard} guarantees the convergence to a realizable moment $\bcM$.
To start, in the following lemma we show that realizability is guaranteed at each iteration of the solver in Eq.~\eqref{eq:Picard}.

\begin{lemma}\label{lem:solver_realizability}
	Let $\vect{\mathcal{U}}\in\mathcal{R}$ and $\lambda \leq \frac{1}{1+v}$ in Eq.~\eqref{eq:richardson_fixed_pt}. 
	Then, the solver in Eq.~\eqref{eq:Picard} guarantees that $\vect{\mathcal{M}}^{[k+1]}=(\cD^{[k+1]},\bcI^{[k+1]})^{\intercal}\in\mathcal{R}$, provided that $\vect{\mathcal{M}}^{[k]}=(\cD^{[k]},\bcI^{[k]})^{\intercal}\in\mathcal{R}$.
\end{lemma}

\begin{proof}
	We write the iterative update in Eq.~\eqref{eq:Picard} as
	\begin{equation}
		\begin{alignedat}{2}
			\vect{\mathcal{M}}^{[k+1]} = \left(
			\begin{array}{c}
				\mathcal{D}^{[k+1]}  \\
				\mathcal{I}_{j}^{[k+1]}
			\end{array}
			\right)	
			&= (1-\lambda)
			\left(\begin{array}{c}
				\mathcal{D}^{[k]} - \frac{\lambda}{1-\lambda} v^{i}\mathcal{I}_{i}^{[k]}\\
				\mathcal{I}_{j}^{[k]} - \frac{\lambda}{1-\lambda} v^{i}\mathsf{k}_{ij}^{[k]}\mathcal{D}^{[k]}
			\end{array}\right) + 
			\lambda 
			\left(
			\begin{array}{c}
				\mathcal{N}\\
				\mathcal{G}_{j}
			\end{array}\right)\\ &=: (1-\lambda)\,\widetilde{\vect{\mathcal{M}}}^{[k]} + \lambda\, {\vect{\mathcal{U}}}.
		\end{alignedat}
	\end{equation}
	Since the realizable set $\cR$ is convex and ${\vect{\mathcal{M}}}^{[k+1]}$ is a convex combination of $\widetilde{\vect{\mathcal{M}}}^{[k]}$ and $\vect{\mathcal{U}}\in\cR$, it suffices to show that $\widetilde{\vect{\mathcal{M}}}^{[k]}\in\cR$.
	We observe that the entries in $\widetilde{\vect{\mathcal{M}}}^{[k]}$ takes the exact same form as the ones on the right-hand side of Eq.~\eqref{eq:ConservedToPrimitive}, except with $\vect{v}$ replaced by $-\frac{\lambda}{1-\lambda} \vect{v}$. It then follows from Lemma~\ref{lem:ForwardRealizability} that $\widetilde{\vect{\mathcal{M}}}^{[k]}\in\cR$ if $\frac{\lambda}{1-\lambda} v \leq1$, i.e., $\lambda\leq\frac{1}{1+v}$.
\end{proof}

It is well-known that, when solving a fixed-point problem defined by a contraction operator, the Picard iteration converges to the unique fixed point (see, e.g., \cite{hairer1993solving}). 
We show below in Proposition~\ref{prop:contraction} that the fixed-point operator $\bcH_{\bcU}$ defined in Eq.~\eqref{eq:richardson_fixed_pt} is a contraction under mild assumptions on $v_h$, which thus guarantees the convergence of the iterative solver in Eq.~\eqref{eq:Picard}.
The proof of Proposition~\ref{prop:contraction} uses results from the following two technical lemmas.

\begin{lemma}\label{lemma:dD_term}
	For any $\bcM\in\cR$,  $\| \partial_{\cD}(v^{i}\mathsf{k}_{ij}\cD)\|  \leq v$.
\end{lemma}
\begin{proof}
	See \ref{sec:proof_of_dD} for the proof.
\end{proof}

\begin{lemma}\label{lemma:dI_term}
	For any $\bcM\in\cR$,  $\| \nabla_{\bcI}(v^{i}\mathsf{k}_{ij}\cD) \|   \leq 2v$.
\end{lemma}
\begin{proof}
	See \ref{sec:proof_of_dI} for the proof.
\end{proof}

We now state and prove Proposition~\ref{prop:contraction}. 

\begin{prop}\label{prop:contraction}
	Suppose $v<\sqrt{2}-1$ and $\lambda\in(0,1]$. Then, $\bcH_{\bcU}$ defined in Eq.~\eqref{eq:richardson_fixed_pt} is a contraction operator, i.e., there exists some $L<1$ such that 
	\begin{equation}
		\|\bcH_{\bcU}(\Mone) - \bcH_{\bcU}(\Mtwo) \| 	\leq L \| \Mone - \Mtwo \|\:, \quad \forall \Mone, \Mtwo \in \cR\:.
	\end{equation}
\end{prop}

\begin{proof}
	First, for convenience, we denote $\Delta \cD = \Done - \Dtwo $ and $\Delta \cI_j = \Ione_j - \Itwo_j$.
	It then follows from the definition of $\bcH_{\bcU}$ and the triangle inequality that	
	\begin{equation}
		\|\bcH_{\bcU}(\Mone) - \bcH_{\bcU}(\Mtwo)\| 
		\leq (1-\lambda)\left\|\left(\begin{array}{l}
			\Delta \cD \\
			\Delta \cI_j
		\end{array}\right)\right\| +
		\lambda
		\left\|\left(\begin{array}{c}
			v^{i}\,\Delta \cI_i\\
			v^{i}(\Kone_{ij}\Done - \Ktwo_{ij}\Dtwo )
		\end{array}\right)\right\| 
	\end{equation}
	Thus, it suffices to show that, there exists some $\tilde{L}<1$ such that  
	\begin{equation}\label{eq:contraction_proof_1}
		\left\|\left(\begin{array}{c}
			v^{i}\,\Delta \cI_i\\
			v^{i}(\Kone_{ij}\Done - \Ktwo_{ij}\Dtwo )
		\end{array}\right)\right\| 
		\leq \tilde{L}
		\left\|\left(\begin{array}{l}	\Delta \cD \\	\Delta \cI_j \end{array}\right)\right\|,\quad
		\forall \Mone, \Mtwo \in \cR. 
	\end{equation}
	Lemmas~\ref{lemma:dD_term} and \ref{lemma:dI_term} imply that the gradients of $v^{i}\mathsf{k}_{ij}\cD$ in the $\cD$ and $\bcI$ directions are bounded. Thus, we have
	\begin{equation}
		\begin{alignedat}{2}
			\| v^{i}(\Kone_{ij}\Done - \Ktwo_{ij}\Dtwo ) \| 
			&\leq 
			\| \partial_{\cD}(v^{i}\mathsf{k}_{ij}\cD)\|    \| \Delta \cD \| + 
			\| \nabla_{\bcI}(v^{i}\mathsf{k}_{ij}\cD) \| 	\| \Delta \cI_j \|\\
			&\leq v \| \Delta \cD \| +  2v \| \Delta \cI_j \|\:,
		\end{alignedat}
	\end{equation}
	which leads to
	\begin{equation}\label{eq:vkD_bound}
		\begin{alignedat}{2}
		\| v^{i}(\Kone_{ij}\Done - \Ktwo_{ij}\Dtwo ) \|^2 
		&\leq v^2 \| \Delta \cD \|^2 + 4v^2 \| \Delta \cD \| \|\Delta \cI_j \| +  4v^2 \|\Delta \cI_j \|^2\\
		&\leq (3+2\sqrt{2}) v^2 \| \Delta \cD \|^2 + (2+2\sqrt{2})v^2 \|\Delta \cI_j \|^2\:,
		\end{alignedat}
	\end{equation}
	where the second inequality follows from the inequality, $2ab \leq (\sqrt{2}+1) a^2 + (\sqrt{2}-1) b^2$, with $a=\sqrt{2}v\|\Delta\cD\|$ and $b=\sqrt{2}v\|\Delta \cI_j\|$.
Taking the square of the left-hand side in Eq.~\eqref{eq:contraction_proof_1} and applying the inequality in Eq.~\eqref{eq:vkD_bound} gives
	\begin{equation}
		\begin{alignedat}{2}
			\left\|\left(\begin{array}{c}
				v^{i}\,\Delta \cI_i\\
				v^{i}(\Kone_{ij}\Done - \Ktwo_{ij}\Dtwo )
			\end{array}\right)\right\| ^2 &=  v^2 \| \Delta \cI_j \|^2 + \| v^{i}(\Kone_{ij}\Done - \Ktwo_{ij}\Dtwo ) \|^2\\
			&\leq (3+2\sqrt{2}) v^2 (\| \Delta \cD \|^2 + \| \Delta \cI_j \|^2)\:.
		\end{alignedat}
	\end{equation}
	Let $\tilde{L}:=\sqrt{3+2\sqrt{2}} \, v$, the claim then holds when 
	$v < \big(\sqrt{3+2\sqrt{2}}\,\big)^{-1} = \sqrt{2}-1$.
\end{proof}

\begin{theorem}\label{thm:convergence}
	Suppose $v<\sqrt{2}-1$ and $\lambda \leq \frac{1}{1+v}$ in Eq.~\eqref{eq:richardson_fixed_pt}.
	Then, for any given conserved moment $\bcU=(\cN,\bcG)^{\intercal}\in\cR$ and initial primitive moment $\bcM^{[0]}\in\cR$, the iterative solver in Eq.~\eqref{eq:Picard} converges to the unique realizable primitive moment $\bcM$ that satisfies Eq.~\eqref{eq:ConservedToPrimitive} as $k\to\infty$.
\end{theorem}
\begin{proof}
	This theorem is a direct consequence from the realizability-preserving property of the solver shown in Lemma~\ref{lem:solver_realizability} and the contraction property of $\bcH_{\bcU}$ proved in Proposition~\ref{prop:contraction}.
\end{proof}

The results in Theorem~\ref{thm:convergence} lead to the following corollary on the uniqueness of realizable primitive moments associated with realizable conserved moments.

\begin{corollary}\label{coro:UtoM}
	Suppose $v<\sqrt{2}-1$.  
	For any conserved moment $\vect{\mathcal{U}}\in\cR$, there exists a unique realizable primitive moment $\vect{\mathcal{M}}$ that satisfies Eq.~\eqref{eq:ConservedToPrimitive}.  
\end{corollary}

\subsection{Implicit Collision Update}
\label{sec:collision}

In this section, we prove Proposition~\ref{prop:realizabilityImplicit}, which states that, under Assumption~\ref{assum:exact_closure}, the implicit update in Eq.~\eqref{eq:imexBackwardEuler} preserves realizability when the iterative solver in Eq.~\eqref{eq:collision_Picard} is used. 
We first show in the following lemma that realizability is preserved in each iteration when starting from a realizable moment.

\begin{lemma}\label{lem:collision_solver_realizability}
	Let $\vect{\mathcal{U}}^{\old}=(\cN^{\old},\bcG^{\old})^{\intercal}\in\mathcal{R}$, $\lambda \leq \frac{1}{1+v}$, and $\kappa\geq\chi\geq0$ in Eq.~\eqref{eq:collision_fixed_point1}. 
	Then, the solver in Eq.~\eqref{eq:collision_Picard} guarantees that $\vect{\mathcal{M}}^{[k+1]}=(\cD^{[k+1]},\bcI^{[k+1]})^{\intercal}\in\mathcal{R}$, provided that $\vect{\mathcal{M}}^{[k]}=(\cD^{[k]},\bcI^{[k]})^{\intercal}\in\mathcal{R}$.
\end{lemma}

\begin{proof}
	We follow the approach in the proof of Lemma~\ref{lem:solver_realizability} and write Eq.~\eqref{eq:collision_Picard} as
	\begin{equation}
		\begin{alignedat}{2}
			\vect{\mathcal{M}}^{[k+1]} 
			&= (1-\lambda)\,\Lambda
			\left(\begin{array}{c}
				\mathcal{D}^{[k]} - \frac{\lambda}{1-\lambda} v^{i}\mathcal{I}_{i}^{[k]}\\
				\mathcal{I}_{j}^{[k]} - \frac{\lambda}{1-\lambda} v^{i}\mathsf{k}_{ij}^{[k]}\mathcal{D}^{[k]}
			\end{array}\right) + 
			\lambda\, \Lambda
			\left(
			\begin{array}{c}
				\mathcal{N}^{\old}+ \dt \chi\,\cD_{0}\\
				\mathcal{G}_{j}^{\old}
			\end{array}\right)\\ &=: (1-\lambda)\,\Lambda\,\tilde{\vect{\mathcal{M}}}^{[k]} + \lambda\,\Lambda\, \tilde{\vect{\mathcal{U}}}^{\old}.
		\end{alignedat}
	\end{equation}
 In the proof of Lemma~\ref{lem:solver_realizability}, we have shown that $\tilde{\vect{\mathcal{M}}}^{[k]}\in\cR$ when $\lambda\leq\frac{1}{1+v}$. Also, it is clear that $\tilde{\bcU}^{\old}\in\cR$ because ${\bcU}^{\old}\in\cR$, $\chi\geq0$ and $\cD_0\geq0$. Since $\kappa\geq\chi\geq0$, we have $\mu_{\chi}\geq\mu_{\kappa}\geq0$, implying that $\Lambda\,\bcM\in\cR$ for all $\bcM\in\cR$ based on the definition of $\cR$ in Eq.~\eqref{eq:realizableSet}.
	Therefore, $\Lambda\,\tilde{\vect{\mathcal{M}}}^{[k]}$ and $\Lambda\, \tilde{\vect{\mathcal{U}}}^{\old}$ are both realizable, which, together with the convexity of $\cR$, completes the proof.
\end{proof}

Similar to the moment conversion problem considered in Section~\ref{sec:momentConversionRealizability}, we next show in the following proposition that the operator $\bcQ$ in Eq.~\eqref{eq:collision_fixed_point1} is a contraction, which implies convergence of the Picard iteration method in Eq.~\eqref{eq:collision_Picard}.

\begin{prop}\label{prop:collision_contraction}
	Suppose $v<\sqrt{2}-1$, $\lambda\in(0,1]$, and $\kappa\geq\chi\geq0$. Then, $\bcQ$ is a contraction operator, i.e., there exists some $L<1$ such that 
	\begin{equation}
		\|\bcQ(\Mone) - \bcQ(\Mtwo) \| 	\leq L \| \Mone - \Mtwo \|\:, \quad \forall \Mone, \Mtwo \in \cR\:.
	\end{equation}
\end{prop}
\begin{proof}
	From the definitions of $\bcH_{\bcU}$ and $\bcQ$ in Eqs.~\eqref{eq:richardson_fixed_pt} and \eqref{eq:collision_fixed_point1}, we observe that 
	\begin{equation}
		\begin{alignedat}{2}
		\|\bcQ(\Mone) - \bcQ(\Mtwo) \| &= \|\Lambda(\bcH_{\bcU}(\Mone) - \bcH_{\bcU}(\Mtwo)) \|\\
		&\leq \|\Lambda\|\|\bcH_{\bcU}(\Mone) - \bcH_{\bcU}(\Mtwo) \|
		\end{alignedat}
	\end{equation}
	for all $\Mone, \Mtwo \in \cR$. Since $\kappa\geq\chi\geq0$ and $\lambda>0$, we have $0\leq\mu_{\kappa}\leq \mu_{\chi} \leq 1$, i.e., $\|\Lambda\|\leq 1$. The claim is thus a direct consequence of Proposition~\ref{prop:contraction}. 
\end{proof}

\begin{theorem}\label{thm:collision_convergence}
	Suppose $v<\sqrt{2}-1$, $\lambda \leq \frac{1}{1+v}$, and $\kappa\geq\chi\geq0$ in Eq.~\eqref{eq:collision_fixed_point1}.
	Then, for any given conserved moment $\bcU^{\old}=(\cN^{\old},\bcG^{\old})^{\intercal}\in\cR$ and initial primitive moment $\bcM^{[0]}\in\cR$, the iterative solver in Eq.~\eqref{eq:collision_Picard} converges to a unique realizable primitive moment $\bcM$ as $k\to\infty$. Further, the conserved moment $\bcU$ associated to $\bcM$ via Eq.~\eqref{eq:ConservedToPrimitive} is also realizable and solves the implicit system in Eq.~\eqref{eq:ImplicitSystem}.
\end{theorem}
\begin{proof}
	The convergence to a unique $\bcM\in\cR$ is given by the realizability-preserving property in Lemma~\ref{lem:collision_solver_realizability} and the contraction property in Proposition~\ref{prop:collision_contraction}. Realizability of $\bcU$ follows from Lemma~\ref{lem:ForwardRealizability}, and the formulation of the fixed-point problem in Eq.~\eqref{eq:collision_fixed_point1} guarantees that $\bcU$ solves the implicit system in Eq.~\eqref{eq:ImplicitSystem}.
\end{proof}

\subsection{Extension to Approximate Moment Closures}
\label{sec:approxClosure}

In the earlier sections, we have shown the realizability-preserving property of the numerical scheme \eqref{eq:imexForwardEuler}--\eqref{eq:imexBackwardEuler} under Assumption~\ref{assum:exact_closure}, in which the use of exact moment closures is assumed.
As discussed in Sections~\ref{sec:model} and \ref{sec:closure}, the approximate Minerbo closure is often used in practice to reduce the computational cost, where the approximate Eddington factor $\psi_{\mathsf{a}}$ and heat-flux factor $\zeta_{\mathsf{a}}$, defined respectively in Eqs.~\eqref{eq:psiApproximate} and \eqref{eq:zetaApproximate}, are considered.
In this section, we show that the realizability-preserving and convergence analyses for the conserved-to-primitive moment conversion (Eq.~\eqref{eq:ConservedToPrimitive}) and the implicit update (Eq.~\eqref{eq:imexBackwardEuler}) given in Sections~\ref{sec:momentConversionRealizability} and \ref{sec:collision} can be extended to the case when the approximate Minerbo closure is used.

When the approximate Eddington factor $\psi_{\mathsf{a}}$ in Eq.~\eqref{eq:psiApproximate} is used, we replace Lemma~\ref{lem:ForwardRealizability} with the following lemma.

\begin{lemma}\label{lemma:MtoU}
		Suppose $v<1$. Let $\vect{\mathcal{U}}$ be given as in Eq.~\eqref{eq:ConservedToPrimitive} with $\vect{\mathcal{M}}\in\mathcal{R}$,
	then $\vect{\mathcal{U}}\in\mathcal{R}$.  
\end{lemma}

\begin{proof}
	Since $\vect{\mathcal{M}}=\big(\mathcal{D},\vect{\mathcal{I}}\big)^{\intercal}\in\mathcal{R}$, we know that $\mathcal{D}>0$ and $\mathcal{D}-\mathcal{I}\geq0$.
	To show $\vect{\mathcal{U}}=\big(\mathcal{N},\vect{\mathcal{G}}\big)^{\intercal}\in\mathcal{R}$, we first prove $\mathcal{N}>0$.
	By definition, 
	\begin{equation}
		\mathcal{N} = \mathcal{D} + v^{i} \mathcal{I}_{i} \geq \mathcal{D} - v {\mathcal{I}} > \mathcal{D} - \mathcal{I} \geq0\:,
	\end{equation}	
	where the Cauchy-Schwartz inequality and the assumption that $v<1$ are used.
	We next prove that $\mathcal{N}^2 - \mathcal{G}^2 \geq0$ with $\mathcal{G} := |\vect{\mathcal{G}}|$, which implies $\mathcal{N} - \mathcal{G} \geq0$.	
	Writing $\mathcal{N}^2$ and $\mathcal{G}^2$ in terms of the primitive moments leads to
	\begin{equation}
		\begin{alignedat}{2}
			\mathcal{N}^2 &= \mathcal{D}^2 + 2 (v^{i} \mathcal{I}_{i}) \mathcal{D} + (v^{i} \mathcal{I}_{i})^2\:,\\
			&= \mathcal{D}^2 \big( 1 + 2\, v^{i}\, \hat{n}_{i} h + (v^{i} \,\hat{n}_{i})^2 h^2 \big)\:,\\
			\mathcal{G}^2 
			&= \mathcal{I}^2 +  2 \mathcal{I}^{j}(v^{i}\,\mathsf{k}_{ij})\mathcal{D} + (v_{\ell}\,\mathsf{k}^{\ell j})(v^{i}\,\mathsf{k}_{ij}) \mathcal{D}^2\:.\\
			&= \mathcal{D}^2 \big( h^2 +  2\, \hat{n}^{j} \,v^{i}\,\mathsf{k}_{ij} h+ v_{\ell}\,\mathsf{k}^{\ell j}\,v^{i}\,\mathsf{k}_{ij}\, \big)\:.
		\end{alignedat}
		\label{eq:NormalizedNG}
	\end{equation}
	Using the definition of $\mathsf{k}_{ij}$ in Eq.~\eqref{eq:VariableEddingtonTensor} we obtain
	\begin{align}
		\hat{n}^{j} \,v^{i}\,\mathsf{k}_{ij} &= \hat{n}^{j} \,v^{i}\,\frac{1}{2} \big( (1-\psi_{\mathsf{a}}) \delta_{ij} + (3\psi_{\mathsf{a}}-1) \hat{n}_i \hat{n}_j\big) = \psi_{\mathsf{a}}\,v^{i}\,\hat{n}_i\:,\\
		v_{\ell}\,\mathsf{k}^{\ell j}\,v^{i}\,\mathsf{k}_{ij} 
		&= \frac{1}{4} \big((1-\psi_{\mathsf{a}})^2 v^2 + (1+\psi_{\mathsf{a}})(3\psi_{\mathsf{a}} - 1) (v_i\,\hat{n}^{i})^2 \big)\:.
	\end{align}
	Plugging these terms into Eq.~\eqref{eq:NormalizedNG}, denoting $s:= v^i \hat{n}_i$, and using the assumption that $v<1$ leads to a sufficient condition for $\mathcal{N}^2 - \mathcal{G}^2 \geq0$: $\forall s\in[-1,1]$ and $\forall h \in [0, 1]$,
	\begin{equation}
		\big(1 - h^2  - \frac{1}{4} (1-\psi_{\mathsf{a}})^2 \big) + 2 (1-\psi_{\mathsf{a}})s h + \big( h^2- \frac{1}{4}  (1+\psi_{\mathsf{a}})(3\psi_{\mathsf{a}} - 1)  \big) s^2 \geq0\:.
	\end{equation}
	From Lemma~\ref{lemma:polynomial_bounds}~(e), we have $1-\psi_{\mathsf{a}}\geq0$. Thus, by applying the inequality $2sh\geq -1-s^2h^2$ to the second term above, it suffices to show that
	\begin{equation}
		\big[\psi_{\mathsf{a}} - h^2  - \frac{1}{4} (1-\psi_{\mathsf{a}})^2 \big] +\big[ \psi_{\mathsf{a}} h^2- \frac{1}{4}  (1+\psi_{\mathsf{a}})(3\psi_{\mathsf{a}} - 1)  \big] s^2 \geq0\:.
	\end{equation}
	Since $\psi_{\mathsf{a}} - h^2  - \frac{1}{4} (1-\psi_{\mathsf{a}})^2\geq0$ (Lemma~\ref{lemma:polynomial_bounds}~(f)) and $s^2\in[0,1]$,
	\begin{equation}
		\big[\psi_{\mathsf{a}} - h^2  - \frac{1}{4} (1-\psi_{\mathsf{a}})^2  + \psi_{\mathsf{a}} h^2- \frac{1}{4}  (1+\psi_{\mathsf{a}})(3\psi_{\mathsf{a}} - 1)  \big] s^2 \geq0\:,
	\end{equation}
	which then becomes 
	\begin{equation}
		(\psi_{\mathsf{a}} - h^2)(1-\psi_{\mathsf{a}}) s^2\geq0\:.
	\end{equation}
	With $h^2\leq\psi_{\mathsf{a}}\leq1$ from Lemma~\ref{lemma:polynomial_bounds}~(e), the proof is complete.
\end{proof}

Lemma~\ref{lemma:MtoU} extends Lemma~\ref{lem:ForwardRealizability} by showing that the mapping from primitive moments to conserved moments via Eq.~\eqref{eq:ConservedToPrimitive} preserves realizability even when the approximate Minerbo closure is used. 
However, there is not an analogous extension of Lemma~\ref{lem:BackwardRealizability} to the case of approximate closures.
To show that the map from conserved to primitive moments is realizability-preserving with the approximate closure, we verify that the analysis from Lemma~\ref{lem:solver_realizability} to Corollary~\ref{coro:UtoM} is still valid when the approximate closure is considered.
Specifically, when $\psi$ is replaced by $\psi_{\mathsf{a}}$, the result of Lemma~\ref{lem:solver_realizability} can be obtained by invoking Lemma~\ref{lemma:MtoU} rather than Lemma~\ref{lem:ForwardRealizability} in the proof, the results of Lemmas~\ref{lemma:dD_term} and \ref{lemma:dI_term} hold since it is shown in \ref{sec:polynomial_bounds} that $\psi_{\mathsf{a}}$ also satisfies the required properties of $\psi$, and the remainder of the analysis stays identical to the exact closure case considered in Section~\ref{sec:momentConversionRealizability}.
Therefore, we have shown that, when $\psi$ is replaced by $\psi_{\mathsf{a}}$, the iterative solver in Eq.~\eqref{eq:Picard} converges to the unique, realizable primitive moment that satisfies Eq.~\eqref{eq:ConservedToPrimitive} for the given conserved moment, which implies that the conserved to primitive moment map from Eq.~\eqref{eq:ConservedToPrimitive} still preserves realizability when $v<\sqrt{2}-1$.
Further, we also verified that the convergence and realizability-preserving properties of the iterative solver in Eq.~\eqref{eq:collision_Picard} for the implicit system in Eq.~\eqref{eq:collision_fixed_point1} given in Theorem~\ref{thm:collision_convergence} also hold in the approximate closure case by applying the same arguments to the analysis in Section~\ref{sec:collision}.

\section{Conservation Property}
\label{sec:conservation}

\subsection{Simultaneous Number and Energy Conservation of the DG Scheme}
\label{sec:dgConservation}

It has been shown in Proposition~\ref{prop:EnergyandMomentumConservation} that the two-moment model in Eqs.~\eqref{eq:spectralNumberEquation}--\eqref{eq:spectralNumberFluxEquation} conserves the Eulerian-frame energy up to $\mathcal{O}(v)$. 
In this section, we discuss the simultaneous Eulerian-frame number and energy conservation properties of the two-moment model with the discontinuous Galerkin phase-space discretization presented in Section~\ref{sec:dgMethod}.  
We are primarily concerned with consistency with the Eulerian-frame energy equation for the phase-space advection problem.  
For this reason, we consider the collisionless case.  

Eulerian-frame number conservation follows from the first component of the semi-discrete DG scheme in Eq.~\eqref{eq:dgSemiDiscrete} (treating the general case with $d_{\bx}=3$), 
\begin{align}
	&\big(\,\p_{t}\cN_{h},\varphi_{h}\,\big)_{\bK}
	=-\sum_{i=1}^{3}\int_{\tilde{\vect{K}}^{i}}
  	\Big[\,
    		\widehat{\cF_{\cN}^{i}}\big(\vect{\mathcal{U}}_{h},\vect{v}_{h}\big)\,\varphi_{h}|_{x_{\hi}^{i}}
    		-\widehat{\cF_{\cN}^{i}}\big(\vect{\mathcal{U}}_{h},\vect{v}_{h}\big)\,\varphi_{h}|_{x_{\lo}^{i}}
  	\,\Big]\,\tau\,d\tilde{\vect{z}}^{i} 
	+\sum_{i=1}^{3}\big(\,\cF_{\cN}^{i}(\vect{\mathcal{U}}_{h},\vect{v}_{h}),\pd{\varphi_{h}}{i}\,\big)_{\vect{K}} \nonumber \\
	&\hspace{36pt}
	-\int_{\vect{K}_{\vect{x}}}
  	\Big[\,
    		\varepsilon^{3}\,\widehat{\cF_{\cN}^{\varepsilon}}\big(\vect{\mathcal{U}}_{h},\vect{v}_{h}\big)\,\varphi_{h}|_{\varepsilon_{\hi}}
    		-\varepsilon^{3}\,\widehat{\cF_{\cN}^{\varepsilon}}\big(\vect{\mathcal{U}}_{h},\vect{v}_{h}\big)\,\varphi_{h}|_{\varepsilon_{\lo}}
  	\,\Big]\,d\vect{x} 
  	+\big(\,\varepsilon\,\cF_{\cN}^{\varepsilon}(\vect{\mathcal{U}}_{h},\vect{v}_{h}),\pd{\varphi_{h}}{\varepsilon}\,\big)_{\vect{K}},
	\label{eq:dgSemiDiscreteNumber}
\end{align}
where $\cF_{\cN}^{i}$ and $\cF_{\cN}^{\varepsilon}$, respectively, are the first component of the position and energy space fluxes, defined in Eq.~\eqref{eq:phaseSpaceFluxes}, and $\widehat{\cF_{\cN}^{i}}$ and $\widehat{\cF_{\cN}^{\varepsilon}}$ are the corresponding numerical fluxes, defined in Eqs.~\eqref{eq:numericalFluxPosition} and \eqref{eq:numericalFluxEnergy}.  
Setting $\varphi_h=1$ as the test function in Eq.~\eqref{eq:dgSemiDiscreteNumber} results in the equation for the element-integrated Eulerian-frame number density.  
Note that the volume terms (the second and fourth terms) on the right-hand side of Eq.~\eqref{eq:dgSemiDiscreteNumber} vanish when $\varphi_{h}=1$.  
Then, because the numerical fluxes $\widehat{\cF_{\cN}^{i}}$ and $\widehat{\cF_{\cN}^{\varepsilon}}$ are continuous on element interfaces, summation over all phase-space elements $\bK\in D$ results in cancellation of all interior fluxes, and the resulting rate of change in the total Eulerian-frame particle number is only due to the flow of particles though the boundary of the domain $D$.  
That is, the DG scheme for the Eulerian-frame particle number is conservative by construction.  

As for Eulerian-frame energy conservation, similar to Eq.~\eqref{eq:energyEquationEulerian} in Proposition~\ref{prop:EnergyandMomentumConservation}, the element-integrated Eulerian-frame energy equation can be derived by adding the Eulerian-frame number equation in Eq.~\eqref{eq:dgSemiDiscreteNumber}, with $\varphi_h = \varepsilon$, and the sum of the three number flux equations in Eq.~\eqref{eq:dgSemiDiscrete} with test functions $\varphi_h = \varepsilon v_h^{j}$.  
To accommodate this choice of test functions, the approximation space $\mathbb{V}_h^k$ must include the piecewise linear function in the energy dimension, i.e., $k\geq1$.  
Let $\cE_h:=\varepsilon\, ( \mathcal{N}_h +  v^{j}_h \,{\cG}_{h,j} )$ denote the discretized Eulerian-frame energy density.  
Then, the resulting equation for the element-integrated Eulerian-frame energy takes the form
\begin{align}
	&\big(\,\p_{t}\cE_{h}\,\big)_{\bK} := \big(\,\p_{t}\cN_{h},\varepsilon\,\big)_{\bK} + \big(\,\p_{t}\cG_{j,h},\varepsilon v_{h}^{j}\,\big)_{\bK} \nonumber \\
	&=-\sum_{i=1}^{3}\int_{\tilde{\vect{K}}^{i}}
  	\Big[\,
    		\widehat{\cF_{\cE}^{i}}\big(\vect{\mathcal{U}}_{h},\vect{v}_{h}\big)|_{x_{\hi}^{i}}
    		-\widehat{\cF_{\cE}^{i}}\big(\vect{\mathcal{U}}_{h},\vect{v}_{h}\big)|_{x_{\lo}^{i}}
  	\,\Big]\,\tau\,d\tilde{\vect{z}}^{i}
	-\int_{\vect{K}_{\vect{x}}}
  	\Big[\,
    		\varepsilon^{3}\,\widehat{\cF_{\cE}^{\varepsilon}}\big(\vect{\mathcal{U}}_{h},\vect{v}_{h}\big)|_{\varepsilon_{\hi}}
    		-\varepsilon^{3}\,\widehat{\cF_{\cE}^{\varepsilon}}\big(\vect{\mathcal{U}}_{h},\vect{v}_{h}\big)|_{\varepsilon_{\lo}}
  	\,\Big]\,d\vect{x} \nonumber \\
  	&\hspace{36pt}
  	+\big(\,\varepsilon\,\cF_{\cN}^{\varepsilon}(\vect{\mathcal{U}}_{h},\vect{v}_{h})\,\big)_{\vect{K}}
	+\sum_{i=1}^{3}\big(\,\cF_{\cG_{j}}^{i}(\vect{\mathcal{U}}_{h},\vect{v}_{h}),\varepsilon\pd{v_{h}^{j}}{i}\,\big)_{\vect{K}}
	+\cO(v_{h}^{2}),
	\label{eq:dgSemiDiscreteEnergy}
\end{align}
where we have defined the position space numerical fluxes,
\begin{equation}
	\widehat{\cF_{\cE}^{i}}\big(\vect{\mathcal{U}}_{h},\vect{v}_{h}\big)|_{x_{\hi/\lo}^{i}}
	=\varepsilon\,\big[\,\widehat{\cF_{\cN}^{i}}\big(\vect{\mathcal{U}}_{h},\vect{v}_{h}\big) + v_{h}^{j}\,\widehat{\cF_{\cG_{j}}^{i}}\big(\vect{\mathcal{U}}_{h},\vect{v}_{h}\big)\,\big]|_{x_{\hi/\lo}^{i}},
	\label{eq:numericalFluxEulerianEnergy}
\end{equation}
the energy space numerical fluxes, 
\begin{equation}
	\widehat{\cF_{\cE}^{\varepsilon}}\big(\vect{\mathcal{U}}_{h},\vect{v}_{h}\big)|_{\varepsilon_{\hi/\lo}}
	=\varepsilon\,\widehat{\cF_{\cN}^{\varepsilon}}\big(\vect{\mathcal{U}}_{h},\vect{v}_{h}\big)|_{\varepsilon_{\hi/\lo}}, 
\end{equation}
and $v_h^{2}:=|\vect{v}_h|^{2}$.  
Here, $\cF_{\cG_{j}}^{i}$ and $\widehat{\cF_{\cG_{j}}^{i}}$ represent the fluxes and the corresponding numerical fluxes for the number flux equation, defined in Eqs.~\eqref{eq:phaseSpaceFluxes} and \eqref{eq:numericalFluxPosition}, respectively.  
The third and fourth term on the right-hand side of Eq.~\eqref{eq:dgSemiDiscreteEnergy}, which emanate from the energy derivative term of the number equation and the spatial derivative of the number flux equations, respectively, can be written as
\begin{equation}
	\big(\,\varepsilon\,\cF_{\cN}^{\varepsilon}(\vect{\mathcal{U}}_{h},\vect{v}_{h})\,\big)_{\vect{K}}
	+\sum_{i=1}^{3}\big(\,\cF_{\cG_{j}}^{i}(\vect{\mathcal{U}}_{h},\vect{v}_{h}),\varepsilon\pd{v_{h}^{j}}{i}\,\big)_{\vect{K}}
	=\int_{\bK}\varepsilon\cK^{i}_{\hspace{2pt}j}\,\big[\,\p_{i}v_{h}^{j}-(\p_{i}v^{j})_{h}\,\big]\,\tau d\bz + \cO(v_{h}^{2}),
	\label{eq:cancellations}
\end{equation}
where $(\p_{i}v^{j})_h$ is the discretized velocity derivative that satisfies Eq.~\eqref{eq:velocityDerivatives}.  

Provided (i) the numerical flux in Eq.~\eqref{eq:numericalFluxEulerianEnergy} is uniquely defined on element interfaces and (ii) the first term on the right-hand side of Eq.~\eqref{eq:cancellations} vanishes, Eq.~\eqref{eq:dgSemiDiscreteEnergy} is, to $\cO(v_{h}^{2})$, a phase-space conservation law for the element-integrated Eulerian-frame energy, in accordance with Proposition~\ref{prop:EnergyandMomentumConservation}.  
These requirements --- which generally require the discrete velocity $\vect{v}_h$ to be continuous across the elements --- are not satisfied exactly by the DG scheme proposed here. 
Since the components of $\vect{v}_h$ are represented by piecewise polynomials, which are discontinuous on element boundaries, the violation in Eulerian-frame energy conservation may be larger than what would be expected from $\cO(v_h^2)$ contributions alone.  
We will investigate the simultaneous conservation of Eulerian-frame number and energy numerically in Section~\ref{sec:numericalResults}.  

\subsection{Energy Limiter}
\label{sec:EnergyLimiter}

In addition to the potential violations of Eulerian-frame energy conservation, beyond the $\cO(v^{2})$ violations inherent to the model, from discontinuous $\vect{v}_h$ as discussed in Section~\ref{sec:dgConservation}, the realizability-enforcing limiter introduced in Section~\ref{sec:realizabilityLimiter} can also affect the conservation of energy. 
In fact, for small velocities (including $v=0$, when the total energy should be preserved to machine precision) the realizability-enforcing limiter is the dominant source of non-conservation of the Eulerian-frame energy.  
To improve Eulerian-frame energy conservation, we propose an ``energy limiter'' that corrects the change of energy induced by the realizability-enforcing limiter via a redistribution of particles between energy elements through a sweeping procedure.  
This approach maintains Eulerian-frame number and energy conservation across all energy elements for a given spatial element, at the expense of local number conservation in each energy element.  
The energy limiter does not correct for Eulerian-frame energy conservation violations inherent to the $\cO(v)$ two-moment model or due to discontinuous $\vect{v}_h$ (see Section~\ref{sec:dgConservation}).

To facilitate the discussion, we denote the element integrated Eulerian-frame number and energy by $\textsf{N}$ and $\textsf{E}$, respectively.  
Given evolved moments $\bcU_h=(\cN_h,\bcG_h)$ on element $\vect{K}$, the element-integrated number and energy can be computed by
\begin{align}
	\textsf{N}_{\vect{K}} 
	&= \int_{\vect{K}} {\cN}_{h} \varepsilon^2 d\vect{z}
	:=|\bK|\sum_{\vect{k}=1}^{|S_{\otimes}^{\vect{K}}|} w_{\vect{k}}^{(2)} {\cN}_{\vect{k}},
	\quad\text{and} \label{eq:ConservedNumber} \\
	\textsf{E}_{\vect{K}} 
	&= \int_{\vect{K}} ({\cN}_{h} + v^j \cG_{h,j})\varepsilon^3 d\vect{z}
	:=|\bK|\sum_{\vect{k}=1}^{|S_{\otimes}^{\vect{K}}|}  w_{\vect{k}}^{(3)} ({\cN}_{\vect{k}} + v^j \cG_{\vect{k},j}), \label{eq:ConservedEnergy}
\end{align}
where $S_{\otimes}^{\vect{K}}$ denote the set of local DG nodes as defined in Eq.~\eqref{eq:localNodes}, $\cN_{\vect{k}}$ and $\bcG_{\vect{k}}$ denotes the nodal values at points in $S_{\otimes}^{\vect{K}}$, and the weights $w_{\vect{k}}^{(2)}$ and $w_{\vect{k}}^{(3)}$ are given by the tensor product of the $(k+1)$-point one-dimensional LG quadrature rules introduced in Section~\ref{sec:dgMethod}, weighted by $\varepsilon^2$ and $\varepsilon^3$, respectively.  
Let $\widetilde{\vect{\mathcal{U}}}_{h}:=\texttt{RealizabilityLimiter}(\widehat{\vect{\mathcal{U}}}_{h})$ be the output of the realizability-enforcing limiter given a potentially non-realizable solution $\widehat{\vect{\mathcal{U}}}_{h}$, and let $(\widetilde{\textsf{N}}_{\vect{K}},\widetilde{\textsf{E}}_{\vect{K}})$ and $(\widehat{\textsf{N}}_{\vect{K}},\widehat{\textsf{E}}_{\vect{K}})$ denote the element-integrated number and energy, defined in Eqs.~\eqref{eq:ConservedNumber} and \eqref{eq:ConservedEnergy}, for $\widetilde{\vect{\mathcal{U}}}_{h}$ and $\widehat{\vect{\mathcal{U}}}_{h}$, respectively.  
As discussed in Section~\ref{sec:realizabilityLimiter}, the realizability-enforcing limiter gives a solution $\widetilde{\vect{\mathcal{U}}}_{h}$ that is realizable on $\widetilde{S}_{\otimes}^{\vect{K}}$ while maintaining number conservation in each element; i.e., $\widetilde{\textsf{N}}_{\vect{K}} = \widehat{\textsf{N}}_{\vect{K}}$.
However, in part because of the additional factor of $\varepsilon$ in the definition of the element-integrated energy in Eq.~\eqref{eq:ConservedEnergy}, the limiter results in energy changes (i.e., $\widetilde{\textsf{E}}_{\vect{K}} \neq \widehat{\textsf{E}}_{\vect{K}}$), which can lead to $\sum_{{\vect{K}}\in\cT}\widetilde{\textsf{E}}_{\vect{K}} \neq \sum_{{\vect{K}}\in\cT} \widehat{\textsf{E}}_{\vect{K}}$; i.e., a change in the global Eulerian-frame energy.  

The proposed energy limiter corrects Eulerian-frame energy conservation violations by redistributing particles via a sweeping procedure in the energy dimension to produce $\vect{\mathcal{U}}_{h}:=\texttt{EnergyLimiter}(\widetilde{\vect{\mathcal{U}}}_{h})$, as detailed in Algorithm~\ref{algo:energyLimiter}.  
Here we let $\cT_{\vect{x}}$ denote the collection of all spatial elements $\vect{K}_{\vect{x}}$ and let $\cT_{\varepsilon}=\{{K}_{\varepsilon,n}\}_{n=1}^{N_\varepsilon}$ denote the collection of all energy elements ${K}_{\varepsilon}$ that cover the energy domain $D_\varepsilon$.  
For a given spatial element $\vect{K}_{\vect{x}}\in\cT_{\vect{x}}$, the proposed energy limiter sweeps through elements $\vect{K}=K_{\varepsilon}\times \vect{K}_{\vect{x}}$ for all $K_{\varepsilon}\in\cT_{\varepsilon}$ in a user-prescribed order to redistribute particles in a way that the number and energy are both conserved for the given spatial element $\vect{K}_{\vect{x}}\in\cT_{\vect{x}}$, i.e., 
\begin{equation}
	\sum_{K_{\varepsilon}\in\cT_{\varepsilon}}\textsf{N}_{K_{\varepsilon}\times \vect{K}_{\vect{x}}} 
	= \sum_{K_{\varepsilon}\in\cT_{\varepsilon}} \widehat{\textsf{N}}_{K_{\varepsilon}\times \vect{K}_{\vect{x}}} 
	\quand 
	\sum_{K_{\varepsilon}\in\cT_{\varepsilon}}\widetilde{\textsf{E}}_{K_{\varepsilon}\times \vect{K}_{\vect{x}}} 
	= \sum_{K_{\varepsilon}\in\cT_{\varepsilon}} \widehat{\textsf{E}}_{K_{\varepsilon}\times \vect{K}_{\vect{x}}},
\end{equation}
which then leads to global number and energy conservation,
$\sum_{{\vect{K}}\in\cT}\textsf{N}_{\vect{K}} = \sum_{{\vect{K}}\in\cT} \widehat{\textsf{N}}_{\vect{K}}$ and
$\sum_{{\vect{K}}\in\cT}\textsf{E}_{\vect{K}} = \sum_{{\vect{K}}\in\cT} \widehat{\textsf{E}}_{\vect{K}}$, by summing over all spatial elements.

The sweeping procedure and the particle redistribution strategy used in the energy limiter are listed in Algorithm~\ref{algo:energyLimiter}.  
Specifically, when the Eulerian-frame energy conservation violation $\delta\textsf{E}$ is nonzero, the energy limiter redistributes particles between elements in a pairwise manner to correct $\delta\textsf{E}$.  
The pairwise redistribution strategy is detailed in Algorithm~\ref{algo:energyCorrection}, where a pair of scaling coefficients $(\theta_{1},\theta_{2})$ is computed by solving a linear system that requires the sum of the scaled energies to correct $\delta \textsf{E}$ while preserving the sum of particles (see Line~3).  
When at least one of the coefficients ($\theta_{1}$ and $\theta_{2}$) is less than a prescribed threshold $\theta_{\min}>-1$, a damping factor $\gamma$ is applied so that $\theta_{1}>\theta_{\min}$ and $\theta_{2}>\theta_{\min}$, which preserves moment realizability.
(The moment realizability property is invariant to scaling by a positive scalar.)  
When the linear system does not have a solution, or when $\min(\theta_{1},\theta_{2})<\theta_{\min}$, the output of Algorithm~\ref{algo:energyCorrection} does not fully correct $\delta \textsf{E}$, and the remainder is propagated to the next pair of elements in the sweeping procedure.  
As shown in Algorithm~\ref{algo:energyLimiter}, beginning on Line~18, a backward sweep will be launched after the forward sweep when $|\delta {\textsf{E}}|>\delta$; i.e., when the energy conservation violation is not fully corrected.
Here $\delta>0$ is a user-specified tolerance on the energy conservation violation. 
In the implementation, we choose to omit the condition in Line~19 and perform the full backward sweep in order to improve the computational efficiency on GPUs.
Moreover, to avoid numerical issues, we restrict the damping factor $\gamma$ in Algorithm~\ref{algo:energyCorrection} such that the resulting corrected moment ${\vect{\mathcal{U}}}_{h}$ remains a strictly positive number density; i.e., $\cN_h>0$.
In the numerical results reported in Section~\ref{sec:numericalResults}, we choose $\theta_{\min}=-0.5$ and permute the energy elements in an ascending order based on the associated energy values.  
We observe that the forward and backward sweeping procedure is sufficient for correcting the energy conservation violations introduced by the realizability-enforcing limiter --- i.e., $\delta\textsf{E}\to0$ during the sweeping procedure --- and that the additional scaling introduced in this energy correction process has no noticeable adverse impact on the solution to the two-moment system.  \\

\begin{algorithm}[H]
	\normalsize
	\medskip
	\caption{${\vect{\mathcal{U}}}_{h}=\texttt{EnergyLimiter}(\widetilde{\vect{\mathcal{U}}}_{h})$} 
	\label{algo:energyLimiter}
	
	{\bf Inputs:} Discretized moments before and after the realizability-enforcing limiter, i.e., $\widehat{\vect{\mathcal{U}}}_{h}$ and $\widetilde{\vect{\mathcal{U}}}_h$; a permutation of the energy elements, denoted as $K_{\varepsilon,n}$, $n=1,\dots, N_\varepsilon$. \\
	{\bf Parameter:} $\delta$ (Energy conservation violation tolerance)

	${\vect{\mathcal{U}}}_{h}\leftarrow\widetilde{\vect{\mathcal{U}}}_{h}$	\tcp*{Initialize the output moment}
	\For{each spatial element $\vect{K}_{\vect{x}}\in\cT_{\vect{x}}$}{
		\For{$n=1,\dots,N_\varepsilon$}{		
			Compute $({\textsf{N}}_{K_{\varepsilon,n}\times \vect{K}_{\vect{x}}}, {\textsf{E}}_{K_{\varepsilon,n}\times \vect{K}_{\vect{x}}} )$ from ${\vect{\mathcal{U}}}_{h}$ using Eqs.~\eqref{eq:ConservedNumber} and \eqref{eq:ConservedEnergy}\;
			Compute $\widehat{\textsf{E}}_{K_{\varepsilon,n}\times \vect{K}_{\vect{x}}}$ from $\widehat{\vect{\mathcal{U}}}_{h}$ using Eq.~\eqref{eq:ConservedEnergy}\;
			
			${\textsf{N}}_n \leftarrow {\textsf{N}}_{K_{\varepsilon,n}\times \vect{K}_{\vect{x}}} $,\, 
			${\textsf{E}}_n \leftarrow {\textsf{E}}_{K_{\varepsilon,n}\times \vect{K}_{\vect{x}}} $,\,
			$\widehat{\textsf{E}}_n \leftarrow \widehat{\textsf{E}}_{K_{\varepsilon,n}\times \vect{K}_{\vect{x}}} $\;
			
		}
		\If{$\sum_{n=1}^{N_\varepsilon}{\textsf{E}}_n\neq\sum_{n=1}^{N_\varepsilon} \widehat{\textsf{E}}_n$}{
			$\delta {\textsf{E}} \leftarrow {\textsf{E}}_1 - \widehat{\textsf{E}}_1$\;
			\tcc{Forward sweep}
			\For{$n=1,\dots,N_\varepsilon-1$}{
				$\delta {\textsf{E}} \leftarrow \delta {\textsf{E}} + {\textsf{E}}_{n+1} - \widehat{\textsf{E}}_{n+1}$\;
				\If{$|\delta {\textsf{E}}|>\delta$}{
					$(\theta_{n},\theta_{n+1})=\texttt{ComputeCorrection}(\textsf{N}_{n}, \textsf{E}_{n}, \textsf{N}_{n+1}, \textsf{E}_{n+1}, \delta\textsf{E})$\;
					\tcc{Update corrected moments, numbers, and energies}
					${\vect{\mathcal{U}}}_{h}\leftarrow (1+\theta_{n}) \, {\vect{\mathcal{U}}}_{h}$ on $K_{\varepsilon,n}\times \vect{K}_{\vect{x}}$, and
					${\vect{\mathcal{U}}}_{h}\leftarrow (1+\theta_{n+1}) \, {\vect{\mathcal{U}}}_{h}$ on $K_{\varepsilon,n+1}\times \vect{K}_{\vect{x}}$\;					
					$({\textsf{N}}_{n}, {\textsf{E}}_{n} )\leftarrow (1+\theta_{n}) \, ({\textsf{N}}_{n}, {\textsf{E}}_{n} )$, and
					$({\textsf{N}}_{n+1}, {\textsf{E}}_{n+1} )\leftarrow (1+\theta_{n+1}) \, ({\textsf{N}}_{n+1}, {\textsf{E}}_{n+1} )$\;
					$\delta {\textsf{E}} \leftarrow {\textsf{E}}_{n} + {\textsf{E}}_{n+1} + \delta {\textsf{E}}$\;
					
				}
			}
		\tcc{Backward sweep}
		\For{$n=N_\varepsilon-1,\dots,2$}{
		\If{$|\delta {\textsf{E}}|>\delta$}{
			$(\theta_{n},\theta_{n-1})=\texttt{ComputeCorrection}(\textsf{N}_{n}, \textsf{E}_{n}, \textsf{N}_{n-1}, \textsf{E}_{n-1}, \delta\textsf{E})$\;
					\tcc{Update corrected moments, numbers, and energies}
			${\vect{\mathcal{U}}}_{h}\leftarrow (1+\theta_{n}) \, {\vect{\mathcal{U}}}_{h}$ on $K_{\varepsilon,n}\times \vect{K}_{\vect{x}}$, and
			${\vect{\mathcal{U}}}_{h}\leftarrow (1+\theta_{n-1}) \, {\vect{\mathcal{U}}}_{h}$ on $K_{\varepsilon,n-1}\times \vect{K}_{\vect{x}}$\;					
			$({\textsf{N}}_{n}, {\textsf{E}}_{n} )\leftarrow (1+\theta_{n}) \, ({\textsf{N}}_{n}, {\textsf{E}}_{n} )$, and
			$({\textsf{N}}_{n-1}, {\textsf{E}}_{n-1} )\leftarrow (1+\theta_{n-1}) \, ({\textsf{N}}_{n-1}, {\textsf{E}}_{n-1} )$\;
			$\delta {\textsf{E}} \leftarrow {\textsf{E}}_{n} + {\textsf{E}}_{n-1} + \delta {\textsf{E}}$\;
			
		}
		\Else{
		\texttt{break}\;}
	}
		
	}
	}

\end{algorithm}	

\begin{algorithm}[H]
	\normalsize
	\medskip
	\caption{$(\theta_{1},\theta_{2})=\texttt{ComputeCorrection}(\textsf{N}_{1}, \textsf{E}_{1}, \textsf{N}_{2}, \textsf{E}_{2}, \delta\textsf{E})$} 
	\label{algo:energyCorrection}
	
	{\bf Inputs:} $\textsf{N}_{1}, \textsf{E}_{1}, \textsf{N}_{2}, \textsf{E}_{2}, \delta\textsf{E}$ \\
	{\bf Parameter:} $\theta_{\min}>-1$ \\

				Compute $(\theta_{1},\theta_{2})$ by solving
					$\left\{\begin{alignedat}{2}
						\theta_{1} {\textsf{N}}_{1} + \theta_{2} {\textsf{N}}_{2} & = 0\:\\
						\theta_{1} {\textsf{E}}_{1} + \theta_{2} {\textsf{E}}_{2} & = -\delta \textsf{E}\:
					\end{alignedat}\right.$\;
					\If{no solution}{
						$(\theta_{1},\theta_{2})\leftarrow(0,0)$\;
					}
					\If{$\min(\theta_{1},\theta_{2})<\theta_{\min}$}{
						$\gamma \leftarrow \frac{\theta_{\min}}{\min(\theta_{1},\theta_{2})}$\;
						$(\theta_{1},\theta_{2}) \leftarrow  (\gamma\theta_{1},\gamma\theta_{2})$  \tcp*{Limit for realizability}
					}	
	
\end{algorithm}	

\section{Implementation, Programming Models, and Portability}
\label{sec:gpu}

The DG-IMEX method proposed here has been implemented in the toolkit for high-order neutrino radiation hydrodynamics (\thornado).  
Here we briefly discuss some considerations in this process.  

Neutrino transport is only one component (along with, e.g., hydrodynamics, nuclear reaction kinetics, and gravity) of a broader, multiphysics simulation framework needed to model multiscale astrophysical systems, e.g., core-collapse supernova explosions.
However, the number of evolved degrees of freedom is relatively high compared to other components.  
For example, simulations incorporating a two-moment model (four moments), evolving three independent neutrino flavors (six species), with 16 linear elements (k=1) to discretize the energy dimension evolve $4\times6\times16\times(k+1)=768$ degrees of freedom per spatial point.  
As such, spectral neutrino radiation transport represents the bulk of the computational load in such scientific applications.
With this in mind, node-level performance and portability for heterogeneous computing systems are prioritized in \thornado\ development as a collection of modular physics components that can be incorporated into distributed mutliphysics simulation codes (e.g., \flashx\ \cite{dubey_etal_2022}), which are equipped with native infrastructure for distributed parallelism.  
In particular, we target frameworks that utilize adaptive mesh refinement, where simulation data is mapped to smaller grid blocks of relatively even size.  

\thornado\ uses a combination of compiler directives and optimized linear algebra libraries to accelerate all components of the DG-IMEX method.
All of the solver components --- e.g., the computation of numerical fluxes, evaluation of phase-space divergences, and limiters --- are reduced to small, discrete kernels that can be executed either as collapsible, tightly-nested loops over phase-space dimensions or basic linear algebra operations.  
In addition to optimizing many key metrics for GPU performance (e.g., occupancy, register pressure, and memory coalescence), this strategy naturally exposes vector-level parallelism which also benefits performance on modern, multicore CPUs.
This is especially important when invoking iterative solvers, such as those described in Sections~\ref{sec:moment_conversion} and \ref{sec:collision_solver}, across many independent phase-space points.  
Since iteration counts can vary, assigning an even number of phase-space points to each thread can lead to severe load imbalance among GPU threads.  
We address this problem by tracking the convergence of each point independently, removing them from calculations in each kernel until all points have converged.

Our portability strategy focuses on maintaining a single code-base that can efficiently execute on different hardware architectures and software environments.  
\thornado\ contains three distinct implementations of compiler directives that are managed with C preprocessor macros: traditional OpenMP (CPU multi-core), OpenMP offload (GPU), and OpenACC (GPU).  
We refer to code listings in \cite{laiu_etal_2020} for specific examples.  

Interfaces to optimized linear algebra routines are also written in a generic way for portability across different libraries.  
Currently, \thornado\ has linear algebra interfaces supporting several LAPACK and BLAS \cite{laug} routines with GPU implementations from NVIDIA, AMD, Intel, and \texttt{MAGMA} \cite{MAGMA}.  
This approach hides the complexities of managing different interfaces in a single \thornado\ module that can be easily used throughout the code.  
In addition to the individual routine interfaces, each linear algebra package requires specific attention to interoperability with the compiler directives to ensure correct synchronization when using multiple execution streams per device.  
This is managed during initialization with compiler directives and C preprocessor macros.  

We provide timing results and a breakdown of the computational cost associated with key solver components for one of the numerical examples in Section~\ref{sec:numericalResults}.  
\section{Numerical Tests}
\label{sec:numericalResults}

In this section, we demonstrate the performance of our implementation of the DG-IMEX method to solve the $\cO(v)$ two-moment model.  
We consider problems with and without collisions.  
For problems with collisions, we use the IMEX scheme proposed in \cite{chu_etal_2019} (see also \cite{laiu_etal_2021} for details).  
For problems without collisions, we use the optimal second- and third-order accurate strong stability-preserving Runge--Kutta methods from \cite{shuOsher_1988}, referred to as SSPRK2 and SSPRK3, respectively.  
For the tests in Sections~\ref{sec:sineWaveStreaming} and \ref{sec:gaussianDiffusion}, unless stated otherwise, we set the time step to $\dt=0.3\times|K_{\bx}^{1}|/(k+1)$, where $k$ is the polynomial degree.  
For the tests in Sections~\ref{sec:streamingDopplerShift}-\ref{sec:transparentVortex}, we enforce the time step restriction given in Theorem~\ref{thm:realizability}.

Collisions tend to drive the distribution towards isotropy in the angular dimensions of momentum space (i.e., $|\bcI|\to0$), which places the comoving-frame moments $\bcM$ safely inside the realizable domain.  
Therefore, to emphasize the improved robustness resulting from our analysis, our main focus is on phase-space advection problems without collisions, where the moments evolve close to the boundary of the realizable domain.  

\subsection{Moment Conversion Solver}
\label{sec:momentConversion}

The solution of the conserved-to-primitive moment conversion problem in Eq.~\eqref{eq:ConservedToPrimitive} and the implicit system in Eq.~\eqref{eq:ImplicitSystem} contribute the majority of the computational cost of the realizability-preserving scheme.  
In this section, we test the iterative solver for solving the moment conversion problem  Eq.~\eqref{eq:ConservedToPrimitive} with various solver configurations, and the results reported provide guidance for selecting iterative solver configurations for this critical part of the algorithm.

As discussed in Section~\ref{sec:moment_conversion}, we formulate the moment conversion problem in Eq.~\eqref{eq:ConservedToPrimitive} as a fixed-point problem on the primitive moments $\bcM=(\cD,\bcI)^{\intercal}$ of the form stated in Eq.~\eqref{eq:richardson_fixed_pt}.  
In Lemma~\ref{lem:solver_realizability}, we have shown that the moment realizability is preserved in the iterative procedure when Eq.~\eqref{eq:richardson_fixed_pt} is solved with the Picard iteration method in Eq.~\eqref{eq:Picard}
and $\lambda\leq(1+v)^{-1}$ in Eq.~\eqref{eq:richardson_fixed_pt}.
The convergence of Picard iteration is guaranteed in Theorem~\ref{thm:convergence} with the additional assumption that $v < \sqrt{2}-1$.  

We first compare the iteration counts required for convergence of the Picard iteration solver and an Anderson acceleration (AA) solver, using two different choices for $\lambda$.  
The AA technique was first proposed in \cite{Anderson-1965} to accelerate the convergence of fixed-point iterations by accounting for the past iteration history to compute new iterates.  
Here we follow the formulation and implementation in \cite{Walker-Ni-2011,laiu_etal_2021} and apply the AA solver to the moment conversion problem Eq.~\eqref{eq:richardson_fixed_pt}.  
In Figure~\ref{fig:SolverIterCnt}, the iteration counts are reported for the two iterative solvers applied to solve Eq.~\eqref{eq:richardson_fixed_pt} at varying fluid speed $v:=|\vect{v}|$ and flux factor $h=|\bcI|/\cD$, with $\lambda$ chosen to be the largest allowable value, i.e., $\lambda=(1+v)^{-1}$ and a more conservative value $\lambda=0.5$.  
The AA solver uses the memory parameter $m=1$ (defined in \cite{laiu_etal_2021}), so that only information from the previous and current iterate is used.  
The stopping criteria for both solvers are given as
\begin{equation}\label{eq:stoppingCriteria}
	\|\bcM^{[k]} - \bcM^{[k-1]}\|\leq \texttt{tol}\, \|\bcU\|,
\end{equation}
where we consider the norms in the $L^2$ sense and the tolerance $\texttt{tol}=10^{-8}$.
For each choice of $(v,h)$ in Figure~\ref{fig:SolverIterCnt}, the fixed-point problem is solved for 100 randomly generated $\bcU\in\cR$ (varying the direction of $\vect{v}$ and $\bcI/\cD$ randomly), and the averaged iteration counts over these 100 problems are recorded. In each test, the initial guess takes the form $\bcM^{[0]} = \bcU$.
The results in Figure~\ref{fig:SolverIterCnt} illustrate that, for both the Picard iteration and the AA solvers, choosing the parameter $\lambda$ to be the largest allowable value $(1+v)^{-1}$ indeed reduces the iteration counts from the more conservative choice $\lambda=0.5$, particularly in the low velocity regime.  
In addition, it can be found from Figure~\ref{fig:SolverIterCnt} that AA solver consistently outperforms the Picard iteration method, and the advantage of using AA grows as the velocity increases.
We note that the realizability-preserving and convergence properties analyzed in Section~\ref{sec:momentConversionRealizability} are only applicable to the Picard iteration solver, and\ not to the AA solver.%
\footnote{The realizability-preserving and convergence properties of the AA solver require additional conditions such as boundedness of extrapolation coefficients, which we do not enforce in the implementation.}
However, in the numerical results reported in Figure~\ref{fig:SolverIterCnt}, we have not observed convergence failure by any of the solvers, even when the velocity is larger than the upper bound ($v=\sqrt{2}-1$; plotted as a red vertical line in each panel in Figure~\ref{fig:SolverIterCnt}) required in the convergence analysis in Theorem~\ref{thm:convergence}.

In Figure~\ref{fig:SolverIC}, we show results from experimenting with two choices for the initial guess, $\bcM^{[0]} = (\cN, \vect{0})^{\intercal}$ and $\bcM^{[0]} = \bcU = (\cN, \bcG)^{\intercal}$, for the AA solver with $\lambda = (1+v)^{-1}$, which is the best performing configuration shown in Figure~\ref{fig:SolverIterCnt}.  
As shown in Figure~\ref{fig:SolverIC}, initializing with the conserved moment $\bcU$ generally outperforms the isotropic initial condition $(\cN,\vect{0})$, except for the case when the flux factor $h=0$, for which the isotropic initial condition is exactly the primitive moment.  
Since we expect moments with $h=0$ to be rarely encountered in numerical simulations, adopting the AA solver with $\lambda = (1+v)^{-1}$ and initial guess $\bcM^{[0]} = \bcU$ appears to be the best choice.  This conjecture is confirmed in the performance comparison reported in Section~\ref{sec:performance}, where we observe a considerable improvement in terms of computational time by using the initial guess $\bcM^{[0]} = \bcU$.

\begin{figure}[h]
	\centering
	\begin{tabular}{cc}
		\includegraphics[width=0.44\textwidth]{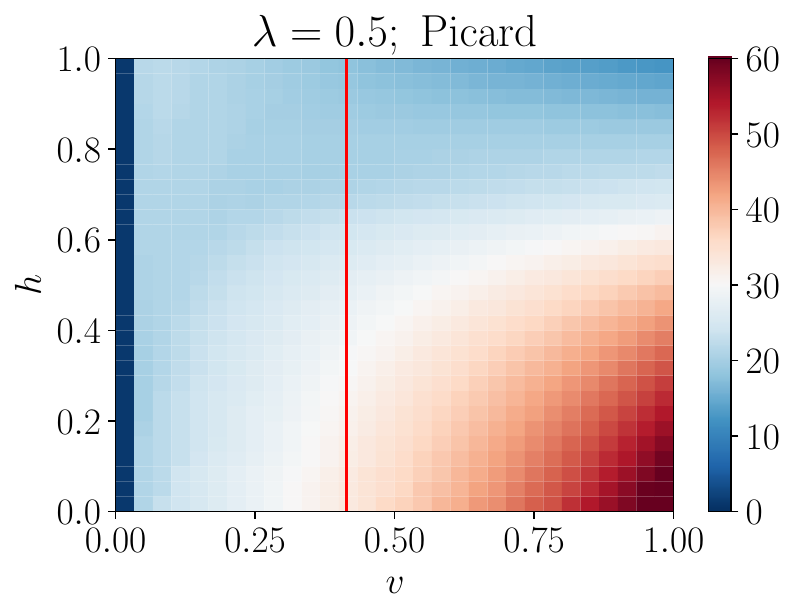} &
		\includegraphics[width=0.44\textwidth]{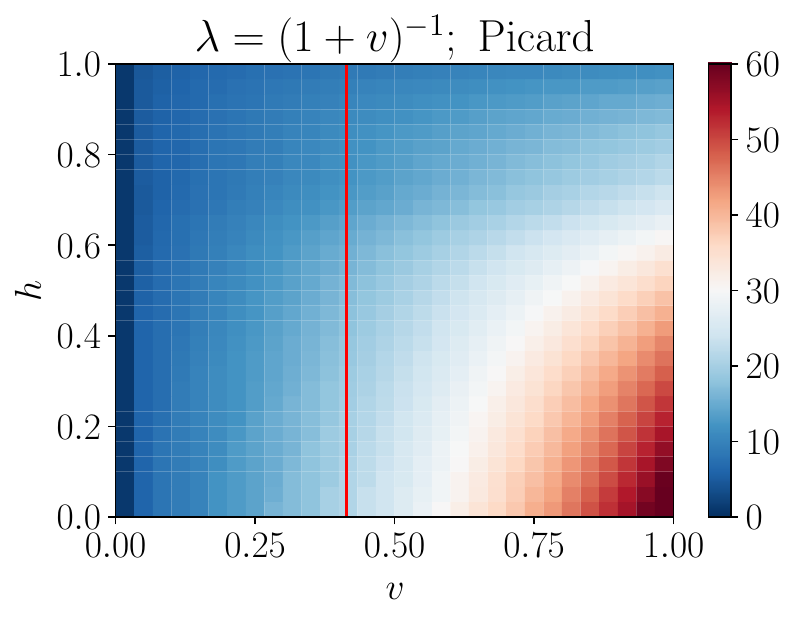} \\
		\includegraphics[width=0.44\textwidth]{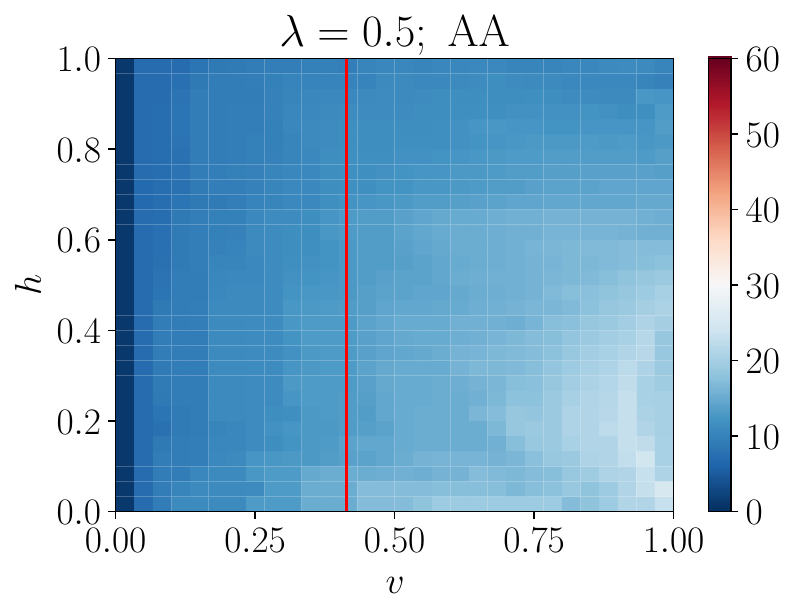} &
		\includegraphics[width=0.44\textwidth]{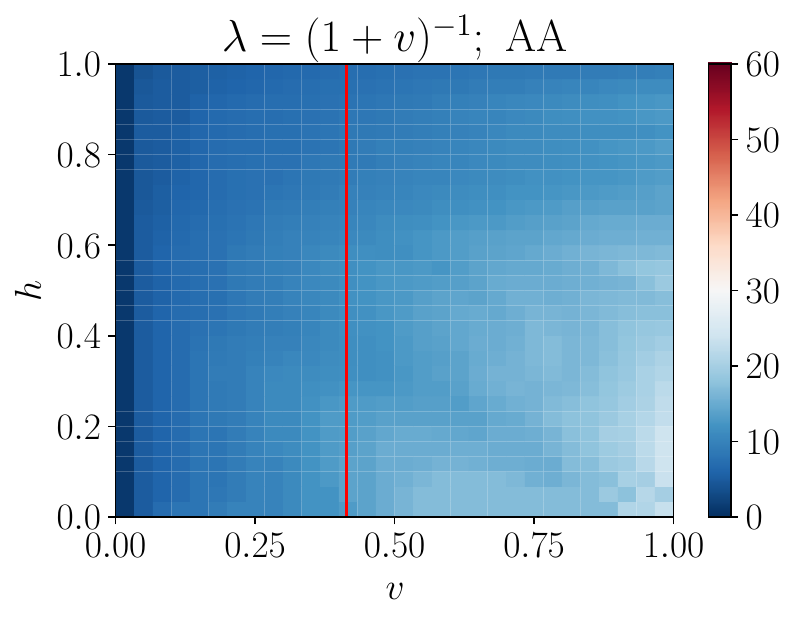}
	\end{tabular}
	\caption{Iteration counts for the Picard iteration (top panels) and AA (bottom panels) solvers with modified Richardson iteration parameter $\lambda = (1+v)^{-1}$ and $\lambda=0.5$ (right and left columns, respectively), applied to the moment conversion problems with various velocity $v$ and flux factor $h$. The reported iteration counts are the average over 100 randomly generated moment conversion problems at each $(v,h)$, where the randomness is applied to the directions of $\vect{v}$ and $\bcI/\cD$. In each panel, the red vertical line indicates the upper velocity bound for guaranteed convergence, $v=\sqrt{2}-1$.}
	\label{fig:SolverIterCnt}
\end{figure}

\begin{figure}[h]
	\centering
	\begin{tabular}{cc}
		\includegraphics[width=0.44\textwidth]{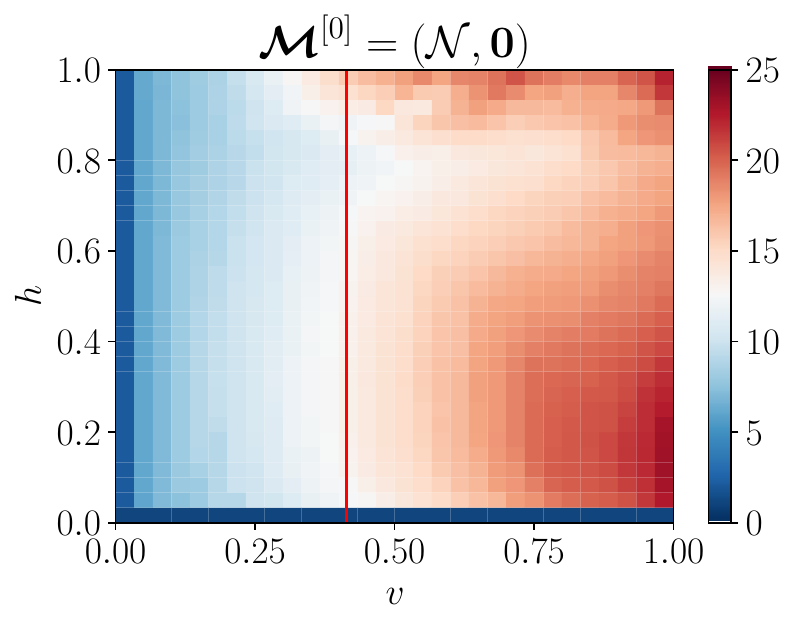} &
		\includegraphics[width=0.44\textwidth]{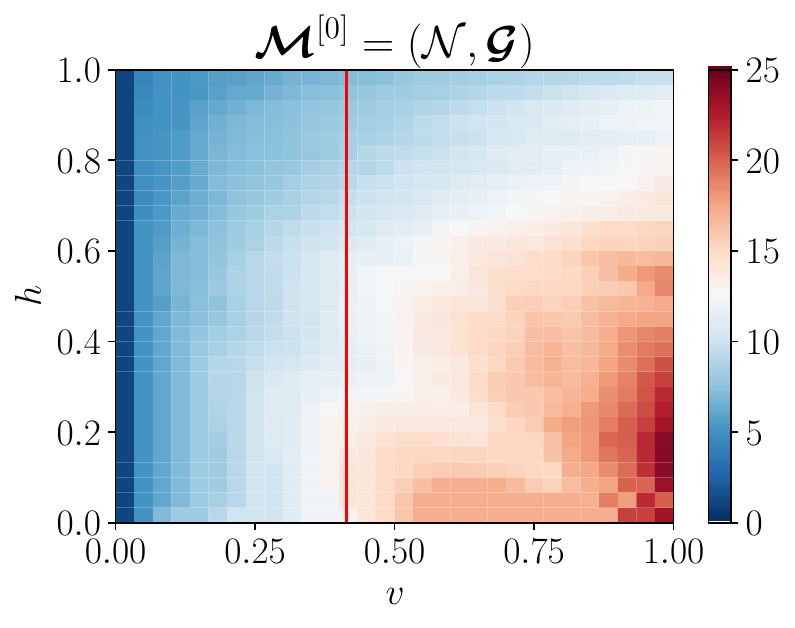} 
	\end{tabular}
	\caption{Iteration counts for the AA solver with modified Richardson iteration parameters $\lambda = (1+v)^{-1}$ applied to the moment conversion problems with various velocity $v$ and flux factor $h$ for two different initial guesses. The reported iteration counts are the average over 100 randomly generated moment conversion problems at each $(v,h)$.}
	\label{fig:SolverIC}
\end{figure}

\subsection{Sine Wave Streaming}
\label{sec:sineWaveStreaming}

The first test we consider that evolves the two-moment system models free-streaming radiation through a background with a spatially (and temporally) constant velocity field in the $x^{1}$-direction.  
That is, we set $\chi=\sigma=0$, while $\vect{v}=[v,0,0]^{\intercal}$, with $v=0.1$.  
The purpose of this test is to verify (i) the correct radiation propagation speed in this idealized setting, and (ii) the expected order of accuracy of the implemented method.  
We consider a periodic one-dimensional unit spatial domain $D_{x^{1}}=[0,1]$.  
The initial number density and flux are set to $\cD(x^{1},0)=\cD_{0}(x^{1})=0.5+0.49\times\sin(2\pi x^{1})$ and $\cI^{1}(x^{1},0)=\cD_{0}(x^{1})$, respectively.  
Then, the flux factor is $h=1$, and the analytic solution is given by $\cD(x^{1},t)=\cI^{1}(x^{1},t)=\cD_{0}(x^{1}-t)$; i.e., the initial profile propagates with unit speed, \emph{independent} of $v$.  
(As noted by \cite{lowrie_etal_2001}, dropping the velocity-dependent terms in the time derivatives of Eqs.~\eqref{eq:spectralNumberEquation} and \eqref{eq:spectralNumberFluxEquation}, as is done in \cite{just_etal_2015,skinner_etal_2019}, the propagation speed becomes $1+v$ for this test, which is unphysical.)  
Since the background velocity is constant, there is no coupling in the energy dimension.  
Therefore, this test is performed with a single energy.  
We run this test until $t=1$, when the initial profile has crossed the grid once before returning to its initial position.  

\begin{figure}[H]
  \centering
  \includegraphics[width=0.5\textwidth]{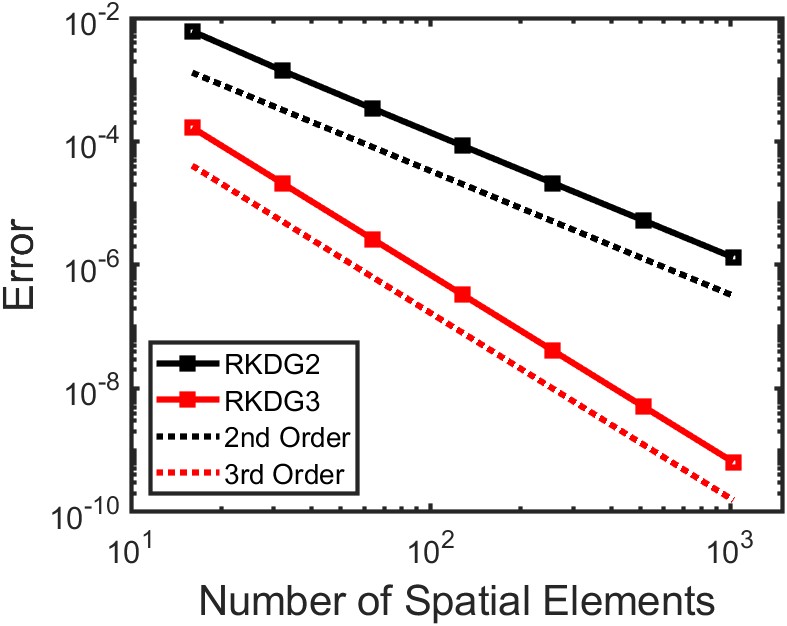}
  \caption{Error in the $L^{2}$ norm versus number of spatial elements $N$ for the Sine Wave Streaming test.  Results obtained with second- and third-order schemes (using $k=1$ polynomials and SSPRK2 time stepping and $k=2$ polynomials and SSPRK3 time stepping, respectively), along with dotted reference lines proportional to $1/N^{k+1}$, are plotted using black and red lines, respectively.}
  \label{fig:SinewaveStreamingOrder}
\end{figure}

In Figure~\ref{fig:SinewaveStreamingOrder}, we plot the error in the $L^2$ norm versus the number of spatial elements for the second-order scheme, using linear polynomials ($k=1$) and SSPRK2 time stepping, and the third-order scheme, using quadratic polynomials ($k=2$) and SSPRK3 time stepping.  
The results confirm the expected convergence rate to the exact solution.  

\subsection{Gaussian Diffusion}
\label{sec:gaussianDiffusion}

The next test we consider, adopted from \cite{just_etal_2015}, models the diffusion of particles through a medium moving with constant velocity in the $x^{1}$-direction.  
We consider a purely scattering medium and set $\sigma=3.2\times10^{3}$ and $\chi=0$, and let $\vect{v}=[v,0,0]^{\intercal}$, with $v=0.1$.  
We let the spatial domain be periodic, $D_{x^{1}}=[0,3]$, with initial conditions $\cD_{0}(x^{1})=\exp[-(x^{1}-x_{0}^{1})^{2}/(4t_{0}\kappa_{\mathsf{D}})]$ and $\cI_{0}^{1}(x^{1})=-\kappa_{\mathsf{D}}\p_{x^{1}}\cD_{0}$, where $\kappa_{\mathsf{D}}=(3\sigma)^{-1}$, and we set $x_{0}^{1}=1$ and $t_{0}=5$.  
Then, the evolution of the number density is approximately governed by the advection-diffusion equation
\begin{equation}
	\p_{t}\cD+\p_{x^{1}}\big(\,\cD v - \kappa_{\mathsf{D}}\p_{x^{1}}\cD\,\big) = 0,
	\label{eq:advectionDiffusion}
\end{equation}
whose analytical solution is given by \cite{just_etal_2015}
\begin{equation}
	\cD(x^{1},t) = \sqrt{\f{t_{0}}{t_{0}+t}}\exp\Big\{\,-\f{((x^{1}-v t)-x_{0}^{1})^2}{4(t_{0}+t)\kappa_{\mathsf{D}}}\,\Big\}.  
	\label{eq:advectionDiffusionAnalytic}
\end{equation}
Since there is no coupling in the energy dimension ($v$ is constant), we perform this test with a single energy.  
We use quadratic elements ($k=2$) and the IMEX time stepping scheme from \cite{chu_etal_2019}, integrating the collision term implicitly.  
For this test, the time step is set to $\dt=C_{\rm CFL}\times |K_{\bx}^{1}|$, where $C_{\rm CFL}$ is specified below.  
The purpose of this test is to investigate the performance of the DG-IMEX scheme in a regime where both advection and diffusion contribute to the evolution of the number density.  
For $t>0$, the Gaussian profile is advected with the flow, while the amplitude decreases and the width increases due to diffusion.  

The left panel in Figure~\ref{fig:GaussianDiffusion} shows the number density versus $x^{1}-v t$ for various times as the Gaussian profile propagates once across the periodic domain and returns to its initial position at $t=30$.  
At this time the amplitude is reduced by a factor $\sqrt{5/35}\approx0.378$.  
For this simulation, the spatial domain is discretized using $96$ elements, so that the radio of the element width to the mean-free path is $|K_{\bx}^{1}|\sigma=10^{2}$.  
The numerical solution (open circles) agrees well with the expression in Eq.~\eqref{eq:advectionDiffusionAnalytic} (solid lines).  

The right panel in Figure~\ref{fig:GaussianDiffusion} shows the error in the $L^2$ norm at $t=5$ versus the number of spatial elements $N$.  
Since the expression given by Eq.~\eqref{eq:advectionDiffusionAnalytic} is not an exact solution to the $\cO(v)$ two-moment model in the limit of high scattering opacity, we compare the numerical results to a high-resolution reference solution, computed with $8192$ elements.  
(We have confirmed that for fixed $N$ and $t$, and varying $v$, the difference between the numerical solution and the expression in Eq.~\eqref{eq:advectionDiffusionAnalytic} increases as $\cO(v^{2})$.)  
The solid black curve with squares shows the error obtained with the standard CFL number $C_{\rm CFL}:=C_{\rm CFL}^{0}=0.3/(k+1)$.  
For smaller $N$, the error falls off as $N^{-3}$ (see dashed black reference line), consistent with a third-order convergence rate, while for larger $N$ the convergence rate transitions to first-order (see dotted black reference line).  
The reason for this change in convergence rate is because the IMEX scheme, taken from \cite{chu_etal_2019}, is formally only first-order accurate.  
Spatial discretization errors dominate for small $N$, but since these errors decrease with the third-order rate, temporal errors become dominant for large $N$.  
To verify this, we also plot convergence results obtained after reducing the time step by a factor of $25$, $C_{\rm CFL}:=C_{\rm CFL}^{0}/25$; see solid red curve with squares.  
For this case, the error decreases with the third-order rate for all $N$.  
We also show the error obtained with a second-order IMEX scheme (IMEXRKCB2 from \cite{cavaglieriBewley2015}), using the standard CFL number.  
Due to better temporal accuracy, the error decreases with the third-order rate, but the scheme does not satisfy the convex-invariant conditions delineated in \cite{chu_etal_2019}, and can therefore not be guaranteed to maintain moment realizability by our analysis.  

\begin{figure}[H]
  \centering
  \begin{tabular}{cc}
    \includegraphics[width=0.44\textwidth]{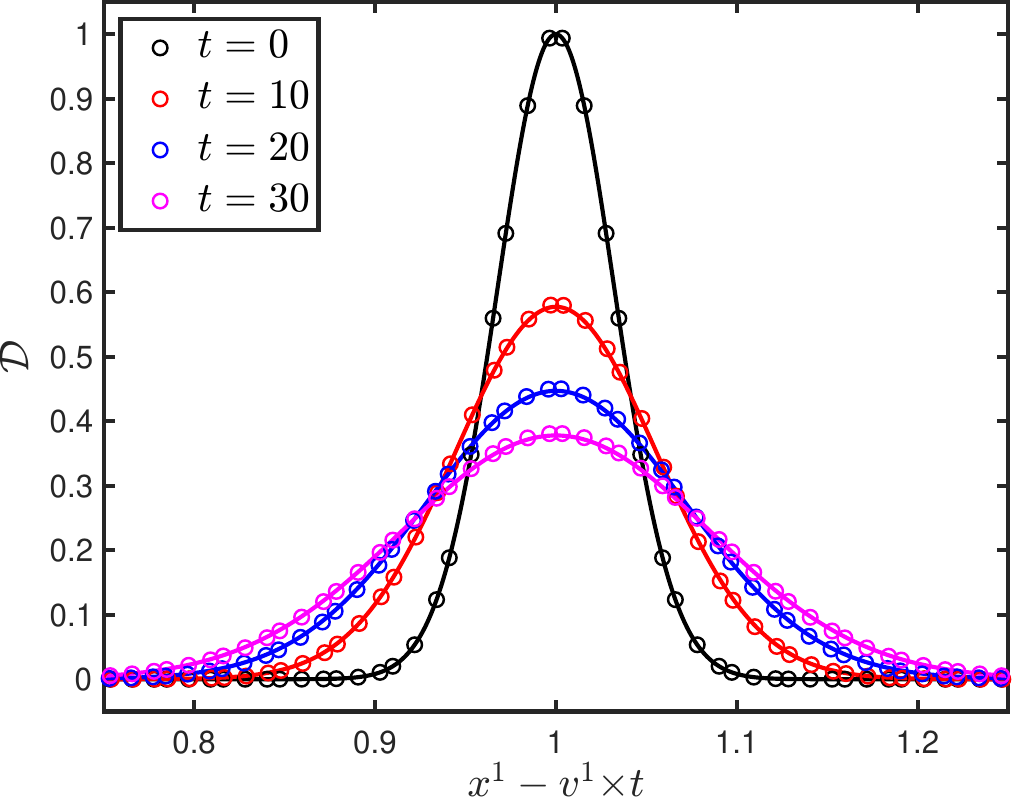} &
    \includegraphics[width=0.44\textwidth]{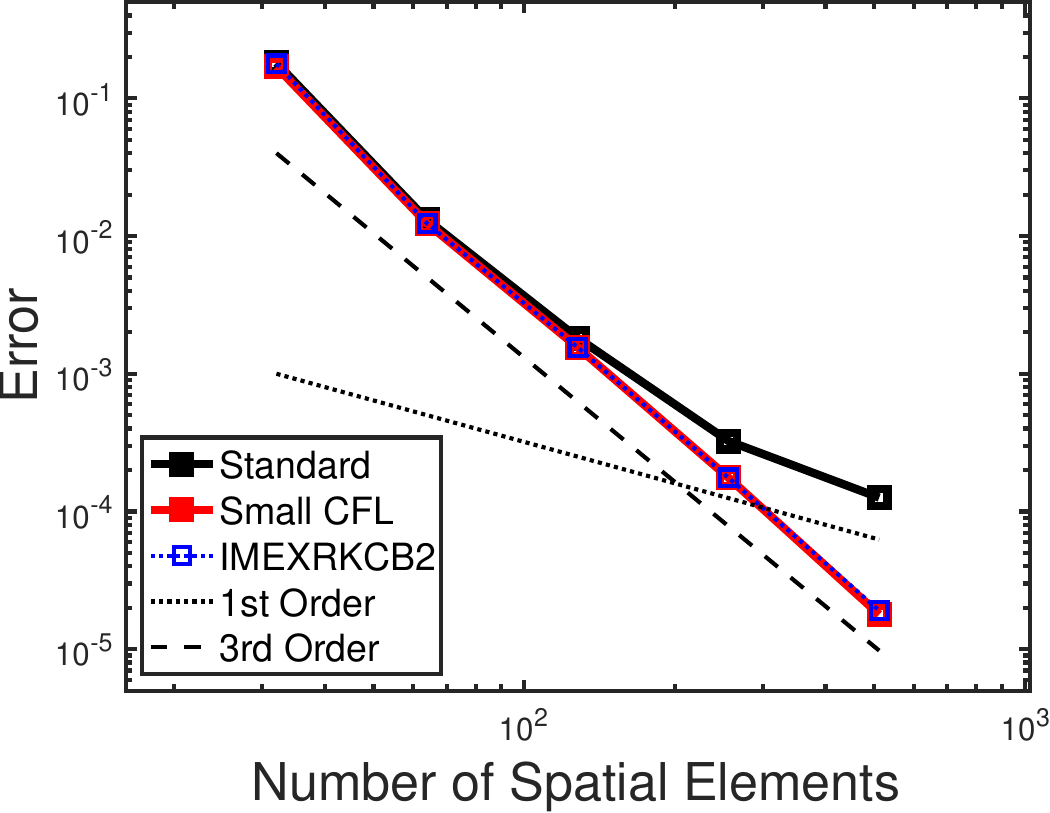}
  \end{tabular}
   \caption{Results for the Gaussian diffusion test.  The left panel shows the numerical solution (open circles) versus the shifted coordinate $x^{1}-v t$ for various times, compared with the analytic solution (solid lines) to the advection-diffusion equation in Eq.~\eqref{eq:advectionDiffusion}.  The right panel shows the error in the $L^{2}$ norm versus the number of elements $N$.  The error is computed with respect to a high-resolution reference run with $N=8192$.  Convergence results are shown for standard and reduced CFL number, using the first-order SSP-IMEX scheme from \cite{chu_etal_2019} (solid black and red, respectively; see text for details), and a second-order IMEX scheme from \cite{cavaglieriBewley2015} with standard CFL number (dotted blue).  The dotted and dashed black reference lines are proportional to $N^{-1}$ and $N^{-3}$, respectively.  }
  \label{fig:GaussianDiffusion}
\end{figure}

\subsection{Streaming Doppler Shift}
\label{sec:streamingDopplerShift}

This test, adopted from \cite{vaytet_etal_2011} (see also \cite{just_etal_2015,skinner_etal_2019}), models the propagation of free-streaming radiation along the $x^{1}$-direction through a background with a spatially varying velocity field.  
Because the two-moment model adopts momentum-space coordinates associated with a comoving observer, the radiation energy spectra will be Doppler shifted.  
We consider a one-dimensional spatial domain $D_{x^{1}}=[0,10]$.  
Again, we set $\chi=\sigma=0$, while the velocity field is set to $\vect{v}=(v,0,0)^{\intercal}$, where
\begin{equation}
	v(x^{1}) = 
	\left\{
	\begin{array}{ll}
		0, & x^{1}\in[0,2) \\
		v_{\max} \times \sin^{2}[2\pi(x^{1}-2)/6], & x^{1}\in[2,3.5) \\
		v_{\max}, & x^{1}\in[3.5,6.5) \\
		v_{\max} \times \sin^{2}[2\pi(x^{1}-2)/6], & x^{1}\in[6.5,8) \\
		0, & x^{1}\in[8,10]
	\end{array}
	\right.,
\end{equation}
and where we will vary $v_{\max}$.  
We set the energy domain to $D_{\varepsilon}=[0,50]$.  
In this test, we discretize the spatial and energy domains into $128$ and $32$ elements, respectively, and use quadratic elements ($k=2$) and SSPRK3 time stepping.  
In the computational domain, the moments are initially set to $\cD=1\times10^{-40}$ and $\cI^{1}=0$ for all $(x^{1},\varepsilon)\in D_{x^{1}}\times D_{\varepsilon}$.  
At the inner spatial boundary, we impose an incoming, forward-peaked radiation field with a Fermi-Dirac spectrum; i.e., we set $\cD(\varepsilon,x^{1}=0)=1/[\exp(\varepsilon/3-3)+1]$ and $\cI^{1}(\varepsilon,x^{1}=0)=0.999\times\cD(\varepsilon,x^{1}=0)$, so that the flux factor $h\approx1$.  
(We impose outflow boundary conditions at $x^{1}=10$.)  
Then, for $t>0$, a radiation front propagates through the computational domain, and a steady state is established for $t\gtrsim10$, where the spectrum is Doppler-shifted according to the velocity field.  
From special relativistic considerations, similar to \cite{just_etal_2015}, the analytical spectral number density in the steady state can be written as
\begin{equation}
	\cD_{\rm A}=\f{s^{2}}{\exp(s\varepsilon/3-3)+1}, 
	\label{eq:dopplerSpectraSR}
\end{equation}
where $s=\sqrt{(1+v)/(1-v)}$.  
The purpose of this test is to (i) compare steady state numerical solutions with the prediction from special relativity given by Eq.~\eqref{eq:dopplerSpectraSR}, and (ii) investigate the simultaneous Eulerian-frame number and energy conservation properties of the method as the initial conditions are evolved to steady state.  
To reach an approximate steady state, we run all models until $t=20$.  

\subsubsection{General Solution Characteristics}

\begin{figure}[H]
  \centering
  \begin{tabular}{cc}
    \includegraphics[width=0.44\textwidth]{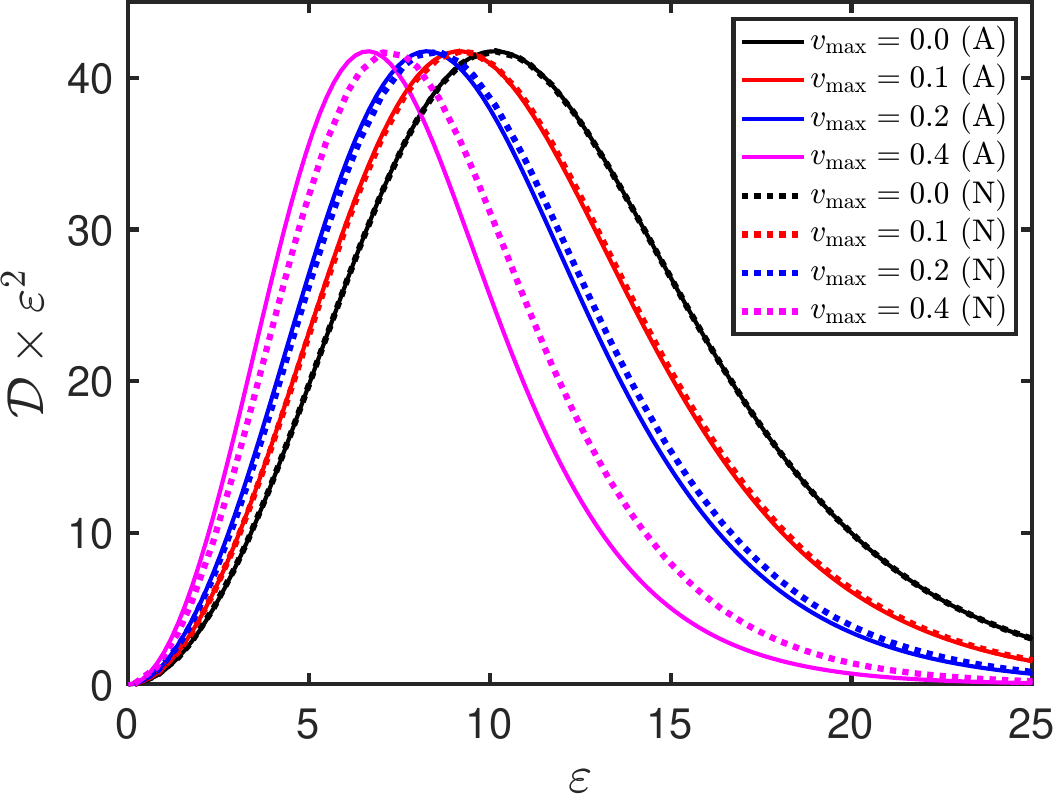} &
    \includegraphics[width=0.44\textwidth]{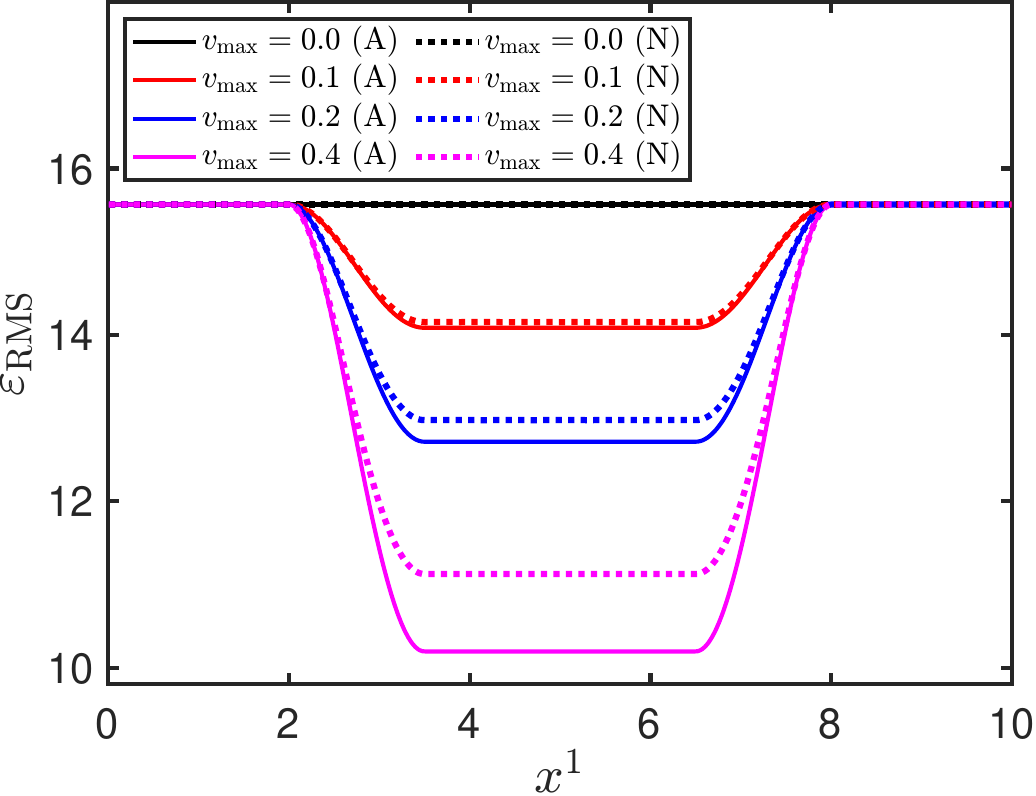} \\
    \includegraphics[width=0.44\textwidth]{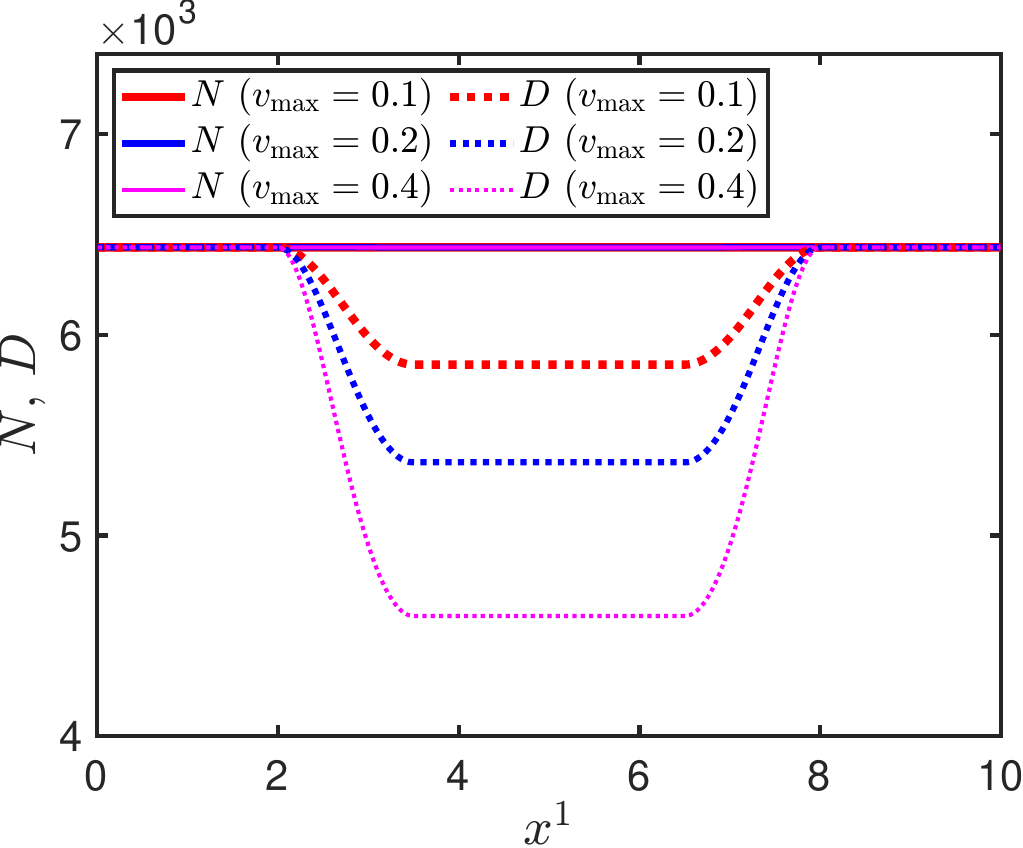} &
    \includegraphics[width=0.44\textwidth]{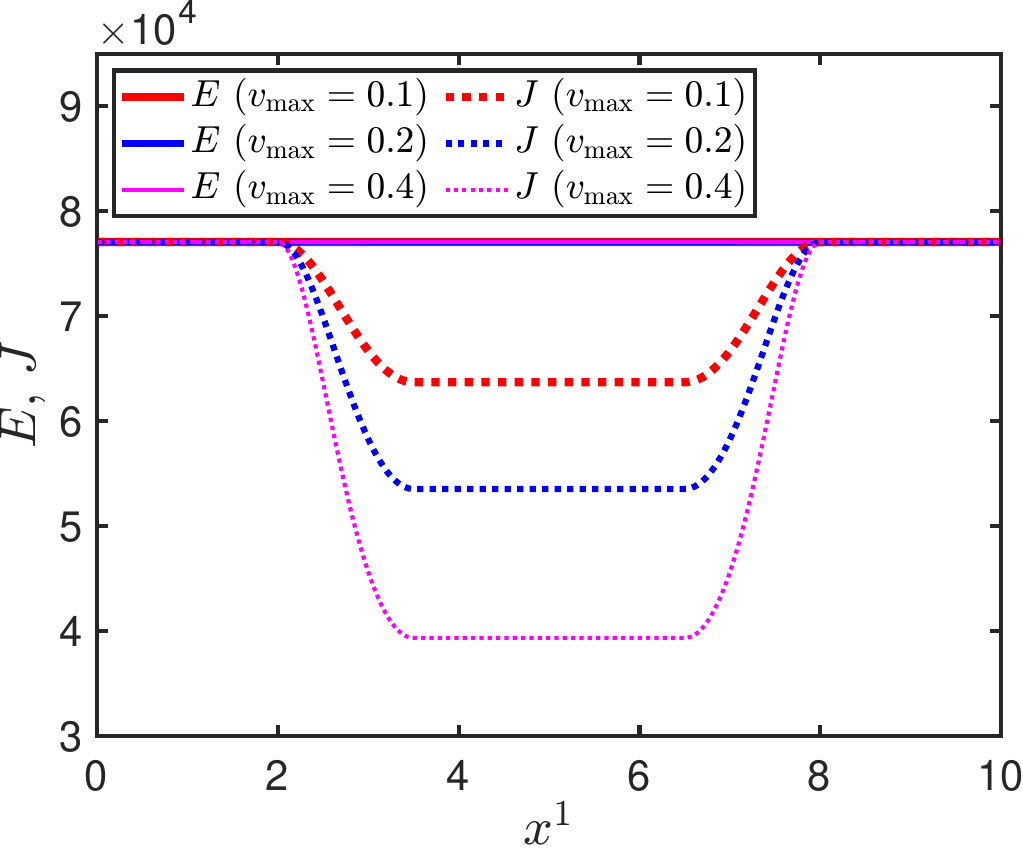}
  \end{tabular}
   \caption{Steady state solutions ($t=20$) for the streaming Doppler shift problem for various $v_{\max}\in\{0.0,0.1,0.2,0.4\}$.  In the top panels, we plot spectra at $x^{1}=5$ ($\cD\varepsilon^{2}$ versus $\varepsilon$; left) and the RMS energy, as defined in Eq.~\eqref{eq:rmsEnergy}, versus position $x^{1}$ (right).  In these panels, solid lines represent the analytic (A) solution from special relativistic considerations, given by Eq.~\eqref{eq:dopplerSpectraSR}, while dotted lines represent the numerical (N) results.  In all panels, black, red, blue, and magenta curves represent runs with $v_{\max}$ set to $0.0$, $0.1$, $0.2$, and $0.4$, respectively.  In the bottom left panel, the Eulerian-frame (solid lines) and comoving-frame (dotted lines) number densities are plotted versus position.  In the bottom right panel, the Eulerian-frame (solid lines) and comoving-frame (dotted lines) energy densities are plotted versus position.}
  \label{fig:StreamingDopplerShift_1}
\end{figure}

Figure~\ref{fig:StreamingDopplerShift_1} displays steady state solution characteristics for models where we have varied $v_{\max}\in\{0.0,0.1,0.2,0.4\}$.  
From the top left panel, we see that the $\cO(v)$ spectra (dotted) --- which for $v_{\max}>0$ are redshifted relative to the case with $v_{\max}=0$ --- agree well with the analytic, special relativistic results (solid) for the lower values of $v_{\max}$, while the difference between the $\cO(v)$ and the special relativistic results are larger when $v_{\max}=0.4$, which is to be expected since $\cO(v^{2})$ terms are no longer negligible.  
The top right panel shows the RMS energy, defined as
\begin{equation}
	\varepsilon_{\rm RMS} = \sqrt{\int_{D_{\varepsilon}}\cD\varepsilon^{5}d\varepsilon/\int_{D_{\varepsilon}}\cD\varepsilon^{3}d\varepsilon},
	\label{eq:rmsEnergy}
\end{equation}
versus position, computed from the numerical, $\cO(v)$ solution and the analytic, special relativistic solution in Eq.~\eqref{eq:dopplerSpectraSR}.  
Indeed, at $x^{1}=5$, the relative difference in $\varepsilon_{\rm RMS}$ between the $\cO(v)$ and the special relativistic result for $v_{\max}=0.1$ is $4.8\times10^{-3}$, while it is $2.0\times10^{-2}$ and $9.1\times10^{-2}$ for $v_{\max}=0.2$ and $v_{\max}=0.4$, respectively.  
That is, the relative error in the RMS energy increases roughly as $\cO(v^{2})$.  

The bottom panels of Figure~\ref{fig:StreamingDopplerShift_1} show the Eulerian- and comoving-frame number densities ($N$ and $D$, respectively; left), and the Eulerian- and comoving-frame energy densities ($E$ and $J$, respectively; right) versus position.  
Here, 
\begin{equation}
	\big\{\,N,\,D,\,E,\,J\,\big\}=4\pi\int_{D_{\varepsilon}}\big\{\,\cN,\,\cD,\,\cE,\,\cD\varepsilon\,\big\}\varepsilon^{2}d\varepsilon,
\end{equation}
where $\cN$ and $\cE$ are defined in Eqs.~\eqref{eq:eachconservedMoments} and \eqref{eq:conservedEnergy}, respectively.  
Relative to where $v=0$, both $D$ and $J$ are lower in the region where $v>0$, which is expected from the redshifted spectra displayed in the top left panel in Figure~\ref{fig:StreamingDopplerShift_1}.  
The Eulerian-frame quantities, $N$ and $E$, are practically unaffected by the velocity field, and remain relatively constant throughout the spatial domain.  

\subsubsection{Simultaneous Number and Energy Conservation}

Next, we investigate the simultaneous number and energy conservation properties of the scheme as a function of time for $t\in[0,20]$.  
Here, a main challenge stems from the fact that, since the flux factor $h\approx1$, the moments evolve close to the boundary of the realizable set $\cR$.  
With high-order polynomials ($k\ge1$), the solution can become non-realizable in one or more quadrature points in some elements, which then triggers the realizability-enforcing limiter discussed in Section~\ref{sec:realizabilityLimiter}.  
The realizability-enforcing limiter preserves the Eulerian-frame particle number, but not the Eulerian-frame energy, which is the reason for introducing the `energy limiter' in Section~\ref{sec:EnergyLimiter}.  
Here, we demonstrate the performance of the energy limiter, and its effect on the simultaneous number and energy conservation properties of the method.  
Recall from Proposition~\ref{prop:EnergyandMomentumConservation} that the $\cO(v)$ two-moment model is conservative for the Eulerian-frame energy only to $\cO(v^{2})$.  

In the context of the current test, the Eulerian-frame number density satisfies a conservation law of the form
\begin{equation}
	\p_{t}N+\p_{1}F_{N}^{1} = 0.  
\end{equation}
Integration over space $D_{x^{1}}$ and time $[t_{0},t]$ gives
\begin{equation}
	\underbrace{\int_{D_{x^{1}}}[\,N(x^{1},t)-N(x^{1},0)]\,dx^{1}}_{\Delta N_{\rm int}(t)} + \underbrace{\int_{0}^{t}[F_{N}^{1}|_{x^{1}=10}-F_{N}^{1}|_{x^{1}=0}]\,d\tau}_{\Delta N_{\rm ext}(t)} = 0, 
	\label{eq:numberBalance}
\end{equation}
where $\Delta N_{\rm int}(t)$ and $\Delta N_{\rm ext}(t)$ represent the change in the total number of particles, from $t_{0}$ to $t$, \emph{interior} and \emph{exterior} to the domain $D_{x^{1}}$, respectively.  
Since there is no creation or destruction of particles, the sum vanishes.  
We can obtain a similar expression for the Eulerian-frame energy (with $E$ replacing $N$ in Eq.~\eqref{eq:numberBalance}), but for the $\cO(v)$ two-moment model considered here, by Proposition~\ref{prop:EnergyandMomentumConservation}, one would in general expect
\begin{equation}
	\Delta E_{\rm int} + \Delta E_{\rm ext} = \cO(v^{2}),
	\label{eq:energyBalance}
\end{equation}
at the continuous level.  
At the discrete level, with consistent discretization of the left-hand side of the two-moment model, Eulerian-frame energy violations of $\cO(v^{2})$ should be considered optimal.  
(Recall the discussion on this issue specific to the DG scheme from Section~\ref{sec:dgConservation}.)  
For this test, given our chosen spatial resolution and use of quadratic elements, velocity jumps across elements are small, and we expect to observe near optimal Eulerian-frame energy conservation properties.  
However, the acceptable level of Eulerian-frame energy nonconservation is application dependent, and should be considered on a case-by-case basis.  

\begin{figure}[H]
  \centering
  \begin{tabular}{cc}
    \includegraphics[width=0.44\textwidth]{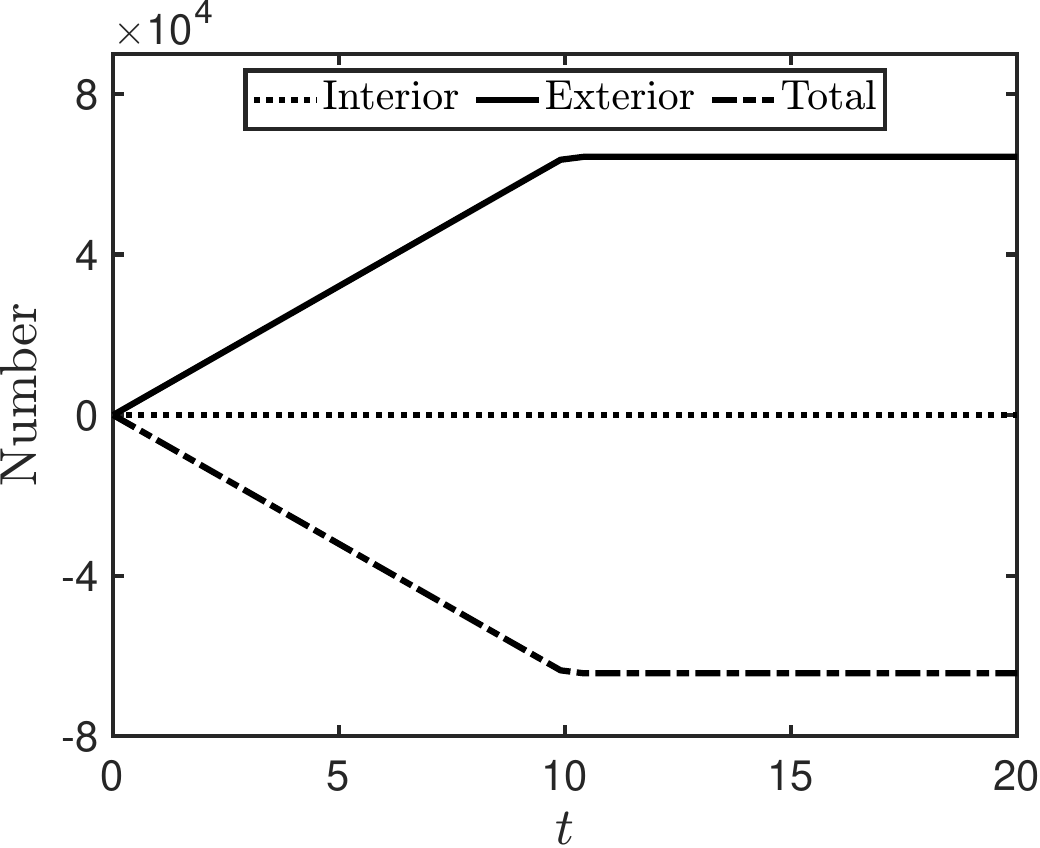} &
    \includegraphics[width=0.44\textwidth]{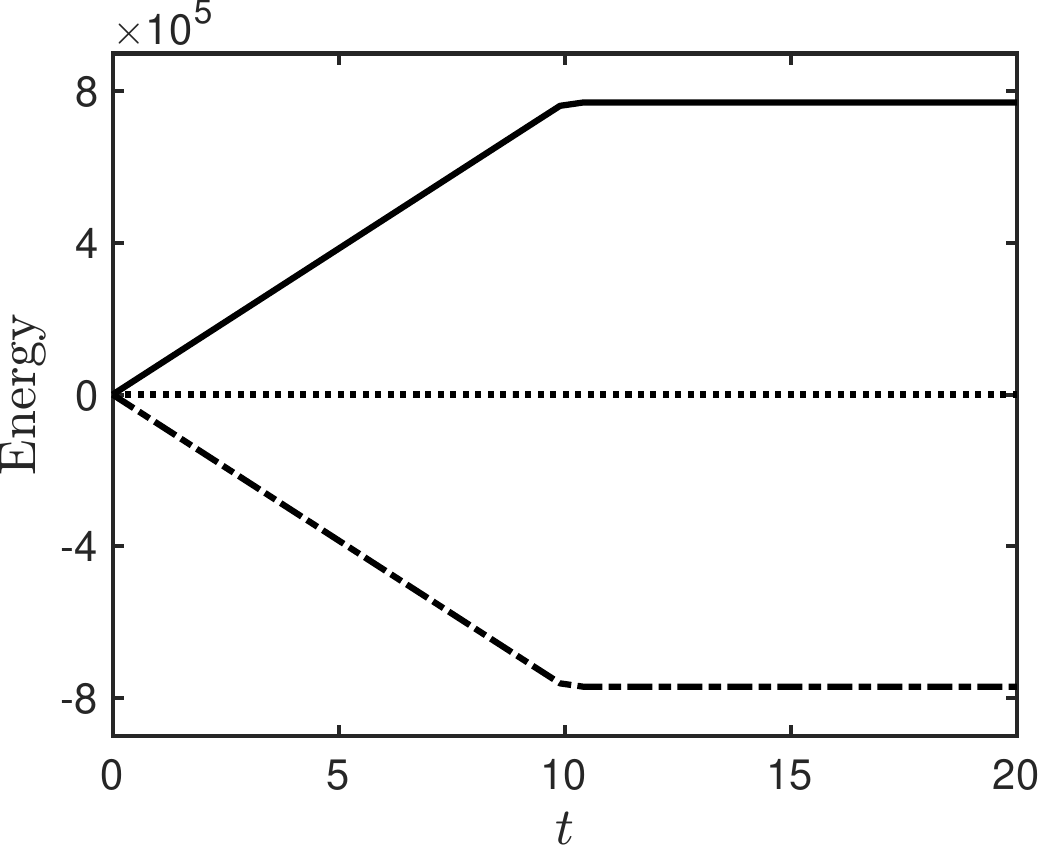}
  \end{tabular}
   \caption{Plot of Eulerian-frame number (left) and energy (right) balance versus time for the streaming Doppler shift problem for a run with $v_{\max}=0.2$. In the left panel, $\Delta N_{\rm int}$ (solid), $\Delta N_{\rm ext}$ (dash-dotted), and $\Delta N_{\rm int}+\Delta N_{\rm ext}$ (dotted), see Eq.~\eqref{eq:numberBalance}, are plotted.  Similarly, in the right panel, $\Delta E_{\rm int}$ (solid), $\Delta E_{\rm ext}$ (dash-dotted), and $\Delta E_{\rm int}+\Delta E_{\rm ext}$ (dotted), see Eq.~\eqref{eq:energyBalance}, are plotted.}
  \label{fig:StreamingDopplerShift_Balance}
\end{figure}

Figure~\ref{fig:StreamingDopplerShift_Balance} plots the number and energy balance versus time for a run with $v_{\max}=0.2$.  
Initially, there are essentially no particles in the computational domain $D_{x^{1}}$, and the flux at the outer boundary is zero.  
For $t>0$, particles enter the domain through the inner boundary, and $\Delta N_{\rm int}$ begins to increase linearly with time, while $\Delta N_{\rm ext}$ decreases at the same rate, and the sum $\Delta N_{\rm int}+\Delta N_{\rm ext}$ remains zero to machine precision (see also Figure~\ref{fig:StreamingDopplerShift_Change}).  
Around $t=10$, the particles that entered the domain at $t=0$ reach the outer boundary, establishing a balance between particles entering and leaving the domain, and the system reaches a steady state where both $\Delta N_{\rm int}$ and $\Delta N_{\rm ext}$ remain unchanged.  
The evolution observed for the Eulerian-frame energy quantities ($\Delta E_{\rm ext}$ and $\Delta E_{\rm ext}$) is similar to that for the particle number, and, on the scale of the ordinate on the right panel in Figure~\ref{fig:StreamingDopplerShift_Balance}, the sum $\Delta E_{\rm ext}+\Delta E_{\rm ext}$ remains close to zero.  

\begin{figure}[H]
  \centering
  \begin{tabular}{cc}
    \includegraphics[width=0.44\textwidth]{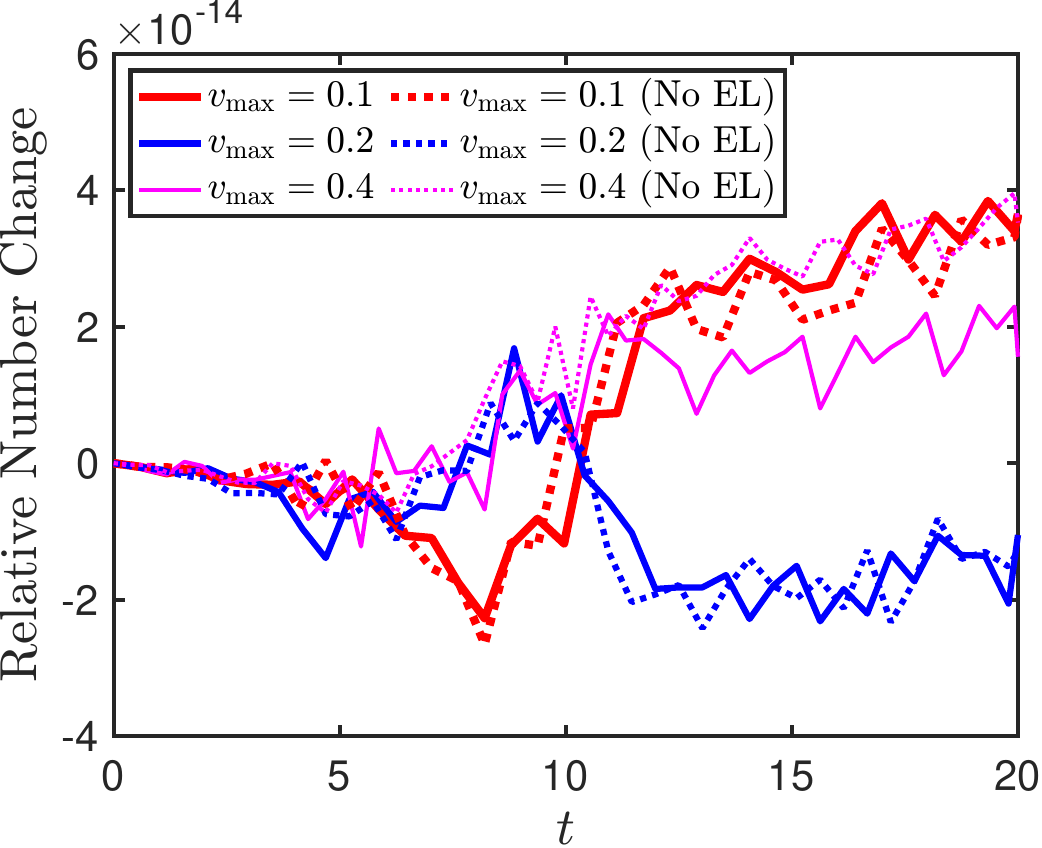} &
    \includegraphics[width=0.44\textwidth]{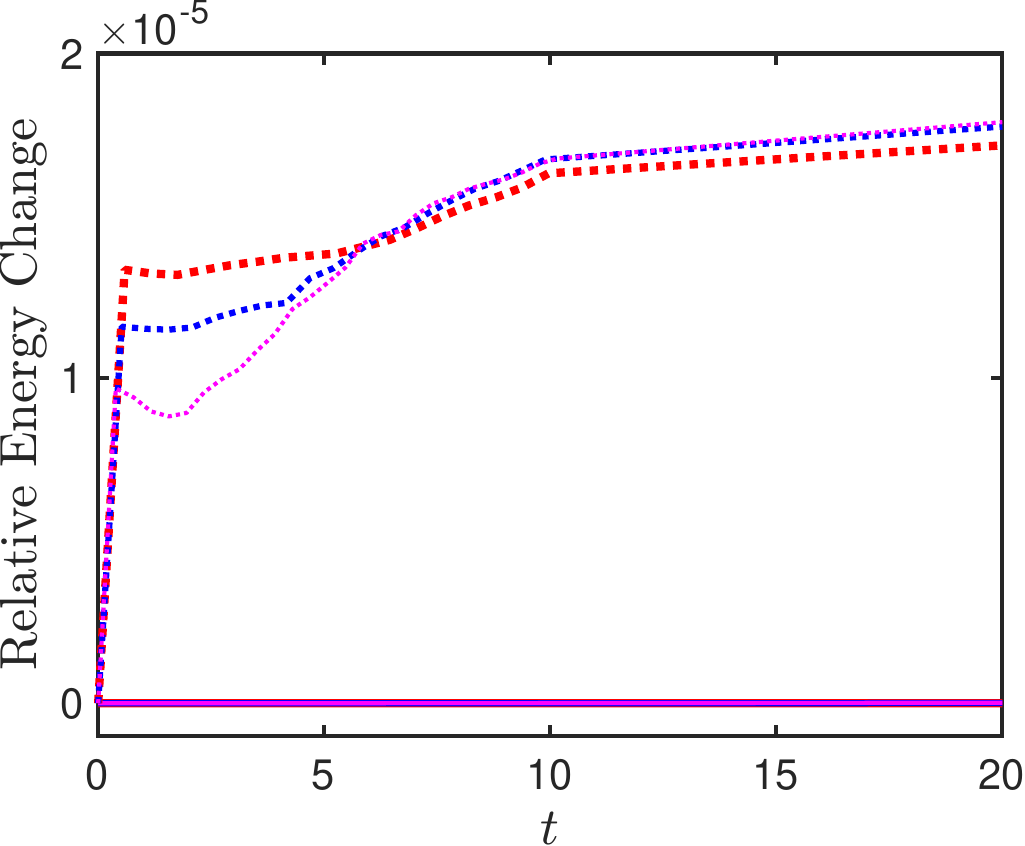} \\
    \includegraphics[width=0.44\textwidth]{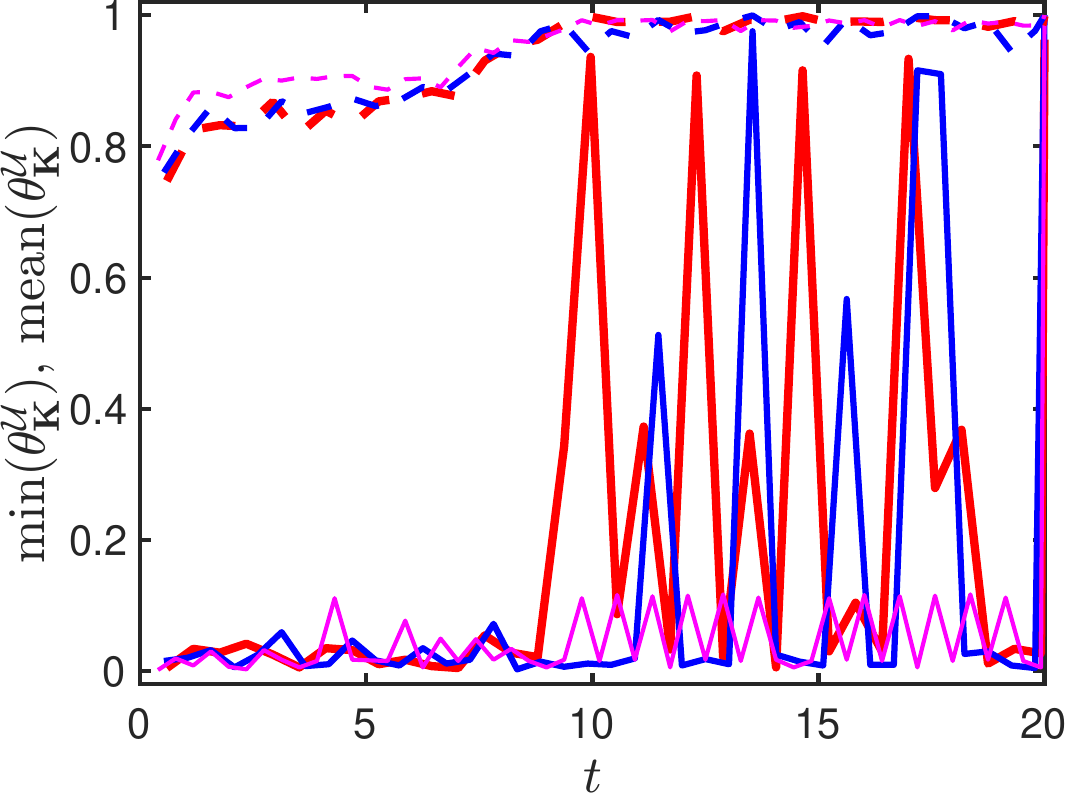} &
    \includegraphics[width=0.44\textwidth]{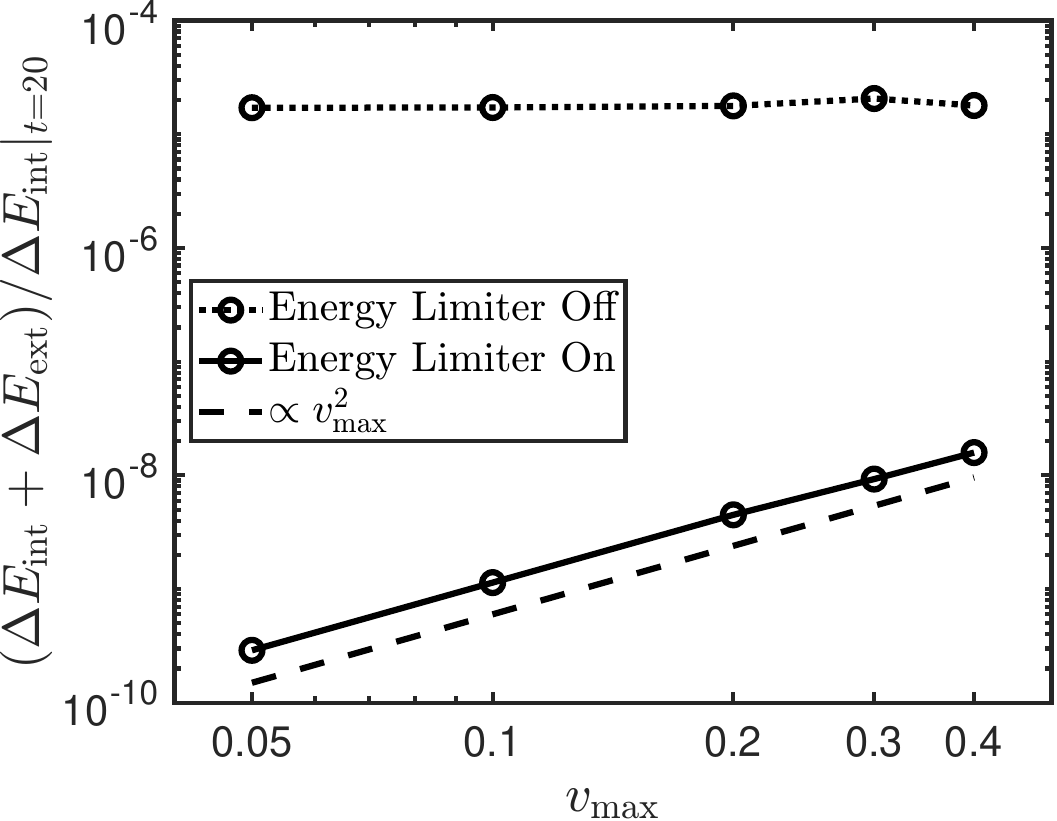}
  \end{tabular}
   \caption{Results from the streaming Doppler shift problem for models with various values of $v_{\max}$.  In the upper left panel the relative change in the Eulerian-frame particle number, defined as $(\Delta N_{\rm int}+\Delta N_{\rm ext})/\Delta N_{\rm int}(t=20)$, is plotted versus time.  Similarly, in the upper right panel, the relative change in the Eulerian-frame energy, defined as $(\Delta E_{\rm int}+\Delta E_{\rm ext})/\Delta E_{\rm int}(t=20)$, is plotted.  In these panels, solid lines represent results obtained with the fiducial algorithm with the energy limiter on, while dotted lines represent results obtained with the energy limiter turned off.  In the lower left panel, the minimum (solid) and mean (dashed) value of the limiter parameter $\theta_{\bK}^{\bcU}$ (see Algorithm~\ref{algo:realizabilityLimiter}) are plotted versus time for the fiducial models with the energy limiter on.  Results with $v_{\max}$ set to $0.1$, $0.2$, and $0.3$ are plotted with red, blue, and magenta curves, respectively.  In the lower right panel, the relative change in the Eulerian-frame energy at $t=20$ is plotted versus $v_{\max}$ for models where the energy limiter is on (solid lines) and off (dotted lines) (the dashed black reference line is proportional to $v_{\max}^{2}$).}
   \label{fig:StreamingDopplerShift_Change}
\end{figure}

Figure~\ref{fig:StreamingDopplerShift_Change} shows select results from runs with $v_{\max}\in\{0.05,0.1,0.2,0.3,0.4\}$ and provides further details on the simultaneous number and energy conservation properties of the scheme when applied to the streaming Doppler shift problem.  
For comparison, we have also run all the models with the energy limiter turned off.  
The relative change in the total number of particles (top left panel) is at the level of machine precision for all values of $v_{\max}$ (with and without the energy limiter), which is expected since the two-moment model is formulated in number conservative form.  

As mentioned earlier, the moments evolve close to the boundary of the realizable set in this test, and the realizability-enforcing limiter is frequently invoked to enforce pointwise realizability in all elements.  
From the lower left panel in Figure~\ref{fig:StreamingDopplerShift_Change}, we observe that the minimum value of the limiter parameter $\theta_{\bK}^{\bcU}$ (solid lines) varies between $0$ and $0.1$ for $t\lesssim10$, while the average value of $\theta_{\bK}^{\bcU}$ (dashed lines) grows from about $0.8$ to $1$ during this time.  
For $t\le10$, a discontinuity in the moments, driven by the inner boundary condition, propagates through the domain and is mainly responsible for triggering the realizability-enforcing limiter.
Recall that $\theta_{\bK}^{\bcU}=0$ implies full limiting where all the moments within an element are set to their respective cell average, while $\theta_{\bK}^{\bcU}=1$ implies no limiting.  
After $t\approx10$, the average $\theta_{\bK}^{\bcU}$ value hovers around unity, but a few elements still require significant limiting when $t>10$, especially for the $v_{\max}=0.4$ model, with the minimum $\theta_{\bK}^{\bcU}$ still dropping down close to zero.  
Closer inspection reveals that a few elements around the location of the velocity gradients ($x^{1}\in[2,3.5)$ and $x^{1}\in[6.5,8)$) require limiting beyond $t=10$.  

The relative change in the Eulerian-frame energy is plotted in the top right panel of Figure~\ref{fig:StreamingDopplerShift_Change}, which reveals a significant improvement in conservation for the fiducial models with the energy limiter turned on (solid lines), when compared to models without the energy limiter (dotted lines).  
(We compute the relative change in number and energy by normalizing by interior values at $t=20$, when the system is in steady state, because the initial interior values are close to zero.)  
For the models without the energy limiter, the relative change in total energy immediately jumps to about $1\times10^{-5}$, and continues to grow for later times, while it remains relatively constant for the fiducial models.  
This implies that the realizability-enforcing limiter is the main driver of Eulerian-frame energy nonconservation for this test, and not the inherent nonconservation properties of the continuum $\cO(v)$ two-moment model.  
(Velocity jumps across elements are small, and we infer from our results that their contribution to energy nonconservation is negligible.)  
The relative change in the Eulerian-frame energy at $t=20$ is plotted versus $v_{\max}$ in the lower right panel of Figure~\ref{fig:StreamingDopplerShift_Change}.  
For the models with the energy limiter turned off, the relative change is essentially independent of $v_{\max}$, while Eulerian-frame energy violations grow as $v_{\max}^{2}$ for the fiducial models.  
Note that the energy limiter only recovers energy conservation violations caused by the realizability-enforcing limiter.  
Energy conservation violations caused by the use of the $\cO(v)$ two-moment model are unaffected by the energy limiter.  
Since we observe the $v_{\max}^{2}$ scaling when using the energy limiter, we posit that the DG discretization maintains consistency with the continuum model on this aspect.  

\subsection{Transparent Shock}
\label{sec:transparentShock}

In this test, we investigate the performance of the method when the background velocity gradient is varied.  
We consider a one-dimensional spatial domain $D_{x^{1}}=[0,2]$, set the opacities $\chi=\sigma=0$, and the velocity field $\vect{v}=(v,0,0)^{\intercal}$, where
\begin{equation}
	v(x^{1}) = \f{1}{2} v_{\max} \times\big[\,1-\tanh\big(\,(x^{1}-1)/H\,\big)\,\big].  
	\label{eq:TransparentShockVelocity}
\end{equation}
We will vary both the velocity magnitude $v_{\max}$ and gradient, parametrized by the length scale $H$.  
The energy domain is again $D_{\varepsilon}=[0,50]$.  
We discretize the spatial and energy domains using $80$ and $32$ elements, respectively, and use quadratic elements ($k=2$) and SSPRK3 time stepping.  
Then, $\Delta x^{1}=0.025$, and $\Delta x^{1}/H=5/6$, $2.5$, and $25$, for $H=3\times10^{-2}$, $10^{-2}$, and $10^{-3}$, respectively.  
We use the same boundary conditions as in the Doppler shift test, and the moments are initially set to $\cD=1\times 10^{-8}$ and $\cI^{1}=0$ for all $(x^{1},\varepsilon)\in D_{x^{1}}\times D_{\varepsilon}$.  
With the given initial and boundary conditions, the equations are integrated until $t=3$, when an approximate steady state has been established.  

Figure~\ref{fig:TransparentShock_Profiles} shows velocity profiles and comoving-frame and Eulerian-frame number densities versus position around the `shock' ($x^{1}\in[0.9,1.1]$) for the different values of $H$ for the case with $v_{\max}=-0.1$.  
(The markers indicate locations of LG quadrature points in each spatial element.)  
For $H=3\times10^{-2}$, the shock is resolved by the spatial grid, for $H=10^{-2}$ it is under-resolved, while for $H=10^{-3}$, the velocity profile is discontinuous.  

The comoving-frame number densities increase across the velocity gradient because of the Doppler effect, increasing the particle energy measured by the comoving observer, who is moving towards the inner boundary.  
Beyond the shock, the values for the computed comoving-frame number densities (solid lines) are about $0.5~\%$ higher than the analytic values obtained using Eq.~\eqref{eq:dopplerSpectraSR} (dashed lines), and this fact is independent of the value of $H$.  
The Eulerian-frame number densities, which should remain unaffected by the presence of the background velocity, are essentially constant across the shock.  
These results indicate that the method is able to capture Doppler shifts correctly, even when velocity gradients are large.  

\begin{figure}[H]
  \centering
  \begin{tabular}{cc}
    \includegraphics[width=0.44\textwidth]{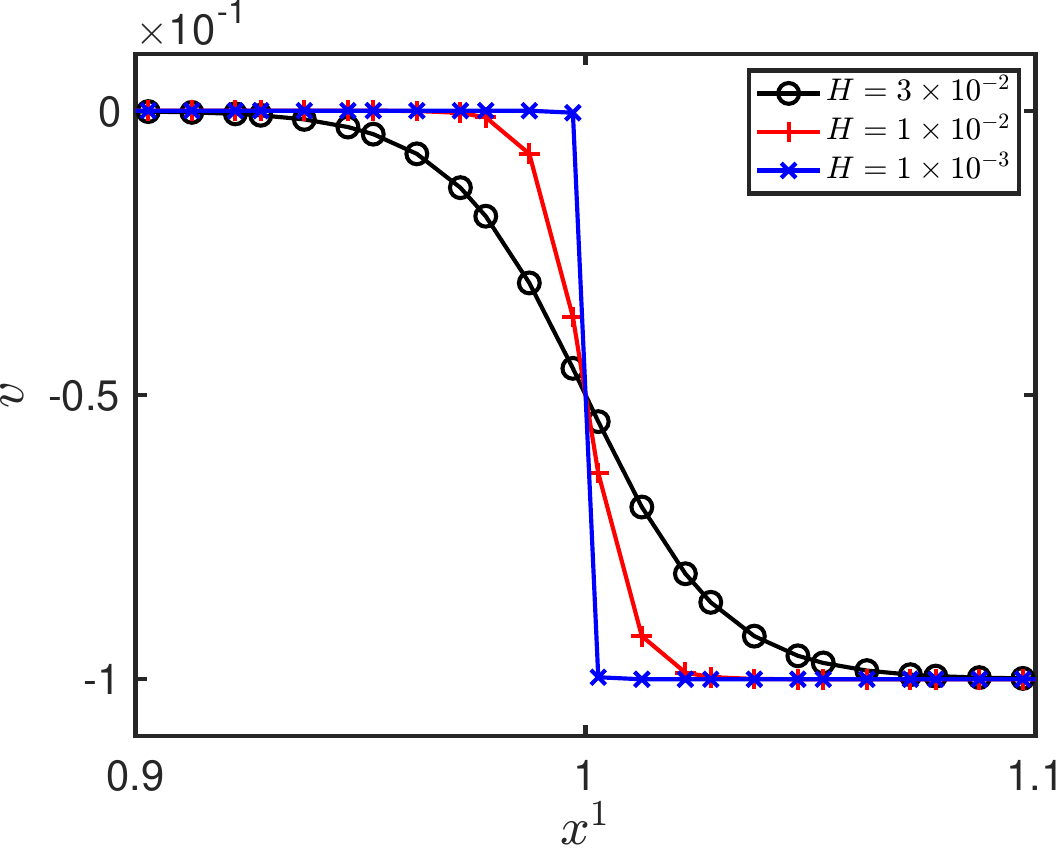} &
    \includegraphics[width=0.44\textwidth]{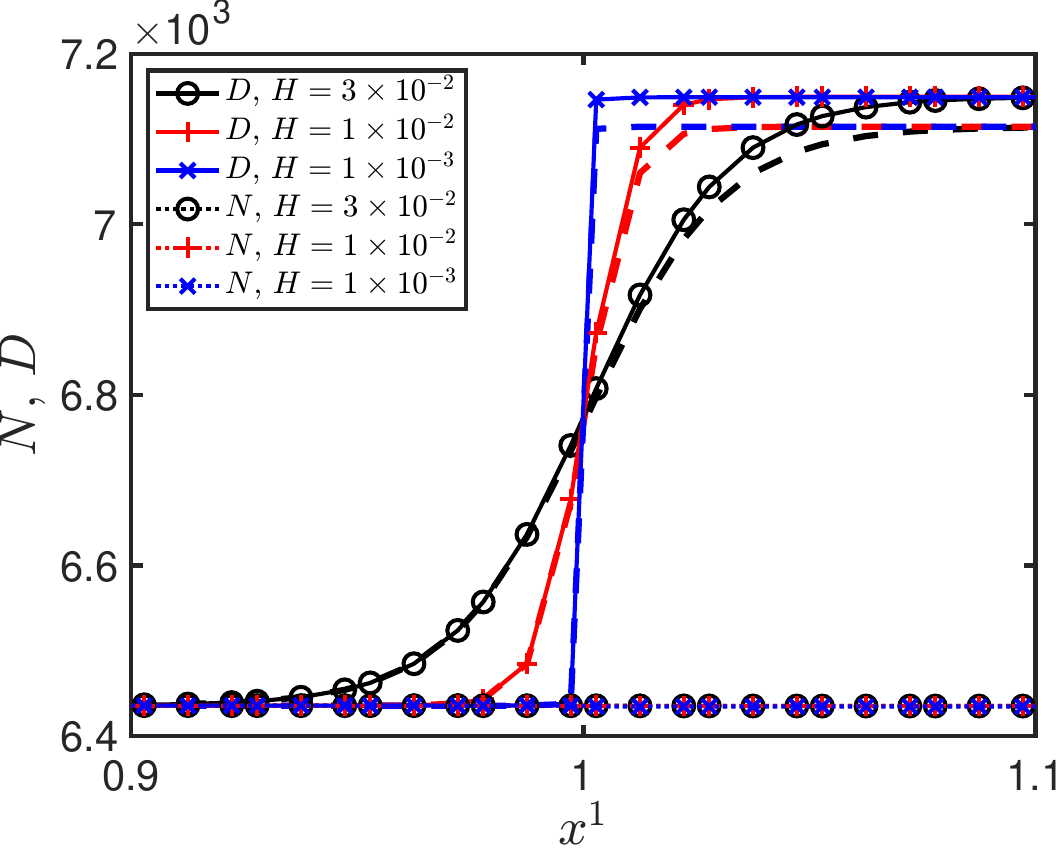}
  \end{tabular}
   \caption{Results from the transparent shock problem for $v_{\max}=-0.1$ and various values of the shock width parameter $H$, defined in Eq.~\eqref{eq:TransparentShockVelocity}.  
   In the left panel, velocity profiles are plotted versus position for $H=3\times10^{-2}$, $10^{-2}$, and $10^{-3}$ (black, red, and blue curves, respectively).  
   In the right panel, the numerical and analytic (special relativistic) comoving-frame number densities (solid lines with markers and dashed lines, respectively), and the numerical Eulerian-frame number densities (dotted lines with markers) are plotted versus position.  
   For the results displayed in the right panel, the line colors correspond to the velocity profile with the matching color plotted in the left panel.}
   \label{fig:TransparentShock_Profiles}
\end{figure}

In Figure~\ref{fig:TransparentShock_Energies}, the left panel displays the relative change in the Eulerian-frame energy, defined by the left-hand side of Eq.~\eqref{eq:energyBalance}, versus $|v_{\max}|$ at $t=3$, for various values of $H$.  
The right panel displays the relative error in $\varepsilon_{\rm RMS}$.  
Both panels display results obtained with and without the energy limiter.  
(The relative change in the Eulerian-frame particle number, not shown, is at the level of machine precision for all models.)  
Considering the results obtained with the energy limiter active, in the left panel we observe that, for a given value of $H$, the relative change in the total energy increases with increasing $|v_{\max}|$, roughly proportional to $|v_{\max}|^{2}$.  
Also, for a given value of $|v_{\max}|$, the relative change in total energy increases with decreasing shock width $H$.  
For the models with the energy limiter turned off, the behavior is different in a large region of the $(v_{\max},H)$-space.  
With $H=3\times10^{-2}$ (dotted black line), the relative change in the Eulerian-frame energy at $t=3$ is around $2\times10^{-5}$; independent of $v_{\max}$.  
This can be attributed to the realizability-enforcing limiter.  
For the models with $H=10^{-2}$ (dotted red line), the energy change is roughly constant until $|v_{\max}|=0.1$, when the relative energy change begins to increase with $|v_{\max}|$ in a manner similar to the models with the energy limiter activated (solid red line).  
The models with the steepest velocity gradient ($H=10^{-3}$; dotted blue line) follow the corresponding models with active energy limiter, and the relative change in the Eulerian-frame energy increases as $|v_{\max}|^{2}$ for all $|v_{\max}|$.  
From this we conclude that the energy limiter can help to recover the $\cO(v^{2})$ Eulerian-frame energy conservation property of the $\cO(v)$ two-moment model for small velocities and velocity gradients.  
The relative error in $\varepsilon_{\rm RMS}$, which increases as $|v_{\max}|^{2}$ for all models, is essentially unaffected by the energy limiter.  

\begin{figure}[H]
  \centering
  \begin{tabular}{cc}
    \includegraphics[width=0.45\textwidth]{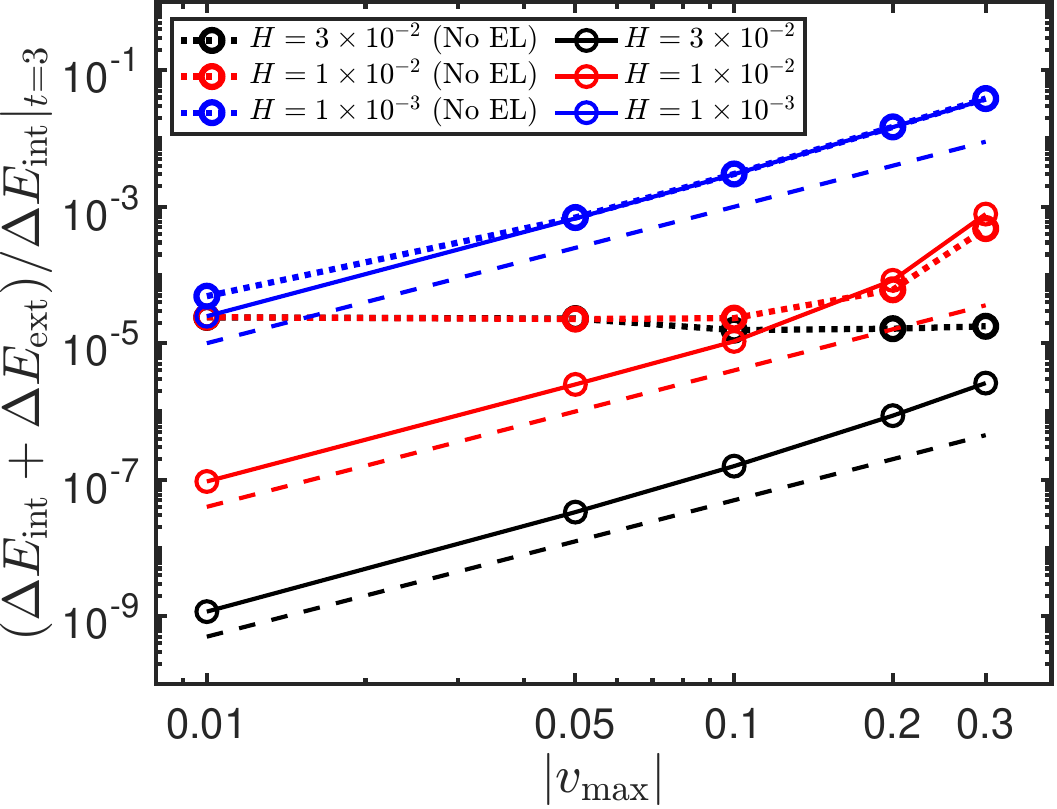} &
    \includegraphics[width=0.45\textwidth]{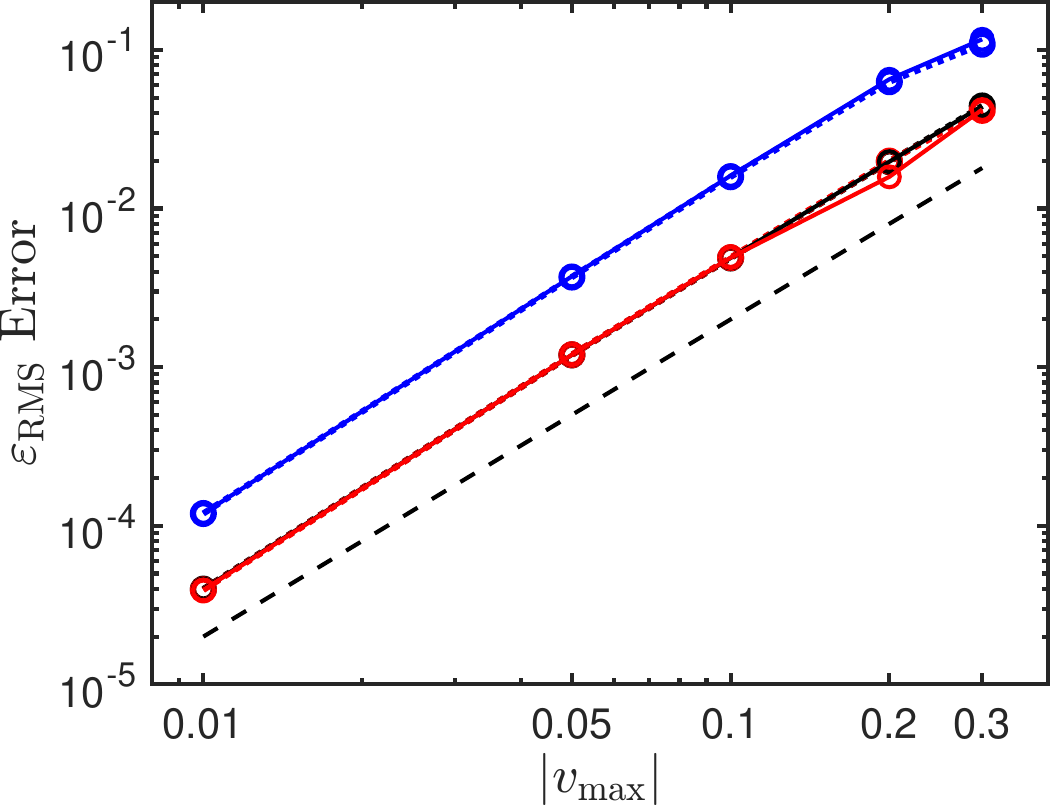}
  \end{tabular}
   \caption{Results for the transparent shock problem, plotted versus $|v_{\max}|$.  The left panel displays the relative change in the Eulerian-frame energy at $t=3$ for models where the energy limiter is on (solid lines) and off (dotted lines).  The right panel displays the absolute relative difference in $\varepsilon_{\rm RMS}$ with respect to the exact solution from special relativity at $x^{1}=2$.  In both panels, black, red, and blue curves correspond to $H=3\times10^{-2}$, $10^{-2}$, and $10^{-3}$, respectively, and the dashed references line are proportional to $|v_{\max}|^{2}$.  }
   \label{fig:TransparentShock_Energies}
\end{figure}

\subsection{Transparent Vortex}
\label{sec:transparentVortex}

The final test, inspired by the test in Section~4.2.3 of \cite{just_etal_2015}, considers evolution in a two-dimensional spatial domain $D_{x^{1}}\times D_{x^{2}}=[-5,5]\times[-5,5]$.  
We set the opacities $\chi=\sigma=0$, and the velocity field is given by $\vect{v}=[v^{1},v^{2},0]^{\intercal}$, where
\begin{subequations}
	\begin{align}
		v^{1}(x^{1},x^{2}) &= 
		- v_{\max} \, x^{2} \, \exp\big[\,(1-r^{2})/2\,\big], \\
		v^{2}(x^{1},x^{2}) &= \hspace{8pt}
		v_{\max} \, x^{1} \, \exp\big[\,(1-r^{2})/2\,\big],
	\end{align}
	\label{eq:vortexVelocityField}
\end{subequations}
and $r=\sqrt{(x^{1})^{2}+(x^{2})^{2}}$.  
The energy domain is $D_{\varepsilon}=[0,50]$, and we discretize the spatial and energy domains using $48\times48$ and $32$ elements, respectively.  
We use quadratic elements ($k=2$) and SSPRK3 time stepping.  
The upper left panel in Figure~\ref{fig:TransparentVortex} shows the velocity field for the case with $v_{\max}=0.1$.  
The main purpose of this test is to investigate the robustness of the method in configurations where the radiation field propagates through a spatially variable velocity field with various relative angles between the radiation flux and velocity vectors.  
The moments are initially set to $\cD=1\times 10^{-8}$ and $\cI^{1}=\cI^{2}=0$ for all $(x^{1},x^{2},\varepsilon)\in D_{x^{1}}\times D_{x^{2}}\times D_{\varepsilon}$.  
At the inner $x^{1}$ boundary, we impose an incoming, radiation field with a Fermi-Dirac spectrum:  We set $\cD(\varepsilon,x^{1}=-5,x^{2})=0.05/[\exp(\varepsilon/3-3)+1]$, $\cI^{1}(\varepsilon,x^{1}=-5,x^{2})=0.95\times\cD(\varepsilon,x^{1}=-5,x^{2})$, and $\cI^{2}(\varepsilon,x^{1}=-5,x^{2})=0$, so that the flux factor is $h=0.95$.  
With these initial and boundary conditions, the moment equations are integrated until a steady state is reached ($t=20$).  

\begin{figure}[H]
  \centering
  \begin{tabular}{cc}
    \includegraphics[width=0.45\textwidth]{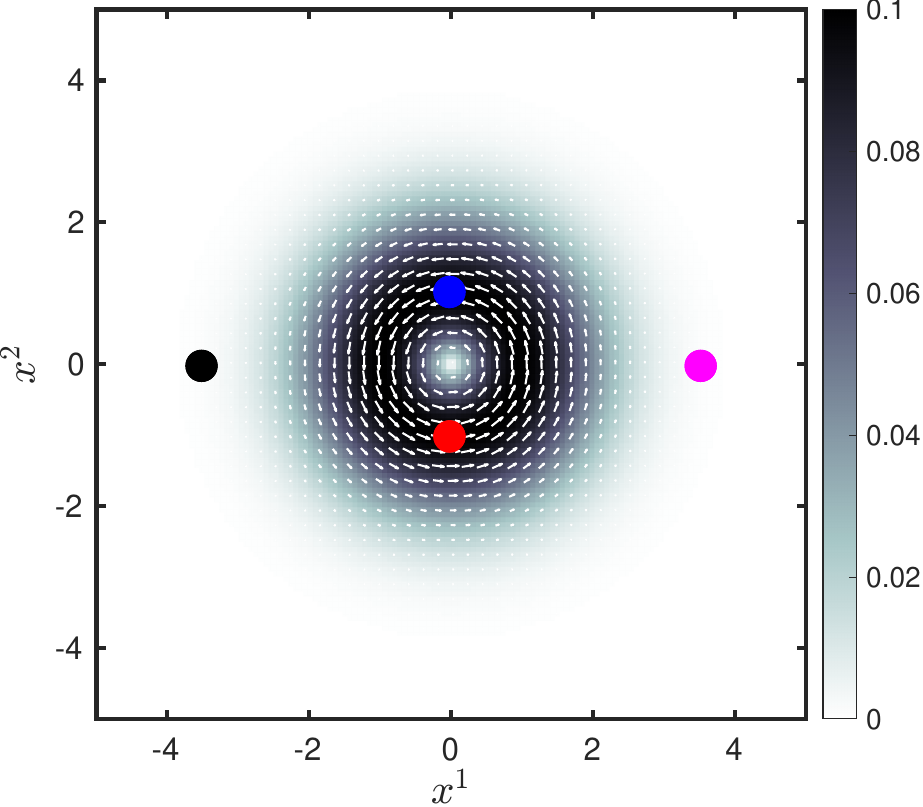} &
    \includegraphics[width=0.40\textwidth]{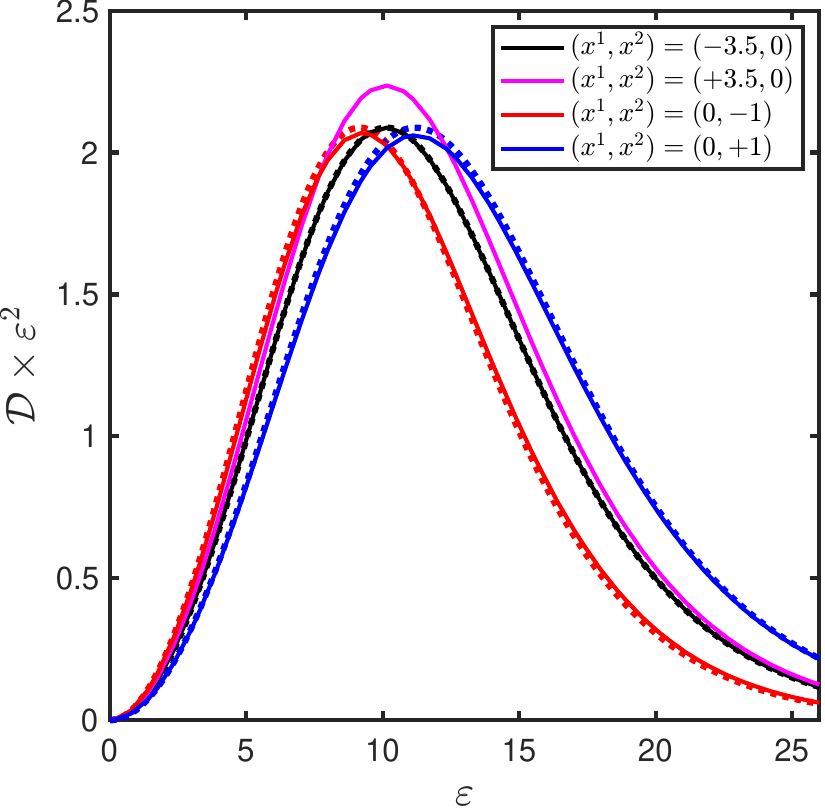} \\
    \includegraphics[width=0.425\textwidth]{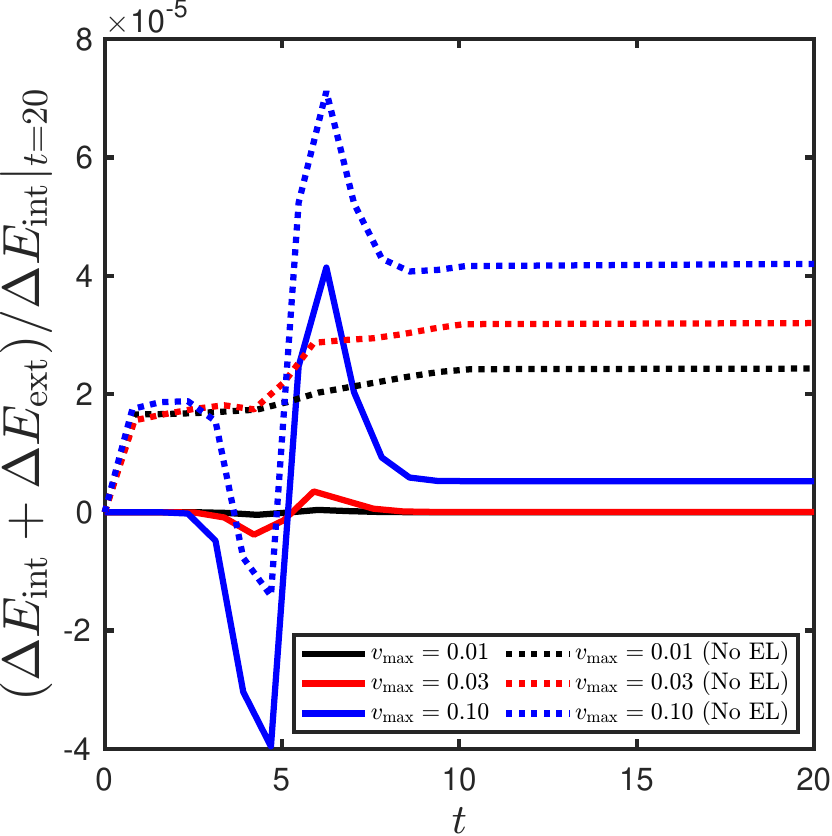} &
    \includegraphics[width=0.425\textwidth]{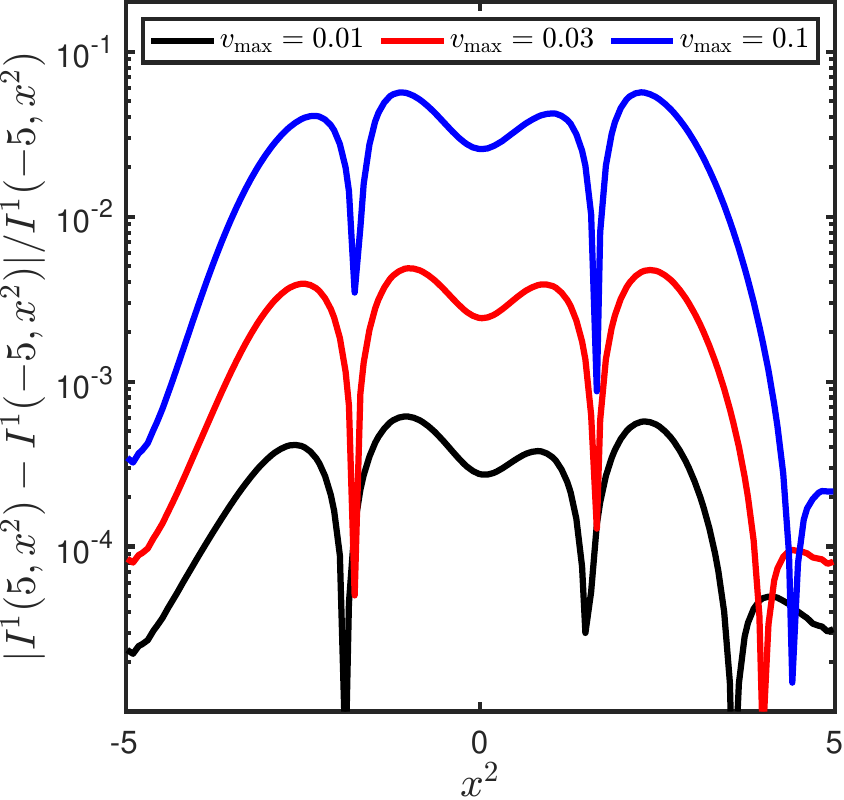}
  \end{tabular}
   \caption{Results for the transparent vortex problem.  In the upper left panel, the magnitude of the velocity, for the case with $v_{\max}=0.1$, is displayed in grayscale with velocity vectors overlaid.  The black, magenta, red, and blue markers indicate spatial positions for which we plot numerical energy spectra in solid lines in the upper right panel, with line colors corresponding to the marker colors in the upper left panel.  The dotted lines are analytic spectra obtained from special relativistic considerations using the local three-velocity.  In the lower left panel, the relative change in Eulerian-frame energy is plotted versus time for models with $v_{\max}$ in Eq.~\eqref{eq:vortexVelocityField} set to $0.01$ (black lines), $0.03$ (red lines), and $0.1$ (blue lines).  Results obtained with and without the energy limiter are plotted using solid and dotted lines, respectively.  The lower right panel plots the relative difference between the incoming and outgoing particle fluxes in the $x^{1}$-direction.}
   \label{fig:TransparentVortex}
\end{figure}

In the upper right panel in Figure \ref{fig:TransparentVortex}, the solid lines represent numerical energy spectra at spatial locations indicated with solid markers of matching color in upper left panel.  
At the location of the black marker the velocity is close to zero, and thus the black line represents the spectrum of the incoming radiation.  
The red and blue markers are located where $\vect{v}=(v_{\max},0,0)^{\intercal}$ and $\vect{v}=(-v_{\max},0,0)^{\intercal}$, respectively, and the spectra at these locations are, respectively, red- and blue-shifted relative to the spectrum sampled at the location of the black marker.  
Analytic spectra at the locations of the black, red, and blue markers are plotted with dotted lines, which indicate good agreement between numerical and analytical solutions across all energies.  
At the locations of the black, red, and blue markers, we find that $\varepsilon_{\rm RMS}$ is approximately $15.6$, $14.2$, and $17.2$, respectively.  
At the location of the magenta marker, which is placed on the opposite side of the vortex (relative to the black marker), the velocity is again close to zero, and it is expected that the spectrum at this location agrees with the spectrum at the location of the black marker.  
Comparing the solid black and magenta lines in the upper right panel, we observe that the spectral number density is consistently higher at the location of the magenta marker (by a constant factor of about $1.07$).  
Comparing $\varepsilon_{\rm RMS}$ at the two locations, we find that the relative difference is less than $10^{-4}$.  

The lower left panel in Figure~\ref{fig:TransparentVortex} plots the relative change in the Eulerian-frame energy versus time for models with $v_{\max}\in\{0.01,0.03,0.1\}$.  
Results obtained with the energy limiter on are plotted with solid lines, while dotted lines correspond to results with the energy limiter turned off.  
For all models, the relative change in the Eulerian-frame energy is less than $10^{-4}$.  
For the models with $v_{\max}=0.1$, the relative change reaches the largest amplitudes for $t\in[4,7]$, when a radiation front, driven by the boundary condition at $x^{1}=-5$, propagates through the vortex.  
The model with the energy limiter makes a better recovery than the corresponding model with the energy limiter turned off.  
For smaller $v_{\max}$, the relative change in the Eulerian-frame energy is clearly much smaller when the energy limiter is active.  
These results demonstrate the contribution to Eulerian-frame energy nonconservation caused by the realizability-enforcing limiter.  
For both suites of models (energy limiter on or off), the relative change in the Eulerian-frame number (not shown) is at the level of machine precision for all models.  

The lower right panel in Figure~\ref{fig:TransparentVortex}, similar to Figure~6~(b) in \cite{just_etal_2015}, shows, for $t=20$, the relative difference between the energy integrated $x^{1}$-component of the number flux densities evaluated at the inner and outer boundaries of $D_{x^{1}}$, defined as $|I^{1}(5,x^{2})-I^{1}(-5,x^{2})|/I^{1}(-5,x^{2})$.  
As discussed by Just et al.~\cite{just_etal_2015}, this quantity should vanish for exact calculations, while errors of $\cO(v^{2})$ are to be expected for the $\cO(v)$ two-moment model.  
Comparing with their results, the curves plotted in our figure share similar features.  
Moreover, for $v_{\max}=0.01$, the maximum relative difference is $6.15\times10^{-4}$, for $v_{\max}=0.03$ it is $4.87\times10^{-3}$, while it is $5.68\times10^{-2}$ for $v_{\max}=0.1$; i.e., the maximum error grows as $v_{\max}^{2}$.  

Despite the growing (with $v_{\max}$) relative difference between the number fluxes at the inner and outer boundaries in the $x^{1}$-direction, we point out that, due to number conservation, the integrated number fluxes through the inner and outer boundaries balance each other.  
That is, in the steady state at $t=20$, $\int_{D_{x^{2}}}I^{1}(-5,x^{2})\,dx^{2}=\int_{D_{x^{2}}}I^{1}(5,x^{2})\,dx^{2}$.  
However, the distribution of particles along the $x^{2}$-direction becomes nonuniform in the wake of the vortex, while a uniform distribution is expected as $|\vect{v}|\to0$.  
We illustrate this further in Figure~\ref{fig:TransparentVortexII}.  
The left panel shows that, within the vortex ($r\lesssim2$), the comoving-frame number density is higher than the reference value $D_{0}$ for $x^{2}>0$, and lower than $D_{0}$ for $x^{2}<0$, which is consistent with the Doppler shift of the spectra in the respective regions.  
In the wake of the vortex, the comoving-frame number density is relatively higher in the region centered around $x^{2}=0$, while it is lower further away (compare red and blue regions for $x^{1}\gtrsim2$ in the left panel in Figure~\ref{fig:TransparentVortexII}).  
The Eulerian-frame number density is relatively unaffected by the vortex for $x^{1}<0$, but exhibits a spatial distribution similar to the comoving-frame number density in the wake.  
In contrast, the spatial distribution of the RMS energy is more consistent with expectations:  Within the vortex, $\varepsilon_{\rm RMS}>\varepsilon_{{\rm RMS},0}$ for $x^{2}>0$, while  $\varepsilon_{\rm RMS}<\varepsilon_{{\rm RMS},0}$ for $x^{2}<0$.  
Moreover, the RMS energy returns to the reference value in the wake of the vortex, with almost uniform distribution along the $x^{2}$-direction.  
We do not have a complete theoretical explanation for the spatial distribution of the number densities in the wake of the vortex, but suspect that the two-moment approximation and the associated closure, which assumes that the radiation field is axisymmetric about a preferred direction in momentum space \cite{levermore_1984}, is insufficient for capturing relativistic aberration effects.  

\begin{figure}[H]
  \centering
  \begin{tabular}{ccc}
    \includegraphics[width=0.31\textwidth]{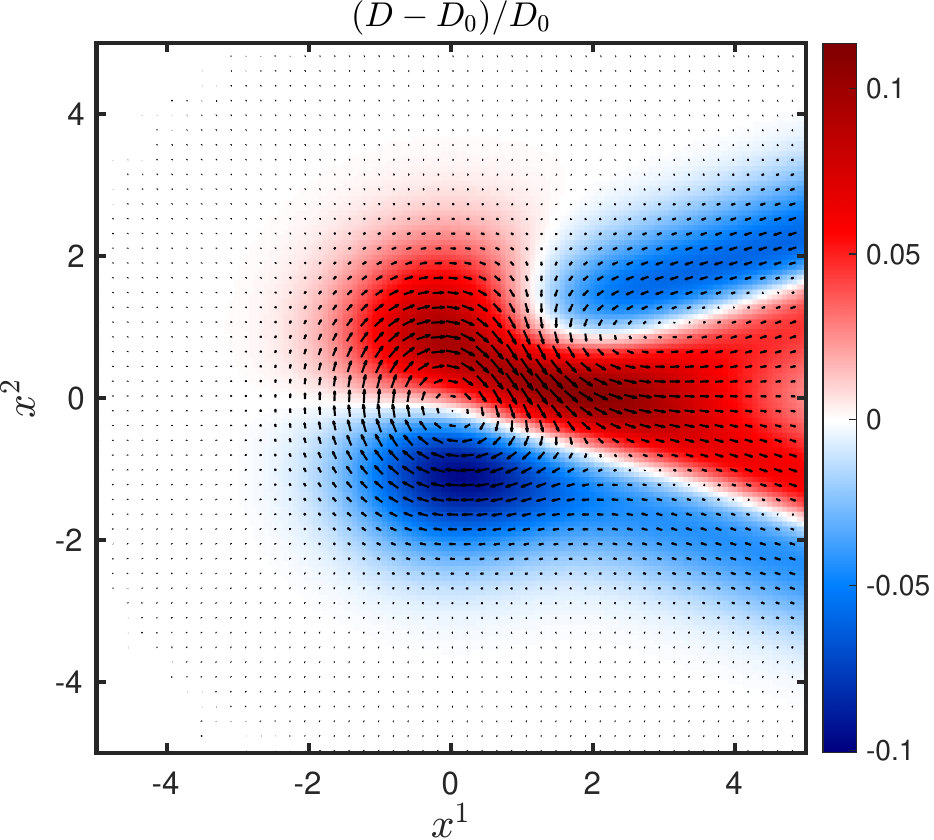} &
    \includegraphics[width=0.31\textwidth]{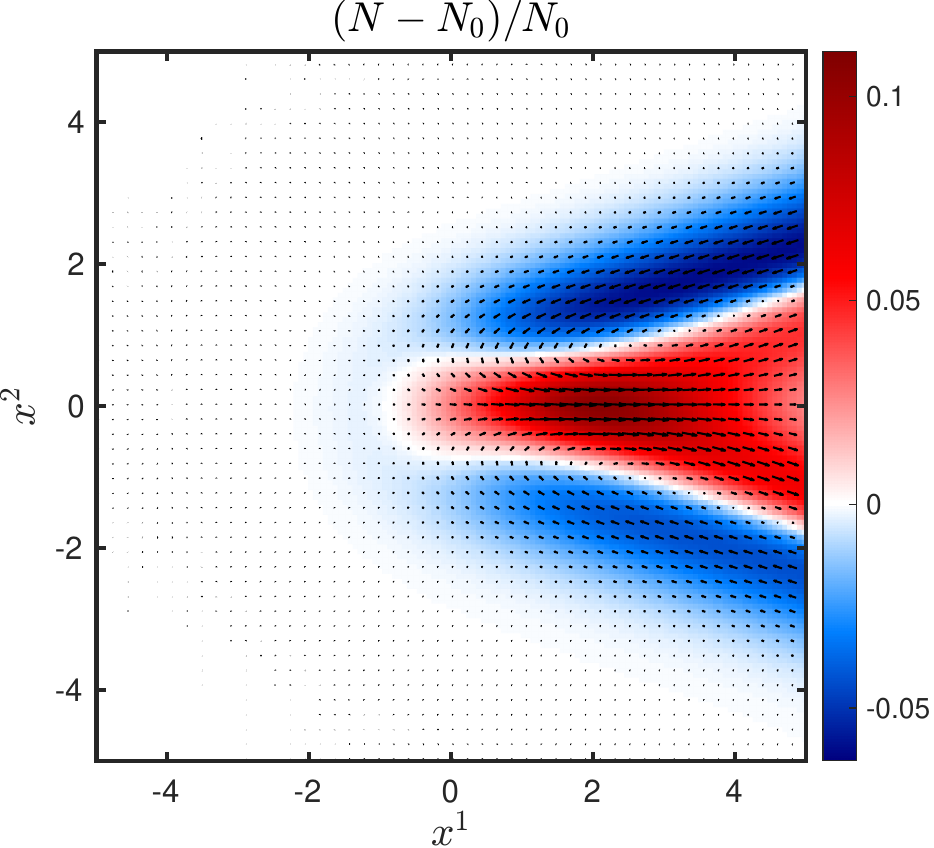} &
    \includegraphics[width=0.31\textwidth]{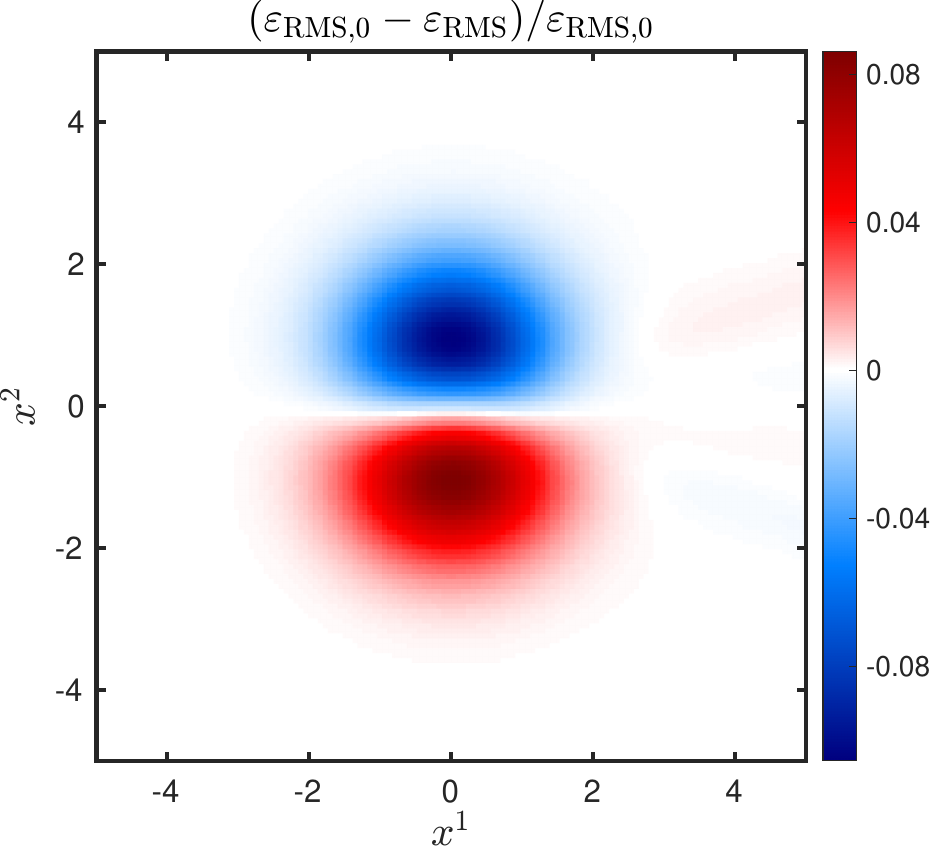}
  \end{tabular}
   \caption{Results for the transparent vortex problem at $t=20$ for a model with $v_{\max}=0.1$.  The left panel shows the relative deviation in comoving-frame number density from $D_{0}=D(x^{1}=-5,x^{2})$, $(D-D_{0})/D_{0}$, with vectors of the comoving-frame number flux $(I^{1}-I_{0}^{1},I^{2})^{\intercal}$ overlaid, where $I_{0}^{1}=I^{1}(x^{1}=-5,x^{2})$ is the first component of the comoving-frame number flux density at the inner boundary in the $x^{1}$-direction, which is subtracted to better illustrate the flow, since $|I^{2}|\ll|I^{1}|$ generally holds.  Similarly, the middle panel shows the corresponding relative deviation in the Eulerian-frame number density $(N-N_{0})/N_{0}$, with vectors of the Eulerian-frame number flux $(F_{N}^{1}-F_{N,0}^{1},F_{N}^{2})^{\intercal}$.  The right panel shows the relative deviation in the RMS energy, $(\varepsilon_{{\rm RMS},0}-\varepsilon_{\rm RMS})/\varepsilon_{{\rm RMS},0}$, where $\varepsilon_{{\rm RMS},0}=\varepsilon_{\rm RMS}(x^{1}=-5,x^{2})$.}
   \label{fig:TransparentVortexII}
\end{figure}

\subsection{Performance Evaluation}
\label{sec:performance}

To demonstrate the GPU functionality and performance characteristics of the DG-IMEX method as implemented in \thornado, we consider the Streaming Doppler Shift test, described in Section~\ref{sec:streamingDopplerShift}, with $v_{\max}=0.1$.
To more accurately capture a production workload, the tests are performed in three spatial dimensions, with the number of elements similar to what would be used for a single process invoking \thornado\ in a multiphysics simulation.  
The benchmark is run in two configurations, using tensor product polynomials of degree $k=1$ and $k=2$, respectively.  
The SSPRK2 time stepper is used for both configurations.  
For $k=1$, we use 16 energy elements and $96\times3\times3$ spatial elements, while 12 energy elements and $64\times2\times2$ spatial elements are used for $k=2$ --- thus keeping the total number of spatial degrees of freedom the same.  
Our goal is to provide a high-level demonstration of performance characteristics and the relative cost of main algorithmic components, while we defer a rigorous performance analysis to future work.  

The tests are performed on a single node of the Summit computer at the Oak Ridge Leadership Computing Facility (OLCF).  
Each Summit node has 2 IBM POWER9 CPUs and 6 NVIDIA V100 GPUs, but here we limit our benchmarks to a single CPU or GPU.  
For the CPU runs, we use seven cores with one thread per core as this is the number of cores that would be available to one process if we divide the resources equally with one GPU per process.  
All runs use version 22.5 of the NVIDIA \texttt{nvfortran} compiler with standard \texttt{-O2} optimizations.  
Optimized linear algebra libraries are provided by IBM ESSL (v6.3.0) on the CPU and NVIDIA cuBLAS (v11.0.3) on the GPU.  
For the GPU runs, all computations are done on the GPU using OpenACC and libraries; the CPU process is only used to launch kernels and manage data transfer.  
In both cases, the salient metric is wall-time per time step (lower is better).  

Figure~\ref{fig:SDS_walltime} shows a breakdown of the relative cost associated with evaluating the major components of the explicit phase-space advection operator.  
The polynomial degree has little effect on the absolute wall-time, especially for the GPU runs.
For the CPU runs, the relative cost of linear algebra (\texttt{MatMul}) is somewhat higher when $k=2$.
As can be seen comparing the right and left panels, the initial guess in the conserved-to-primitive calculation can have a non-trivial impact on the total wall-time by reducing the total number of solver iterations.
We measure a total speedup factor of 8--10 for the V100 relative to the multi-core CPU runs on the POWER9.  
Notably, the relative cost for linear algebra and limiters becomes negligible when using the GPU, and the majority of the computational cost is shifted to the iterative conserved-to-primitive calculations.
We speculate that one approach to further improve the performance would be to combine the calculation of all of the primitive moments on the quadrature set $\widetilde{S}_{\otimes}^{\bK}$, defined in Eq.~\eqref{eq:AllSetUnion}, into a single kernel, rather than to calculate them separately for each evaluation of $\vect{\mathcal{F}}^{i}$ and $\vect{\mathcal{F}}^{\varepsilon}$, defined in Eqs.~\eqref{eq:bilinearFormAdvectionPosition}--\eqref{eq:bilinearFormAdvectionEnergy}, which results in some duplicate evaluations.  
While these savings may be significant for the phase-space advection problem considered here, refactoring will be considered in the context of a more physics-complete implementation.  
With more realistic collision terms included, the relative cost of the explicit phase-space advection part is expected to be small (see, e.g., \cite{laiu_etal_2020,laiu_etal_2021}).  

\begin{figure}[H]
	\centering
	\begin{tabular}{cc}
		\includegraphics[width=0.45\textwidth]{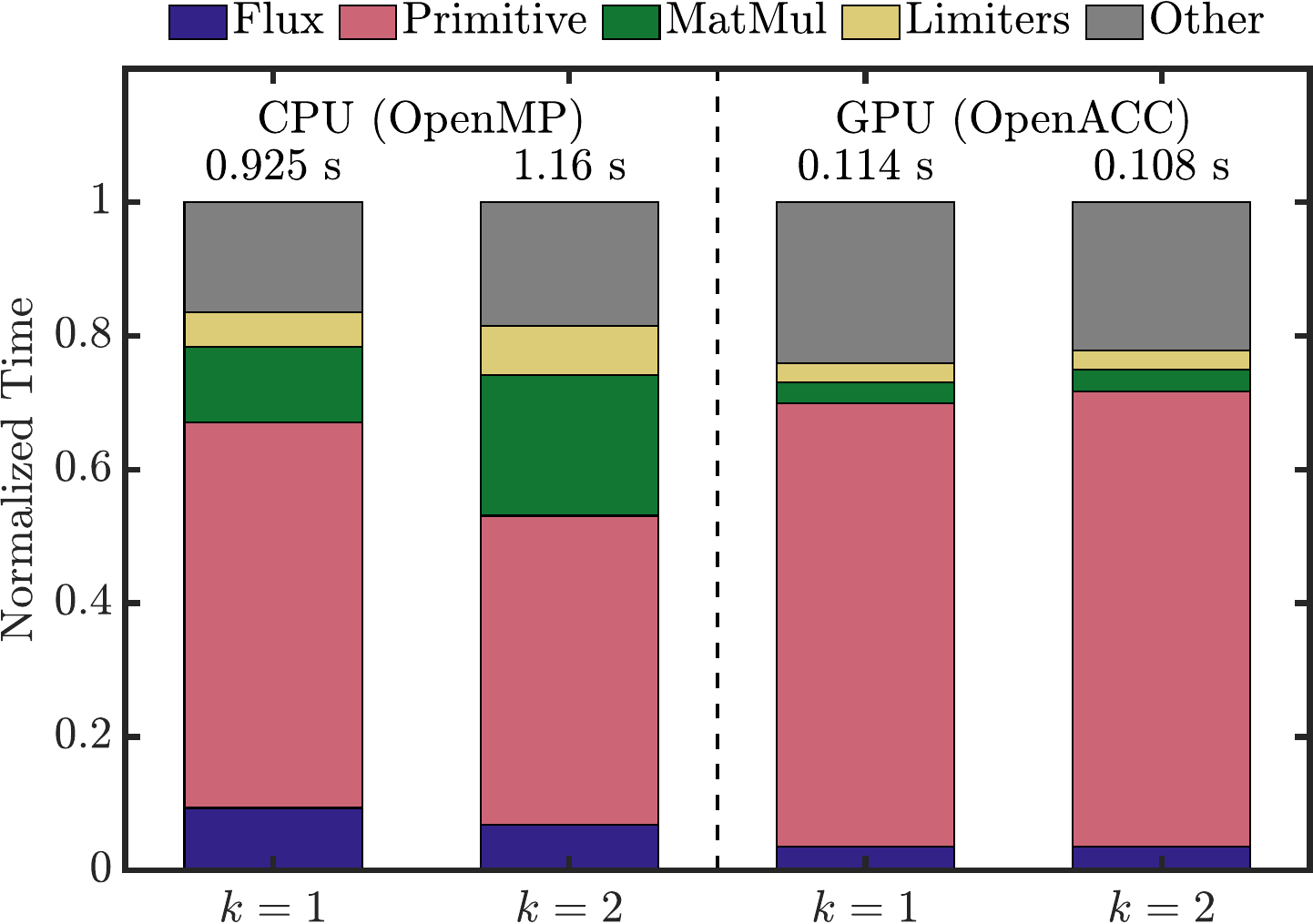} &
		\includegraphics[width=0.45\textwidth]{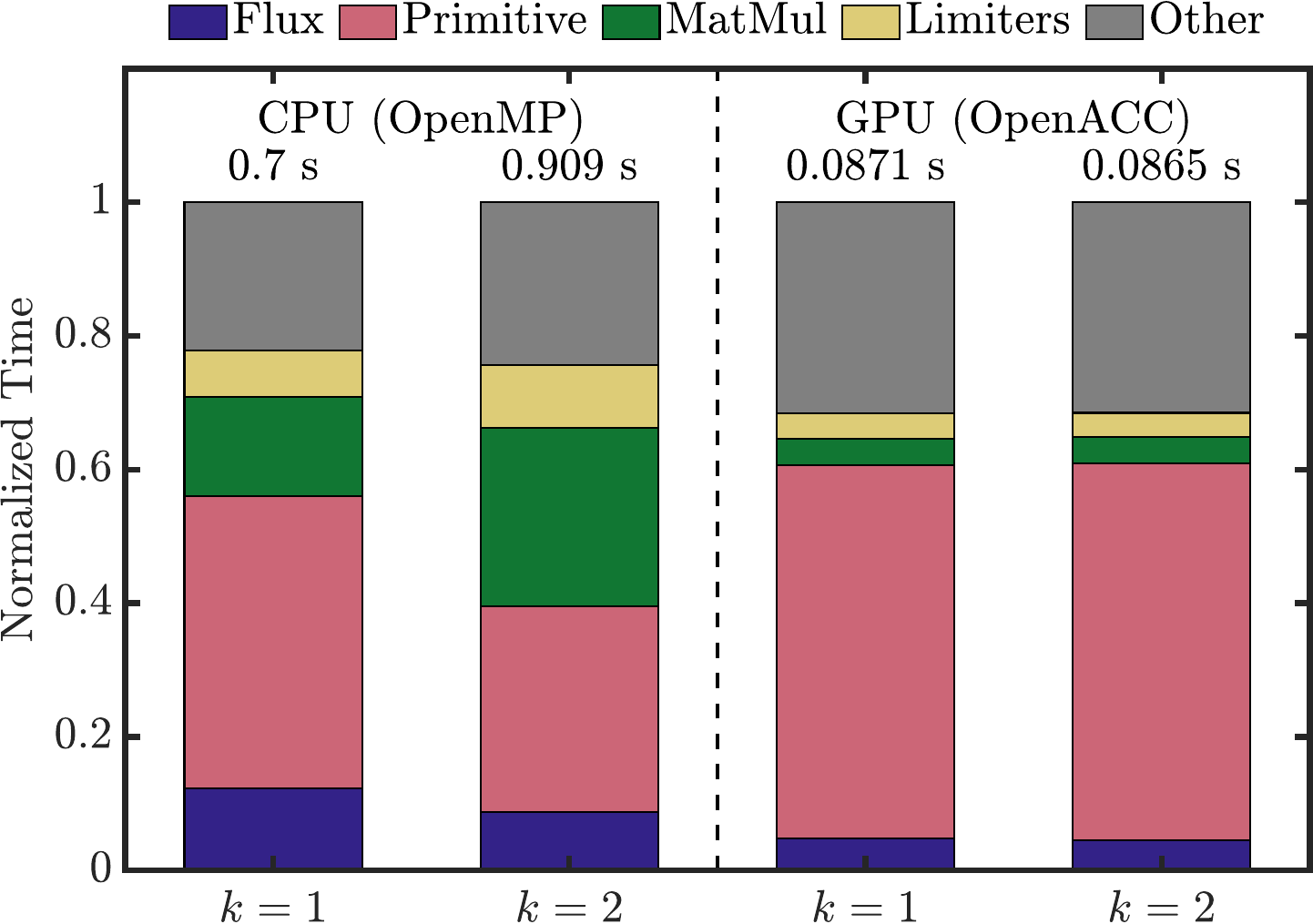}
	\end{tabular}
	\caption{
		Breakdown of normalized wall-time for components of the Streaming Doppler Shift test problem as implemented in \thornado.
		The left panel uses an initial guess of $\bcM^{[0]} = (\cN, \vect{0})^{\intercal}$ in the conserved-to-primitive calculation, and the right panel uses $\bcM^{[0]} = \bcU = (\cN, \bcG)^{\intercal}$.
		Absolute wall-clock times per time step are shown above each bar.
		\texttt{Flux} captures the calculation of fluxes $\vect{\mathcal{F}}^{i}$ and $\vect{\mathcal{F}}^{\varepsilon}$ in Eqs.~\eqref{eq:bilinearFormAdvectionPosition}--\eqref{eq:bilinearFormAdvectionEnergy}.
	    \texttt{MatMul} represents the matrix-matrix multiplications used throughout the explicit operator, e.g., to evaluate polynomials $\vect{\mathcal{U}}_{h}$ in quadrature points on all elements.  
	    \texttt{Primitive} captures the iterative conserved-to-primitive calculation described in Section~\ref{sec:moment_conversion}.
	    \texttt{Limiters} includes the application of the realizability-enforcing limiter described Section~\ref{sec:realizabilityLimiter} and the energy limiter described in Section~\ref{sec:EnergyLimiter}.
	    \texttt{Other} is used to capture all remaining wall-time spent in the explicit step.
    }
	\label{fig:SDS_walltime}
\end{figure}
\section{Summary and Conclusions}
\label{sec:summaryConclusions}

We have proposed and analyzed a realizability-preserving numerical method for evolving a spectral two-moment model for neutral particles interacting with a moving background fluid.  
This number-conservative moment model is based on comoving-frame momentum coordinates, includes special relativistic corrections to $\cO(v)$, and, as a result, contains velocity-dependent terms accounting for spatial advection, Doppler shift, and angular aberration.  
The nonlinear two-moment model solves for comoving-frame angular moments, representing number density and components of the number flux densitiy, and is closed by expressing higher-order moments (rank-two and rank-three tensors) in terms of the evolved moments using the maximum entropy closure (both exact and approximate) due to Minerbo \cite{minerbo_1978}.  
The two-moment model is closely related to that promoted in \cite{lowrie_etal_2001}, predicts wave speeds bounded by the speed of light (Proposition~\ref{prop:waveSpeed}), and is consistent, to $\cO(v)$, with Eulerian-frame energy and momentum conservation (Proposition~\ref{prop:EnergyandMomentumConservation}).  

The numerical method is based on the DG method to discretize phase-space, and IMEX time stepping, where the phase-space advection part is integrated with explicit methods, and the collision term is integrated with implicit methods.  
The discretized spatial and energy derivative terms in the moment equations have been equipped with tailored numerical fluxes, which in the case of exact moment closure (Assumption~\ref{assum:exact_closure}) allow us to derive explicit time-step restrictions that guarantee realizable cell-averaged moments due to these terms, and $\cN>0$ overall.  
Unfortunately, a corresponding time-step restriction was not found for the source terms associated with phase-space advection in the number flux equation to guarantee the second moment realizability condition, in the sense of cell averages, for the evolved moments (i.e., $|\bcG|\le\cN$) in the general multidimensional case.  
However, an analysis in the semi-discrete setting revealed that the moments evolve tangentially to the boundary of the realizable domain when $|\bcG|=\cN$, and we found a sufficient time-step restriction to guarantee realizable cell averages in the one-dimensional, planar geometry case.  
Given a positive cell-averaged number density, a realizability-enforcing limiter is proposed to recover pointwise moment realizability in each element.  
Specific properties of the IMEX scheme (i.e., convex-invariance, as defined in \cite{chu_etal_2019}) extend the applicability of our results beyond the forward-backward Euler sequence analyzed in detail.  

Retention of specific $\cO(v)$ terms in the time derivative of the two-moment system, motivated by the desire to maintain wave speeds bounded by the speed of light and consistency with Eulerian-frame energy and momentum conservation equations, results in increased computational complexity of the numerical scheme in two (related) ways.  
First, since the evolved moments are nonlinear functions of the primitive moments used to close the moment equations, a nonlinear system must be solved to recover primitive moments from evolved moments.  
Second, because the collision operators are formulated in terms of primitive moments, the implicit collision update requires the solution of a similar nonlinear system.  
For both cases, solution methods have been formulated as fixed-point problems, and we have proposed tailored fixed-point operators in Eqs.~\eqref{eq:richardson_fixed_pt} and \eqref{eq:collision_fixed_point1}, for the primitive recovery and implicit collision solve, respectively.  
The fixed-point operators are designed to preserve moment realizability in each iteration (subject to mild conditions on the step size), and we have proven convergence for cases with exact \emph{and} approximate moment closures, subject to the additional constraint $|\vect{v}|\le\sqrt{2}-1$, which is mild when considering the applicability of the $\cO(v)$ model.  
Numerically, we did not observe convergence failures for the primitive recovery problem, even when violating the condition on the velocity, or when combining the algorithm with Anderson acceleration, which our analysis here did not consider.  

The proposed algorithm has been implemented and tested against a series of benchmark problems.  
Using two problems with a constant background velocity --- in the streaming and diffusion regimes, respectively --- we demonstrate the expected rate of error convergence in the $L^{2}$ norm.  
Additional tests with spatially varying (smooth and discontinuous) background velocity fields --- the Streaming Doppler Shift, Transparent Shock, and Transparent Vortex tests --- were used to document the robustness of the proposed algorithm, and qualitative accuracy with respect to special relativistic considerations (e.g., correct Doppler shifts) for sufficiently small background velocities.  
In these tests, the moments evolve close to the boundary of the realizable domain, and the realizability-enforcing limiter is frequently triggered to recover pointwise realizability from (guaranteed) realizable cell averages.  
Without this recovery procedure, the algorithm fails invariably on these challenging problems.  

We have analyzed the simultaneous Eulerian-frame number and energy conservation properties of the proposed method.  
While the DG method provides flexibility in the approximation spaces to capture conservation properties beyond those inherent to the model formulation (i.e., number conservation in the present setting), the approximation of the background velocity by piecewise polynomials from the DG approximation space, which accommodates discontinuities, can result in Eulerian-frame energy conservation errors that exceed the $\cO(v^{2})$ scaling predicted by the continuum model.  
However, we found that the realizability-enforcing limiter is the main contributor to Eulerian-frame energy conservation violations when the background velocity field is smooth and its magnitude is within the range of applicability of an $\cO(v)$ model.  
For this reason, an energy limiter is proposed to recover conservation violations introduced by the realizability-enforcing limiter.  
This limiter trades local number conservation for number \emph{and} energy conservation after integration over the phase-space energy dimension, and has no observed negative impact on solution accuracy, while improving Eulerian-frame energy conservation properties of the method.  
With the energy limiter active, we observe that energy conservation violations scale as $\cO(v^{2})$, in accordance with the continuum model.  
We emphasize that the energy limiter introduces a rescaling of the moments, which does not impact moment realizability.  
However, the proposed strategy to promote Eulerian-frame energy conservation is not feasible without the realizability-preserving property.  

Our goal is to apply the proposed algorithm to model neutrino transport in core-collapse supernova simulations.  
Several extensions are needed to achieve this goal.  
First, the collision term must be extended to include a complete set of neutrino weak interactions, and the model extended to include coupling to dynamical equations for the background fluid.  
Second, because neutrinos are Fermions, for which the Pauli exclusion principle implies an upper bound on the phase-space density and associated bounds on the moments, the analysis should be extended to apply to moment closures based on Fermi-Dirac statistics.  
Third, because special \emph{and} general relativistic effects contribute to the dynamics in nontrivial ways, further development and analysis is required to design realizability-preserving methods for fully relativistic moment models.  
We believe the methodologies developed in this paper can be helpful in these endeavors, and hope to present progress on addressing these challenges in future work.

\appendix

\section{Technical Proofs}
\label{sec:appendix}

\subsection{Various Bounds for the Exact and Approximate Eddington Factors}
\label{sec:polynomial_bounds}

In the following lemma, we list several bounds on functions dependent on the exact or approximate Eddington factors ($\psi$ or $\psi_{\mathsf{a}}$). These bounds are used in the proofs of Lemmas~\ref{lemma:dD_term} and \ref{lemma:dI_term} in \ref{sec:proof_of_dD} and \ref{sec:proof_of_dI}, respectively, as well of the proof of Lemma~\ref{lemma:MtoU} in Section~\ref{sec:approxClosure}.

\begin{lemma}\label{lemma:polynomial_bounds}
	Let $\psi$ be the Eddington factor in the exact Minerbo closure as given in Eq.~\eqref{eq:psiZetaMinerbo} and 
	let 
	\begin{equation}\label{eq:phi_def}
		\phi_1:= 3\psi-1-3\psi^\prime h\quad \text{and}\quad \phi_2:=(3\psi-1)h^{-1}\:.
	\end{equation}
	Then, the following bounds hold when $h\in[0,1]$.
\begin{multicols}{2}
\begin{itemize}
	\item[\textup{(a)}] $-4 \leq \phi_1 \leq 0$,
	\item[\textup{(b)}] $\phi_2^2 - \psi^\prime \phi_2\geq0$, 
	\item[\textup{(c)}] $3(\psi^\prime)^2 - 3 \psi^\prime \phi_2 +\phi_2^2 \geq0$,
	\item[\textup{(d)}] $\partial_h(\phi_2^2 -  \psi^\prime \phi_2 +(\psi^\prime)^2) >0$.
\end{itemize}
\end{multicols}
\noindent
Moreover, Let $\psi_{\mathsf{a}}$ be the approximate Eddington factor defined in Eq.~\eqref{eq:psiApproximate} and 
let 
\begin{equation}
	\phi_{\mathsf{a},1}:= 3\psi_{\mathsf{a}}-1-3\psi_{\mathsf{a}}^\prime h\quad \text{and}\quad \phi_{\mathsf{a},2}:=(3\psi_{\mathsf{a}}-1)h^{-1}\:.
\end{equation}
Then the bounds \textup{(a)--(d)} hold when $(\psi,\phi_1,\phi_2)$ are replaced by $(\psi_{\mathsf{a}},\phi_{\mathsf{a},1},\phi_{\mathsf{a},2})$.
In addition, the following two bounds hold for the approximate Eddington factor when $h\in[0,1)$.
\begin{multicols}{2}
	\begin{itemize}
		\item[\textup{(e)}] $\psi_{\mathsf{a}} - h^2  - \frac{1}{4} (1-\psi_{\mathsf{a}})^2\geq0$, 
		\item[\textup{(f)}] $h^2\leq \psi_{\mathsf{a}} \leq 1$.
	\end{itemize}
\end{multicols}
\end{lemma}
Since both $\psi$ and $\psi_{\mathsf{a}}$ are one-dimensional functions defined between 0 and 1, the proof of the bounds are straightforward but are rather tedious. Instead of giving rigorous proofs for these bounds, we plot the functions of interest in Figure~\ref{fig:psi_bounds}, from which the bounds can be visually verified.
\begin{figure}[H]
	\centering
		\subfloat[]{\includegraphics[width=0.315\textwidth]{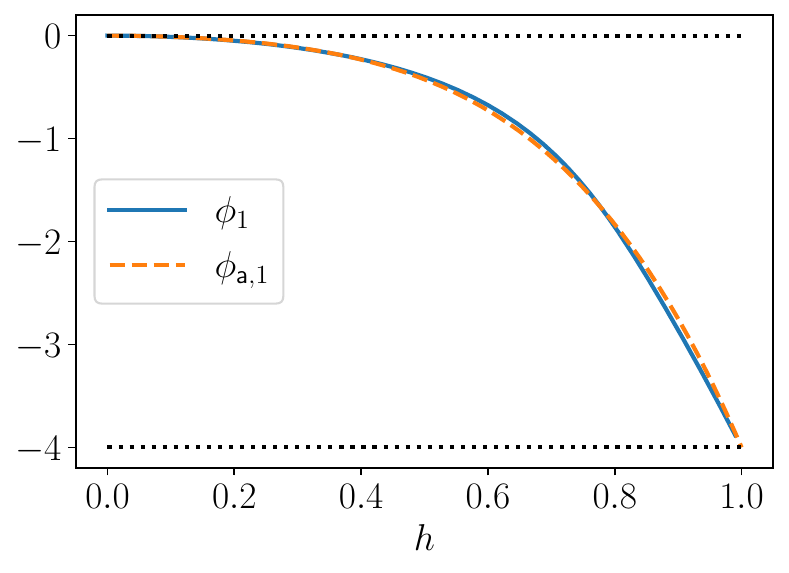}}%
		\subfloat[]{\includegraphics[width=0.32\textwidth]{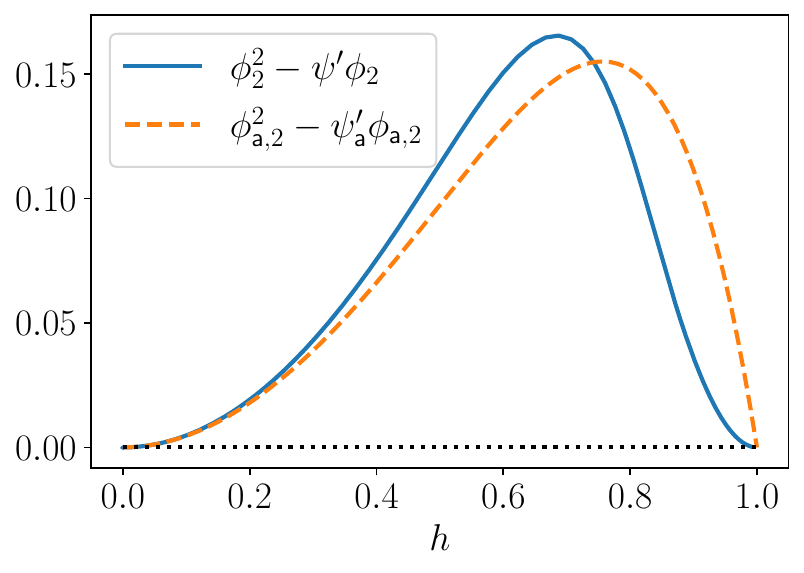}}%
		\subfloat[]{\includegraphics[width=0.302\textwidth]{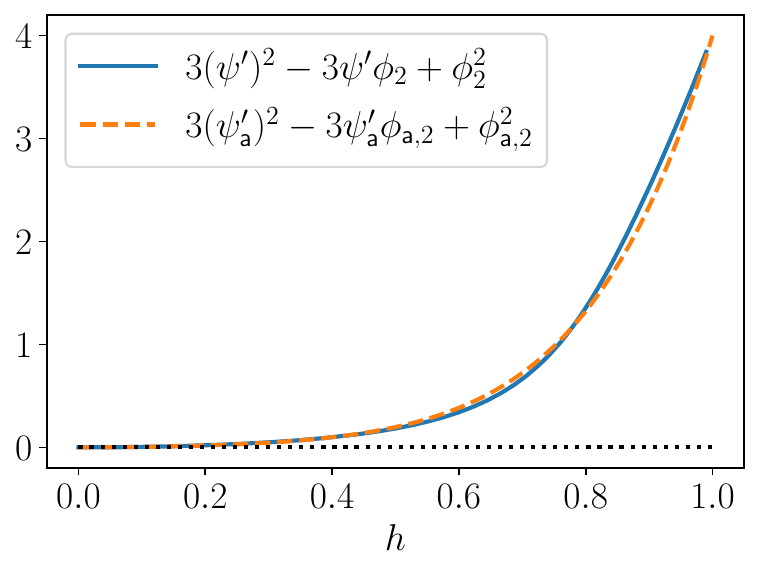}}\\
		\subfloat[]{\includegraphics[width=0.305\textwidth]{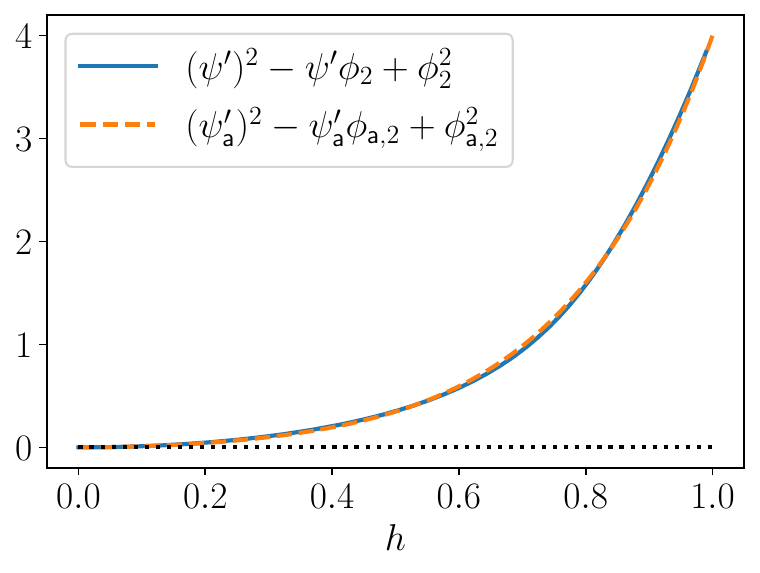}}%
		\subfloat[]{\includegraphics[width=0.32\textwidth]{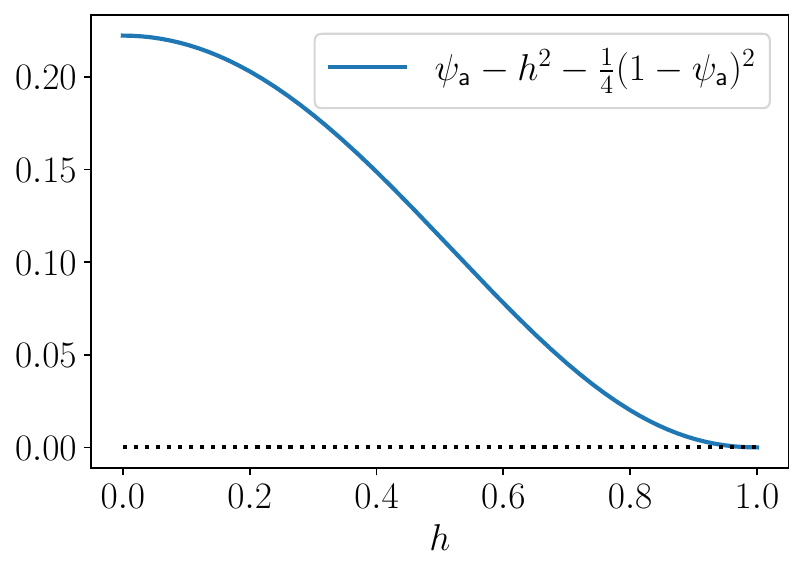}}%
		\subfloat[]{\includegraphics[width=0.31\textwidth]{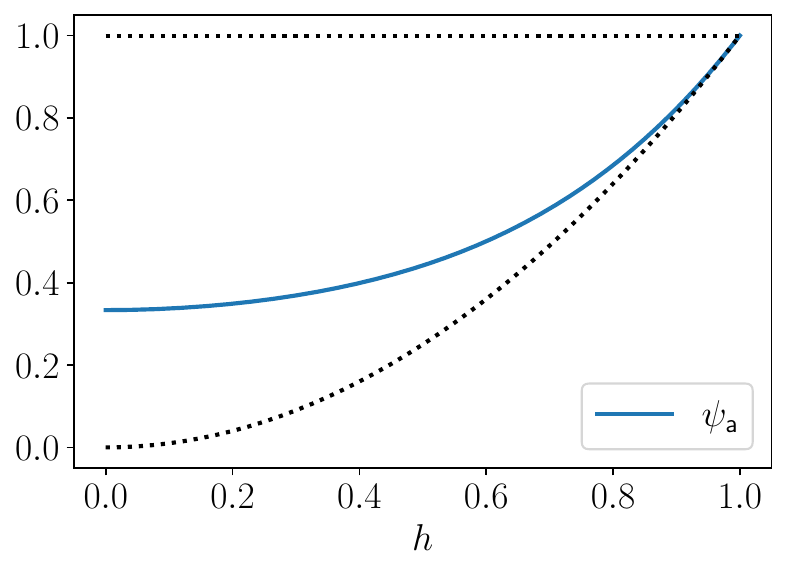} }
\caption{Verification of the bounds \textup{(a)--(f)} in Lemma~\ref{lemma:polynomial_bounds} in the cases when the exact Eddington factor $\psi$ and approximate $\psi_{\mathsf{a}}$ are considered.} 
\label{fig:psi_bounds}
\end{figure}

\subsection{Proof of Lemma~\ref{lemma:dD_term}}
\label{sec:proof_of_dD}
	
\begin{proof}[Proof of Lemma~\ref{lemma:dD_term}.]
Using the definition of the closure terms $\mathsf{k}_{ij}$ in Eq.~\eqref{eq:VariableEddingtonTensor}, we have from chain rule that
	\begin{align}
			v^{i} {\partial_{\mathcal{D}} (\mathsf{k}_{ij} \mathcal{D})}
			&= v^{i} \left( \frac{1}{2} \big[3 \hat{n}_i\hat{n}_j - \delta_{ij}\big] \frac{\p\psi}{\p h} \frac{\p h}{\p \mathcal{D}} \mathcal{D} + \frac{1}{2} \big[(1-\psi) \delta_{ij} + (3\psi -1)\hat{n}_i\hat{n}_j\big] \right)  \nonumber \\
			&= v^{i} \left( - \frac{1}{2} \big[3 \hat{n}_i\hat{n}_j - \delta_{ij}\big] \psi^\prime h + \frac{1}{2} \big[(1-\psi) \delta_{ij} + (3\psi -1)\hat{n}_i\hat{n}_j\big] \right)  \\
			&= \frac{1}{2} (3\psi -1 - 3\psi^\prime h ) \big(v^{i} \hat{n}_i\hat{n}_j - \frac{1}{3} v_{j}  \big) +  \frac{1}{3} v_{j}  
			= \frac{1}{2} \phi_1 \big(v^{i} \hat{n}_i\hat{n}_j - \frac{1}{3} v_{j}  \big) +  \frac{1}{3} v_{j} \:,\nonumber
	\end{align}
where $\phi_1:= (3\psi -1 - 3\psi^\prime h )$ as defined in Eq.~\eqref{eq:phi_def}.
Since $\|\partial_{\cD} (v^{i} \mathsf{k}_{ij} \mathcal{D}) \|^2 = \sum_j \left( v^{i} {\partial_{\mathcal{D}} (\mathsf{k}_{ij} \mathcal{D})}\right)^2$,
summing up the squares leads to
\begin{align}
		\|\partial_{\cD} (v^{i} \mathsf{k}_{ij} \mathcal{D}) \|^2
		&= \frac{1}{4} \phi_1^2 \sum_j \left(v^{i} \hat{n}_i\hat{n}_j - \frac{1}{3} v_{j}  \right)^2 +  \frac{1}{3} \phi_1 v^{j}\left(v^{i} \hat{n}_i\hat{n}_j - \frac{1}{3} v_{j}  \right) + \frac{1}{9} v^{j} v_{j}\nonumber \\
		&= \left(\frac{\phi_1^2}{12} + \frac{\phi_1}{3} \right) (v^{i} \hat{n}_i)^2 + \left(\frac{\phi_1^2}{36} - \frac{\phi_1}{9} + \frac{1}{9}\right) v^{i} v_{i}\:.
\end{align}
It follows from Lemma~\ref{lemma:polynomial_bounds} (a) that $\phi_1(h)\in[-4,0]$ for $h\in[0,1]$.
Therefore, $\left(\frac{\phi_1^2}{12} + \frac{\phi_1}{3} \right) = \frac{\phi_1}{12}\left(\phi_1 + 4 \right) \geq0$ and we have
\begin{equation}
	\|\partial_{\cD} (v^{i} \mathsf{k}_{ij} \mathcal{D}) \|^2 					\leq \left(\big(\frac{\phi_1^2}{12} + \frac{\phi_1}{3} \big) + \big(\frac{\phi_1^2}{36} - \frac{\phi_1}{9} + \frac{1}{9}\big)\right) v^{i} v_{i}
	= \frac{1}{9}\left( \phi_1 + 1\right)^2 v^{i} v_{i}\:.
\end{equation}
Since $\phi_1\in[-4,0]$, $\frac{1}{9}\left( \phi_1 + 1\right)^2\leq1$, which concludes the proof.
\end{proof}

\subsection{Proof of Lemma~\ref{lemma:dI_term}}
\label{sec:proof_of_dI}

\begin{proof}[Proof of Lemma~\ref{lemma:dI_term}.]
	Using the definition of the closure terms $\mathsf{k}_{ij}$ in Eq.~\eqref{eq:VariableEddingtonTensor}, we have from chain rule that
	\begin{equation}
		v^{i} {\partial_{\cI_{k}} (\mathsf{k}_{ij} \cD)}
		= \f{1}{2}\,{\psi^\prime}\,\Big[\,3\,s\,\hat{n}_{j}-v_{j}\,\Big]\,\hat{n}_{k}
		+\f{(3\,\psi-1)}{2h}\,\Big[\,v_{k}\,\hat{n}_{j}+s\delta_{jk}\,-2\,s\,\hat{n}_{j}\,\hat{n}_{k}\,\Big]\:.
	\end{equation}
	To show $\|\nabla_{\bcI}(v^{i}\mathsf{k}_{ij} \cD)\|\leq2v$, we prove in the following that
	\begin{equation}
		\|\nabla_{\bcI}(v^{i}\mathsf{k}_{ij} \cD)\vect{y}\|\leq2v y\:,\quad \forall\vect{y}=\big(y^1, y^2, y^3\big)^{\intercal},
	\end{equation}
	where $y:=\sqrt{y^i y_i}\,$.
	Let $\phi_2:=(3\psi-1)h^{-1}$ as defined in Eq.~\eqref{eq:phi_def}. Then,
	\begin{equation}
		v^{i} {\partial_{\cI_{k}} (\mathsf{k}_{ij} \cD)} y^{k}= \f{1}{2}\,{\psi^\prime}\,\Big[\,3\,s\,\hat{n}_{j}-v_{j}\,\Big]\,(y^{k}\hat{n}_{k})
		+\f{1}{2}\phi_2\,\Big[\,\hat{n}_{j}(y^{k} v_{k})\,+s y_{j}\,-2\,s\,\hat{n}_{j}\,(y^{k}\hat{n}_{k})\,\Big]\:.
	\end{equation}
Summing up the squares leads to
\begin{equation}
	\begin{alignedat}{2}
		\|\nabla_{\bcI}(v^{i}\mathsf{k}_{ij} \cD)\vect{y}\|^2 = &\sum_j \left(	v^{i} {\partial_{\cI_{k}} (\mathsf{k}_{ij} \cD)} y^{k}\right)^2
		= \f{1}{4} \phi_2^2 s^2 y^2 + \f{1}{4}\phi_2^2 \vy^2 +\f{1}{4}(\psi^\prime)^2 v^2 \ny^2 \\
		&\qquad+  \f{1}{4}\Big[3(\psi^\prime)^2 - 2 \psi^\prime \phi_2 \Big]s^2 \ny^2
		+ \f{1}{2}\Big[\psi^\prime \phi_2 -  \phi_2^2 \Big]s \vy \ny\:.
	\end{alignedat}
\end{equation}
	Since $\phi_2^2 - \psi^\prime \phi_2\geq0$ (Lemma~\ref{lemma:polynomial_bounds}~(b)), we apply the inequality $-2ab\leq a^2 + b^2$ and obtain
	\begin{equation}
		\f{1}{2}\Big[\psi^\prime \phi_2 -  \phi_2^2 \Big]s \vy \ny\leq 	\f{1}{4}\Big[\phi_2^2 - \psi^\prime \phi_2  \Big] \vy^2  + 	\f{1}{4}\Big[\phi_2^2 - \psi^\prime \phi_2  \Big]s^2 \ny^2\:.
	\end{equation}
	Therefore,
	\begin{equation}
			\begin{alignedat}{2}
		\|\nabla_{\bcI}(v^{i}\mathsf{k}_{ij} \cD)\vect{y}\|^2
		\leq& \,\f{1}{4} \phi_2^2 s^2 y^2 + \f{1}{4}\Big[2\phi_2^2 - \psi^\prime \phi_2 \Big]\vy^2 +\f{1}{4}(\psi^\prime)^2 v^2 \ny^2 \\
		&+  \f{1}{4}\Big[3(\psi^\prime)^2 - 3 \psi^\prime \phi_2 +\phi_2^2\Big]s^2 \ny^2\:.
			\end{alignedat}
	\end{equation}
	Further, since $\phi_2^2 \geq0$, $(\psi^\prime)^2 \geq0$, $2\phi_2^2 - \psi^\prime \phi_2  \geq 0$ (Lemma~\ref{lemma:polynomial_bounds}~(b)), and $3(\psi^\prime)^2 - 3 \psi^\prime \phi_2 +\phi_2^2 \geq0$ (Lemma~\ref{lemma:polynomial_bounds}~(c)), we can take the upper bounds $s^2\leq v^2$, $\vy^2\leq (vy)^2$, and $\ny^2\leq y^2$ to obtain
	\begin{equation}
		\|\nabla_{\bcI}(v^{i}\mathsf{k}_{ij} \cD)\vect{y}\|^2
		\leq \, \Big[\phi_2^2 -  \psi^\prime \phi_2 +(\psi^\prime)^2  \Big] {(vy)^2}\:.
	\end{equation}
	It follows from Lemma~\ref{lemma:polynomial_bounds}~(d) that $\partial_h(\phi_2^2 -  \psi^\prime \phi_2 +(\psi^\prime)^2) >0$, which implies $\max_{h\in[0,1]}\big[\phi_2^2 -  \psi^\prime \phi_2 +(\psi^\prime)^2\big] = \phi_2^2(1) -  \psi^\prime(1) \phi_2(1) +(\psi^\prime(1))^2  = 4$.
	Thus,
	\begin{equation}
		\|\nabla_{\bcI}(v^{i}\mathsf{k}_{ij} \cD)\vect{y}\|^2
		\leq \, 4 {(vy)^2}\:,\quad \forall\vect{y}=\big(y^1, y^2, y^3\big)^{\intercal},
	\end{equation}
	which proves the claim.
\end{proof}

\bibliography{references}

\end{document}